%% file: cotti_varchenko_markov_arXiv_v4.tex
\documentclass[12pt,reqno]{amsart}
\usepackage{amssymb,amscd,amsthm}

\usepackage{amsmath,amssymb,graphicx,mathrsfs} % my new
\usepackage[colorlinks=true,allcolors = blue]{hyperref} % my new

\usepackage{calc}

\let\frak\mathfrak
\let\Bbb\mathbb

\def\>{\relax\ifmmode\mskip.666667\thinmuskip\relax\else\kern.111111em\fi}
\def\<{\relax\ifmmode\mskip-.333333\thinmuskip\relax\else\kern-.0555556em\fi}
\def\vsk#1>{\vskip#1\baselineskip}
\def\vv#1>{\vadjust{\vsk#1>}\ignorespaces}
\def\vvn#1>{\vadjust{\nobreak\vsk#1>\nobreak}\ignorespaces}

  \let\ssize\scriptstyle
\let\sssize\scriptscriptstyle

\let\Medskip\medskip
\def\medskip{\par\Medskip}
\let\Bigskip\bigskip
\def\bigskip{\par\Bigskip}

\let\Maketitle\maketitle
\def\maketitle{\Maketitle\thispagestyle{empty}\let\maketitle\empty}

\newtheorem{thm}{Theorem}[section]
\newtheorem{cor}[thm]{Corollary}
\newtheorem{lem}[thm]{Lemma}
\newtheorem{prop}[thm]{Proposition}

\numberwithin{equation}{section}

\theoremstyle{definition}
\newtheorem{defn}[thm]{Definition}
\newtheorem{rem}[thm]{Remark}
\newtheorem*{example}{Example}

\let\mc\mathcal
\let\nc\newcommand

\let\al\alpha
\let\bt\beta

\let\Gm\Gamma

\let\la\lambda
\let\La\Lambda
\let\pho\phi
\let\phi\varphi
\let\si\sigma
\let\Si\Sigma

\let\om\omega
\let\Om\Omega

\let\der\partial

\let\ge\geqslant
\let\geq\geqslant
\let\le\leqslant
\let\leq\leqslant

\let\on\operatorname
\let\bi\bibitem
\let\bs\boldsymbol

\def\C{{\mathbb C}}
\def\Z{{\mathbb Z}}
\def\Q{{\mathbb Q}}

\def\B{{\mc B}}

\def\F{{\mc F}}

\def\+#1{^{\{#1\}}}

\def\id{\on{id}}

\def\sln{\mathfrak{sl}_N}

\def\beq{\begin{equation}}
\def\eeq{\end{equation}}
\def\be{\begin{equation*}}
\def\ee{\end{equation*}}

\nc{\bea}{\begin{eqnarray*}}
\nc{\eea}{\end{eqnarray*}}
\nc{\bean}{\begin{eqnarray}}
\nc{\eean}{\end{eqnarray}}
\nc{\bal}{\begin{align*}}
\nc{\eal}{\end{align*}}
\nc{\baln}{\begin{align}}
\nc{\ealn}{\end{align}}
\let\ga\gamma
\let\Ga\Gamma

\nc{\Il}{{\mc I_{\bs\la}}}
\nc{\bla}{{\bs\la}}
\nc{\Fla}{\F_\bla}
\nc{\tfl}{{T^*\Fla}}
\nc{\GL}{{GL_n(\C)}}
\nc{\GLC}{{GL_n(\C)\times\C^*}}

\let\sd s %% \def\sd{\dot s}

\def\ddk_#1{\kk_{#1}\<\>\frac\der{\der\<\>\kk_{#1}}}

\def\bul{\mathbin{\raise.2ex\hbox{$\sssize\bullet$}}}
\def\intt{\mathchoice
{\mathop{\raise.2ex\rlap{$\,\,\ssize\backslash$}{\intop}}\nolimits}
{\mathop{\raise.3ex\rlap{$\,\sssize\backslash$}{\intop}}\nolimits}
{\mathop{\raise.1ex\rlap{$\sssize\>\backslash$}{\intop}}\nolimits}
{\mathop{\rlap{$\sssize\<\>\backslash$}{\intop}}\nolimits}}

\let\kk q %% Q
\let\cc c

\let\Ko K

\def\GZ/{Gelfand-Zetlin}
\def\KZ/{{\slshape KZ\/}}
\def\qKZ/{{\slshape qKZ\/}}
\def\XXX/{{\slshape XXX\/}}

\nc{\slnl}{{\sln (\lambda)}}
\nc{\PCN}{{   (\C[x])^N   }}
\nc{\di}{\on{Diag}}
\nc{\dio}{\on{Diag}_0}
\nc{\Mm}{{\mc M}}
\nc{\Nn}{{\mc N}}

\nc{\A}{{\mc C}}

\nc{\PCr}{{  P  (\C[x])^n   }}

\nc{\Pk}{{(\bs{P}^1)^k}}

\nc{\N}{{\Bbb N}}

\def\D{{\mc D}}

\nc{\Ll}{{\mc L}}

\nc{\ord}{{\on{ord}\,}}

\nc{\Sing}{{\on{Sing}\,}}
\nc{\sing}{{\on{Sing}\,}}

\nc{\Hess}{{\on{Hess}}}

\nc{\R}{{\Bbb R}}

\let\on\operatorname
\nc{\Kk}{{\bs K}}
\nc{\Ap}{{A_\Phi(z)}}
\nc{\ap}{{A_\Phi(z)}}

\nc{\sv}{{\sing V}}
\nc{\cd}{{\C^n-\Delta}}
\nc{\UT}{{U^0}}   %(\tilde \A)}}
\nc{\ep}{\epsilon}

\usepackage[all,cmtip]{xy}
\usepackage{mathdots}
\usepackage{float}
\newlanguage\fakelanguage
\newcommand\cyr{\fontencoding{OT2}\fontfamily{wncyr}\selectfont
   \language\fakelanguage}
\DeclareTextFontCommand{\textcyr}{\cyr}

\numberwithin{equation}{section}

\DeclareMathOperator{\HOM}{\mathscr{H}\text{\kern -3pt {\calligra\large om}}\,}

\makeatletter
\newsavebox{\@brx}
\newcommand{\llangle}[1][]{\savebox{\@brx}{\(\m@th{#1\langle}\)}%
  \mathopen{\copy\@brx\kern-0.5\wd\@brx\usebox{\@brx}}}
\newcommand{\rrangle}[1][]{\savebox{\@brx}{\(\m@th{#1\rangle}\)}%
  \mathclose{\copy\@brx\kern-0.5\wd\@brx\usebox{\@brx}}}
\makeatother

\newcommand{\bsh}{\begin{shaded}}
\newcommand{\esh}{\end{shaded}}

\newcommand{\Mcl}{\mc M^c}
\newcommand{\Zsym}{\Z[\bm z^{\pm 1}]^{\frak S_3}}
\newcommand{\smg}{\Gamma_{\!M}}
\newcommand{\mg}{\Gamma_{\!M}^c}
\newcommand{\svg}{\Gamma_{\!V}}
\newcommand{\vg}{\Gamma_{\!V}^c}
\newcommand{\Gc}[1]{G^c_{#1}}
\newcommand{\Gs}[1]{G_{#1}}
\newcommand{\evo}{{\rm ev}_{\bm s_o}}
\newcommand{\Evo}{{\rm Ev}_{\bm s_o}}

\newcommand{\GH}{G_{\rm Hor}}
\let\bm\boldsymbol
\newcommand{\puqed}{\pushQED{\qed}}
\newcommand{\poqed}{\popQED}

\newcommand{\autG}[1]{G_{#1}^{\rm aut}}
\newcommand{\amg}{\Gamma_M^{\rm aut}}

\newcommand{\gmax}{\Gamma^{\rm max}}

\newcommand{\Mclp}{\mc M_+^c}

\newcommand{\Mclmone}{\mc M_{12}^c}
\newcommand{\Mclmtwo}{\mc M_{13}^c}
\newcommand{\Mclmthr}{\mc M_{23}^c}

\nc{\GI}{{\Ga_MI}}

\makeatletter
\let\saveqed\qed
\renewcommand\qed{%
   \ifmmode\displaymath@qed
   \else\saveqed
   \fi}
\newcommand{\mqed}{\makeatletter\displaymath@qed}

\makeatletter
\newcommand{\pushright}[1]{\ifmeasuring@#1\else\omit\hfill$\displaystyle#1$\fi\ignorespaces}
\newcommand{\pushleft}[1]{\ifmeasuring@#1\else\omit$\displaystyle#1$\hfill\fi\ignorespaces}
\makeatother

\newcommand{\pb}{^\star}

\nc{\Zs}{{\mathbb Z[s_1,s_2,s_3^{\pm 1}]}}

\begin{document}
\title{The $*$-Markov equation for Laurent polynomials}
\author[Giordano Cotti and Alexander Varchenko]{Giordano Cotti$\>^\circ$ and Alexander Varchenko$\>^\star$}
\maketitle
\begin{center}
%\textit{ $^\circ\>$School of Mathematics, University of Birmingham\\ Ring Rd N, Birmingham B15 2TS, UK\/}
\textit{ $^\circ\>$Faculdade de Ci\^encias da Universidade de Lisboa - Grupo de F\'isica Matem\'atica \\
Campo Grande Edif\'icio C6, 1749-016 Lisboa, Portugal\/}

\vskip4pt
\textit{ $^{\star}\>$Department of Mathematics, University
of North Carolina at Chapel Hill\\ Chapel Hill, NC 27599-3250, USA\/}

\vskip4pt
\textit{ $^{\star}\>$Faculty of Mathematics and Mechanics, Lomonosov Moscow State
University\\ Leninskiye Gory 1, 119991 Moscow GSP-1, Russia\/}

\end{center}

{\let\thefootnote\relax
\footnotetext{\vskip5pt 
\noindent
$^\circ\>$\textit{ E-mail}:  gcotti@fc.ul.pt, gcotti@sissa.it %G.Cotti.1@bham.ac.uk, gcotti@sissa.it
\\
$^\star\>$\textit{ E-mail}:  anv@email.unc.edu }
}
\vskip1cm
\begin{center}
\it On the Occasion of the 70th Birthday of Sabir Gusein-Zade
\end{center}
\begin{abstract}

We consider the $*$-Markov equation for the symmetric Laurent polynomials in three variables with integer coefficients,
 which appears as
  an equivariant analog of  the classical Markov equation for integers.
We study how the properties of the Markov equation and its solutions are reflected in the properties 
 of the $*$-Markov equation and its solutions.

\end{abstract}

\tableofcontents

\section{Introduction}

\subsection{Markov equation}
\subsubsection{}  The {\it Markov equation} is the Diophantine  equation
\beq
\label{Me}
a^2+b^2+c^2-abc=0\quad a,b,c\in \Z,
\eeq
with initial solution  $(3,3,3)$. If a triple $(a,b,c)$ is a solution, then a permutation of the triple
is a solution. One may also change the sign of two of the three coordinates of a solution.
The braid group $\mc B_3$ acts on the set of solutions,
\bean
\label{Br a}
\tau_1: (a,b,c) \mapsto (-a,c,b-ac),
\\
\notag
\tau_2: (a,b,c) \mapsto (b, a-bc, -c).
\eean
The classical Markov theorem says that all nonzero
solutions of the Markov equation can be obtained from the initial solution
$(3,3,3)$ by these operations, see \cite{Mar1,Mar2}. This group of symmetries of the equation is called the Markov group.
A solution with positive coordinates is called a Markov triple, the positive coordinates are called the
 Markov numbers.

 The  Markov equation is traditionally studied in the form
\beq
\label{clMe0i}
a^2+b^2+c^2-3abc=0,\quad a,b,c \in\Z.
\eeq
Equations  \eqref{Me} and \eqref{clMe0i}  are equivalent.
 A triple $(a,b,c)\in\Z^3$ is a solution of \eqref{clMe0i} if and only if $(3a,3b,3c)$ is a solution of \eqref{Me}. 

\vsk.2>
The equation was introduced by A.A.\,Markov in \cite{Mar1,Mar2}
 in the analysis of minimal values of indefinite binary quadratic forms
 and  was studied in hundreds of papers,  see for example 
 the book \cite{Aig}
 and references therein. 

\subsection{Motivation from exceptional collections and Stokes matrices}

\subsubsection{}
Our motivation came from the works by 
A.\,Rudakov \cite{Rud} on full exceptional collections in derived categories and by
B.\,Dubrovin \cite{Du96,Du98,Du99} on Frobenius manifolds and isomonodromic deformations.

In 1989 A.\,Rudakov studied the full exceptional collections in the derived category $\D^b(\mathbb P^2)$ of the projective plane $\mathbb P^2$.
 These are triples 
$(E_1,E_2,E_3)$ of objects in $\D^b(\mathbb P^2)$ generating $\D^b(\mathbb P^2)$ and such that 
the matrix of Euler characteristics  $(\chi(E_i^*\otimes E_j))$ has the form
$\left(\begin{array}{cccccc}
 1 & a & b
 \\
 0 & 1 & c
 \\
 0&0&1
 \end{array}\right)$. Rudakov observed that the triple $(a,b,c)$ is a solution of the Markov equation. 
 The braid group $\mc B_3$ naturally acts on the set of
 full exceptional collections
  and the induced action on the set of   matrices of Euler characteristics
   coincides  with the action of the braid group on the set of solutions of the Markov equation.

In the 90's Dubrovin considered the isomonodromic deformations of the quantum differential equation of the projective plane
$\Bbb P^2$, see \cite{Du99}. This is a system of three first order linear 
ordinary differential equations with two singular points: one regular point at
the origin and one irregular at the infinity. Dubrovin observed that the Stokes matrix $S$ of a Stokes basis of the space of
solutions at the infinity is of the form
$S=\left(\begin{array}{cccccc}
 1 & a & b
 \\
 0 & 1 & c
 \\
 0&0&1
 \end{array}
\right)$, where $(a,b,c)$ is a solution of the Markov equation. The braid group $\mc B_3$ 
naturally acts on the set of Stokes bases,  and the induced action on the set of 
Stokes matrices coincides  with the action of the braid group on the set of solutions of the Markov equation. 

These two observations allowed to Dubrovin to conclude that the Stokes bases of the isomonodromic deformations
of 
the quantum differential equation of $\mathbb P^2$ correspond to
 the full exceptional collections in the derived category
$\D^b(\mathbb P^2)$ and more generally to conjecture that the derived
 category of an algebraic variety is responsible for the monodromy data of its quantum differential equation, 
 see \cite{Du98, CDG}.

\subsubsection{}

Recently in \cite{TV} V.\,Tarasov and the second author considered the
equivariant quantum differential equation for $\Bbb P^2$
with respect to the torus $T=(\C^\times)^3$ action on $\Bbb P^2$. 
That equivariant  quantum differential equation is a system of three first order linear
ordinary differential equations depending on three equivariant parameters $\bs z=(z_1,z_2,z_3)$. The system has 
 two singular points: one regular point at the origin and one irregular at the infinity.
It turns out that the Stokes matrix $S$ of a Stokes basis of the space of
solutions at the infinity is of the form
$S(\bs z)=\left(\begin{array}{cccccc}
 1 & a & b
 \\
 0 & 1 & c
 \\
 0&0&1
 \end{array}
\right)$, where $a,b,c$ are symmetric Laurent polynomials in the equivariant
parameters $\bs z$ with integer coefficients.  In \cite{CV} we observed that the Stokes bases correspond to $T$-full
exceptional collections in the equivariant derived category
$\D^b_T(\Bbb P^2)$.  If  $(E_1,E_2,E_3)$ is a $T$-full  exceptional collection, then the
equivariant  Euler characteristic
$\chi_T(E_i^*\otimes E_j)$ is an element of the representation ring of the torus, that  is, a Laurent polynomial
in the equivariant parameters with integer coefficients.
It turns out that if a $T$-full  exceptional collection $(E_1,E_2,E_3)$ corresponds
 to a Stokes basis, then the corresponding
Stokes matrix equals  the matrix  $(\chi_T(E_i^*\otimes E_j))$ of equivariant Euler characteristics.
Moreover the three symmetric Laurent polynomials $(a,b,c)$,
 appearing in this construction, satisfy the equation
\bean
\label{mme}
aa^*+bb^*+cc^*-ab^*c = 3-\frac{z_1^3+z_2^3+z_3^3}{z_1z_2z_3},
\eean
where $f^*(z_1,z_2,z_3):= f(1/z_1,1/z_2,1/z_3)$ for any Laurent polynomial,
see \cite[Formula (3.20)]{CV}.  
If $z_1=z_2=z_3=1$, then the right-hand side of \eqref{mme} equals zero,
the equivariant Euler characteristics $\chi_T(E_i^*\otimes E_j)$ become the non-equivariant Euler characteristics 
$\chi(E_i^*\otimes E_j)$, 
and the triple of symmetric Laurent polynomials $(a,b,c)$ evaluated at  $z_1=z_2=z_3=1$ becomes a solution of the Markov equation \eqref{Me}.

\vsk.2>
We call equation \eqref{mme} for symmetric Laurent polynomials with integer coefficients
the {\it $*$-Markov equation}. 

\vsk.2>

The transition from the Markov equation to the $*$-Markov equation provides us with a deformation
of the Markov numbers by replacing Markov numbers with symmetric Laurent polynomials, which
recover the numbers  after the  evaluation at 
$z_1=z_2=z_3=1$.

\vsk.2>

The goal of this paper is to
observe how the properties of the Markov equation and its solutions 
are reflected in the properties of the $*$-Markov equation  and its solutions.

\subsubsection{}
There are  interesting instances of the transition from the Diophantine Markov equation \eqref{Me} to an equation
of the form
\bea
a(t)^2+b(t)^2+c(t)^2 - a(t)b(t)c(t) = R(t),
\eea
where $a(t),b(t),c(t)$ are unknown functions in some variables $t$ and $R(t)$ is a given function.
Such  deformations among other subjects are related to hyperbolic geometry and cluster algebras, 
see the fundamental papers \cite{CP07, FG07}. 

\vsk.2>
The difference between deformations of this type and the $*$-Markov equation  is that equation \eqref{mme}
includes the  $*$-operation dictated by the equivariant $K$-theoretic setting.
 It is an interesting problem
to find relations of the  $*$-Markov equation to hyperbolic geometry and cluster algebras.

\subsection{$*$-Markov equation and $*$-Markov group}

\subsubsection{}

It is convenient to use the elementary symmetric functions $(s_1,s_2, s_3)$,
\bea
s_1=z_1+z_2+z_3, \quad s_2=z_1z_2+z_1z_3+z_2z_3,\quad s_3=z_1z_2z_3,
\eea
change variables $(a,b,c)$ to $(a,b^*,c)$,
 and reformulate equation \eqref{mme} in a more symmetric form
\bean
\label{mm}
aa^*+bb^*+cc^*-abc = \frac{3s_1s_2-s_1^3}{s_3}.
\eean
The problem is to find  Laurent polynomials $a,b,c \in \Z[s_1,s_2,s_3^{\pm1}]$
satisfying equation
\eqref{mm}.  The equation has the initial solution
\beq\label{ESi}
I=
\left(z_1+z_2+z_3, \frac {z_1+z_2+z_3}{z_1z_2z_3} ,  z_1+z_2+z_3\right)
=\left(s_1,\frac{s_1}{s_3},s_1\right),
\eeq
whose evaluation at  $s_1=s_2=3, s_3=1$ is the initial solution (3,3,3) of the Markov equation.

\vsk.2>
 From now on we call equation \eqref{mm} the $*$-Markov equation.

\subsubsection{}
The group $\Ga_M$ of symmetries of the $*$-Markov equation is
called the $*$-Markov group. It  consists of permutations 
of variables, changes of sign of two of the three variables, the braid group $\mc B_3$ transformations
\bean
\label{tau1i}
&&
\tau_1\colon (a,b,c)\mapsto   (-a^*,\quad c^*,\quad b^*-ac),
\\
&&
\notag
\tau_2\colon (a,b,c)\mapsto (b^*,\quad a^*-bc,\quad -c^*),
\eean
and the new transformations
\bea
\mu_{i,j}\colon 
(a,b,c) \mapsto              (s_3^i a, 
                         s_3^{-i-j} b,
                                s_3^j c),\quad i,j\in\mathbb Z.
\eea
We have an obvious epimorphism of the $*$-Markov group onto the Markov group.
This fact and the Markov theorem imply that for any Markov triple of numbers there 
exists a triple of Laurent polynomials,
solving the $*$-Markov equation, obtained from the initial solution  $I$ by transformations
of the $*$-Markov group, whose evaluation at  $s_1=s_2=3, s_3=1$  gives the Markov triple.

\vsk.2>
It is an open question if any solution of the $*$-Markov equation can be obtained from
the initial solution $I$ by a transformation of the $*$-Markov group.
In analogy with the  Markov equation
we may expect that all solutions lie in $\Ga_MI$.

\subsection{Solutions in the orbit of the initial solution}

\subsubsection{}
As the  first topic of this paper we study $\Ga_MI$, the set of solutions of the $*$-Markov equation obtained from
the initial solution $I$ by  transformations of the $*$-Markov group. 

\vsk.2>
In the interpretation of solutions of the $*$-Markov equation as
matrices $(\chi_T(E_i^*\otimes E_j))$
of equivariant Euler characteristics   for $T$-full exceptional collections,
the set $\Ga_MI$ corresponds to the set of matrices $(\chi_T(E_i^*\otimes E_j))$ for the
 $T$-full exceptional collections in $\D_T^b(\Bbb P^2)$ lying in the braid group orbit of the 
 so-called   Beilinson  $T$-full exceptional collection, see \cite{CV, Bei}.
 
 \vsk.2>
Several first elements of $\Ga_MI$ different from $I$
are 
 \bean
 \label{Xe1} 
&&
(s_1^*,s_1^2-s_2, s_1^*) ,
\\
\label{Xe11}
&&
(s_2^*, s_1^2s_2 - s_1s_3-s_2^2 , (s_1^2-s_2)^*),
\\
 \label{Xe2} 
&&
(s_1^*, s_1^3s_2-2  s_1^2s_3- s_1s_2^2+s_2 s_3       , (s_1^2s_2 - s_1s_3-s_2^2)^*),
\\
 \label{Xe3} 
&&
((s_1^2s_2 - s_1s_3-s_2^2)^*, s_1^2 s_2^3-s_1^3  s_2s_3-s_2s_3^2 +s_1^2 s_3^2-s_2^4  ,  (s_2^2-s_1s_3)^*).
\eean
Evaluated at $s_1=s_2=3, s_3=1$ they represent the Markov triples
(3,6,3), (3,15,6), (3,39,15), (15,87,6), respectively.

\vsk.2>
Random application of generators of the $*$-Markov group
 to the initial solution $I$
 will produce the Laurent polynomial solutions of the $*$-Markov equation, but they will not be
  polynomial. We  make them close to being polynomial as follows.
  
  \vsk.2>

We look for solutions of the $*$-Markov equation in the form $(f_1^*, f_2,f_3^*)$.
We say that a solution $(f_1^*, f_2,f_3^*)$ 
is  a {\it reduced polynomial solution} if each of $f_1, f_2, f_3$ is
a nonconstant polynomial in $s_1,s_2,s_3$  not divisible by $s_3$.

\vsk.2>

For example,  all triples in \eqref{Xe1}-\eqref{Xe3} are reduced polynomial solutions.

\begin{thm}
\label{thm 4.14i}

Let $(a,b,c)$ be a Markov triple,  $0<a < b$, $0< c < b$, $6\leq b$.
Then there exists a unique reduced polynomial solution  $(f_1^*,f_2 , f_3^*)\in \GI$
representing $(a,b,c)$.
\end{thm}

See Theorem \ref{thm 4.14}.

\subsubsection{}
Let $f(s_1,s_2,s_3)$ be a polynomial. We consider
 two degrees of $f$:\ the 
{\it homogeneous degree} $d:={\rm deg} f $ with respect to weights $(1,1,1)$ 
and the {\it quasi-homogeneous degree} $q:={\rm Deg} f$  with respect to weights $(1,2,3)$.
For example,
${\rm deg} (s_1^{a_1} s_2^{a_2}s_3^{a_3}) = a_1+a_2+a_3$,
${\rm Deg }(s_1^{a_1} s_2^{a_2}s_3^{a_3}) = a_1+2a_2+3a_3$.

\vsk.2>

Let $f(s_1,s_2,s_3)$ be a polynomial of homogeneous degree $d$ not divisible by $s_3$, then
\bean
\label{fgi}
g(s_1,s_2,s_3):=s_3^{d} \,f\Big(\frac{s_2}{s_3}, \frac{s_1}{s_3}, \frac{1}{s_3}\Big)
\eean
is a polynomial of homogeneous degree $d$ not divisible by $s_3$. If additionally
$f(s_1,s_2,s_3)$ is a quasi-homogeneous polynomial of quasi-homogeneous degree $q$, then
$g(s_1,s_2,s_3)$ is a quasi-homogeneous polynomial of quasi-homogeneous degree $3d-q$.

The polynomial $g$ is  denoted by $\mu (f)$.

\subsubsection{}
\label{1.4.3}

The assignment to a quasi-homogeneous polynomial $f$ its bi-degree vector $(d,q)$ could be seen as a
``de-quatization'' in the following sense. Let 
\bea
s_1= c_1e^{\alpha + \beta}, 
\qquad 
s_2=c_2e^{\alpha + 2\beta},
\qquad
s_3=c_3e^{\alpha + 3\beta},
\eea
where $\al, \beta$ are real parameters which tend to $ +\infty$ and
$c_1,c_2,c_3$ are fixed generic real numbers.   
If $f(s_1,s_2,s_3)$  is a quasi-homogeneous polynomial of bi-degree $(d,q)$, then
\bea
\ln f(c_1e^{\alpha + \beta}, c_1e^{\alpha + \beta},c_1e^{\alpha + \beta})
\eea
has {\it leading term}  $d\al + q\beta$ independent of the choice of $c_1,c_2,c_3$, which may be considered as a vector
$(d,q)$.

\subsubsection{}
We say that a polynomial $P(s_1,s_2,s_3)$ is a {\it $*$-Markov polynomial} if 
there exists a Markov triple $(a,b,c)$,  $0<a < b$, $0< c < b$, $6\leq b$, with  reduced polynomial
presentation $(f_1^*,f_2,f_3^*)\in \Ga_MI$, such that $P=f_2$. 
The polynomial $s_2$ will also be called a {\it $*$-Markov polynomial}.
\vsk.2>

We say that a polynomial $Q(s_1,s_2,s_3)$ is a {\it dual $*$-Markov polynomial} if 
$Q$ is not divisible by $s_3$ and $\mu( Q)$ is a $*$-Markov polynomial. 
\vsk.2>

For example, $s_1^2-s_2$, $s_2(s_1^2-s_2)-s_3s_1$ are $*$-Markov polynomials, 
since they appear as the middle terms in the reduced polynomial presentations in 
\eqref{Xe1} and \eqref{Xe11} and $s_2^2-s_1s_3$,  $s_1(s_2^2-s_1s_3) - s_3s_2$ are
the corresponding dual $*$-Markov polynomials.

\begin{thm}
\label{thm 4.14i}
Let $(f_1^*,f_2 , f_3^*)\in \GI$ be the reduced polynomial presentation of
a Markov triple $(a,b,c)$,  $0<a < b$, $0< c < b$, $6\leq b$.
Then each of $f_1,f_3$ is either a $*$-Markov polynomial or a dual $*$-Markov polynomial.
If $f_1,f_2,f_3$ have bi-degree vectors $(d_1,q_1), (d_2,q_2), (d_3,q_3)$, then
\bea
(d_2,q_2) = (d_1,q_1) +(d_3,q_3).
\eea

\end{thm}

\begin{thm}
\label{thm bideg}
Let $f(s_1,s_2,s_3)$ be
a $*$-Markov polynomial or a dual $*$-Markov polynomial of bi-degree
$(d,q)$. Then $f(s_1,s_2,s_3)$ is a quasi-homogeneous polynomial with
respect to weights $(1,2,3)$. 
Moreover, $|2q-3d|=1$ if $d$ is odd and $|2q-3d|=2$ if $d$ is even.

\end{thm}

See Theorems \ref{thm 4.14}, \ref{thm rebi}, and examples \eqref{Xe1}-\eqref{Xe3}.

\subsubsection{}
It is convenient to put Markov triples at the vertices of the infinite binary planar tree as in Figure \ref{introdis}
and obtain what is called the Markov tree.
Similarly, we may put at the vertices the triples of polynomials $(f_1,f_2,f_3)$, such that the triples
$(f_1^*,f_2,f_3^*)$ are the reduced polynomial presentations of the corresponding Markov triples.
In that way we would put in Figure \ref{introdis}  the triple  $(f_1,f_2,f_3)$ shown in \eqref{Xe11} instead of (3,15,6), the triple
 $(f_1,f_2,f_3)$ shown in  \eqref{Xe2} instead of (3,39,15), the triple
 $(f_1,f_2,f_3)$ shown in  \eqref{Xe3} instead of (15,87,6).
 Or we may put at the vertices
the triples  $(d_1,q_1), (d_2,q_2), (d_3,q_3)$ of bi-degree vectors of 
the triples $(f_1,f_2, f_3)$, or the triples $(d_1,d_2,d_3)$
 of degrees,  see  Figure \ref{introdis}, or we may
 even put at the vertices the triples of Newton polytopes
of the polynomials $(f_1,f_2, f_3)$, see Sections \ref{sec NPP} and \ref{sec dti}. 
\begin{figure}
\centering
\def\svgscale{.8}
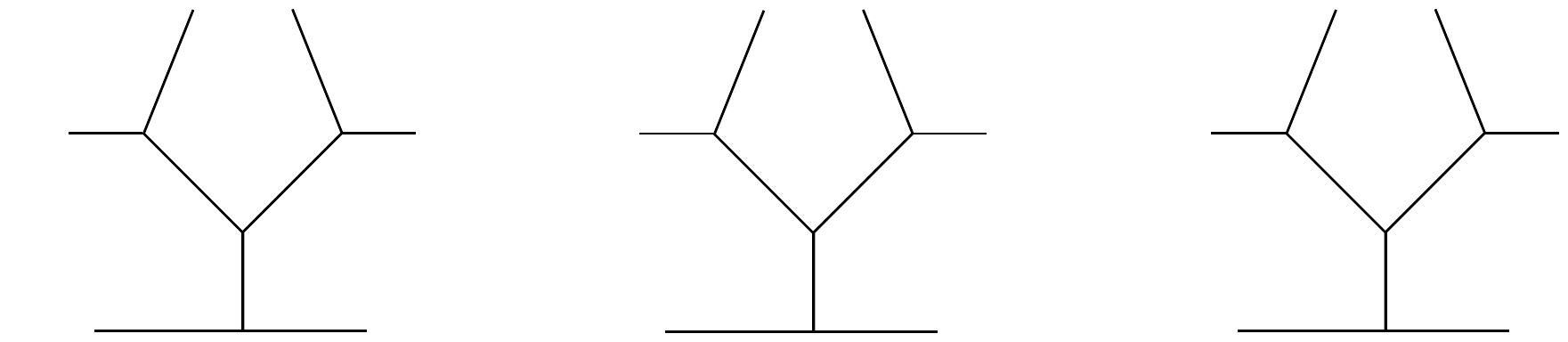
\caption{}
\label{introdis}
\end{figure}

\vsk.2>
These decorated trees have interesting interrelations consisting of ``quatizations'' and 
``de-quantizations'', see short discussion in Section \ref{sec OP}.

\vsk.2>
A compelling
problem is to study asymptotics of these decorations along the
  infinite paths going from the root of the tree to infinity, see remarks in Sections \ref{sec dti} and \ref{sec CS}.

\subsubsection{} It is well known that the Markov triples of the left branch
 of the Markov tree are composed of the odd Fibonacci numbers, multiplied by 3. These triples have the form 
 $(3, 3\phi_{2n+1}, 3\phi_{2n-1})$, where $\phi_{2n+1}$, $\phi_{2n-1}$ are odd Fibonacci numbers. 
 We describe the reduced polynomial presentations 
 $(g^*_{n-1}, F_{2n+1}, F_{2n-1}^*)$ of these Markov triples,
  where 
  $g_{n-1} = s_2$ if $n$ is even,  $g_{n-1}=s_1$ if $n$ is odd,
   and  $F_{2n+1}, F_{2n-1}$ are polynomial in $s_1,s_2,s_3$,
  called the odd $*$-Fibonacci polynomials.
 
 The first of them are 
 \begin{align*}
F_3(\bs s)&=s_1^2-s_2,\\
F_5(\bs s)&=s_1^2s_2 - s_1s_3-s_2^2,\\
F_7(\bs s)&= s_1^3s_2-2  s_1^2s_3- s_1s_2^2+s_2 s_3.
\end{align*}

 We describe the recurrence relations for the odd $*$-Fibonacci polynomials, explicit formulas for them, their Newton polytopes,
 the Binet formula, the Cassini identity, describe the 
 continued fractions for $F_{2n+3}/F_{2n+1}$ and the limit of this ratio as $n\to \infty$.

\subsubsection{}
It is  well known that the Markov triples of the right branch
 of the Markov tree are composed of the odd Pell numbers, multiplied by 3.
  We describe the reduced polynomial presentations of these Markov triples in terms of the polynomials,
  which we call the odd $*$-Pell
polynomials. We develop the properties of the odd $*$-Pell polynomials, which are
analogous to  properties of odd Pell numbers and  to properties
of the odd $*$-Fibonacci polynomials.

\subsubsection{}
The $q$-deformations of Fibonacci and Pell numbers is an active subject related to
several branches of combinatorics and number theory, see, e.g., 
\cite{Ca74, An86, Pa06, MO20}  and references therein.
It would be interesting to determine if these numerous $q$-deformations 
of Fibonacci and Pell numbers could be obtained by specifications of our $*$-deformation
depending on the three parameters $s_1,s_2,s_3$.

 \subsection{$*$-Analogs of the Dubrovin Poisson structure}

In \cite{Du96} Dubrovin  considered  $\C^3$ with coordinates $(a,b,c)$,\ the braid group $\mc B_3$ action
\eqref{Br a}, and  introduced a Poisson structure on $\C^3$,
\bea
\{a,b\}_H = 2c-ab,
\qquad
\{b,c\}_H = 2a-bc,
\qquad
\{c,a\}_H = 2b-ac.
\eea
which is braid group invariant and has the polynomial $a^2+b^2+c^2-abc$ as a Casimir element\footnote{Recall that a function $f$ is a Casimir element for a Poisson structure $\{\,,\,\}$ if
$\{f,g\}=0$ for any $g$.}. 
%, see  
%generalizations  in \cite{Ug99, Bon99,  Bo01, OR02}.

The second topic of this paper is a construction of a $*$-analog of the Dubrovin Poisson structure.
 Our Poisson structure is defined on $\C^6$, is anti-invariant
with respect to  the braid group $\mc B_3$ action
\eqref{tau1i}, 
is invariant with respect to the involution
\bea
(a,a^*, b,b^*,c,c^*) \mapsto (a^*,a,b^*,b,c^*, c),
\eea
 has the polynomials
\bea
&&
aa^*+bb^*+cc^*-abc ,
\\
&&
aa^*+bb^*+cc^*-a^*b^*c^*,
\eea
as Casimir elements, and is log-canonical, see 
Section \ref{sec NPo}. Here the word anti-invariant means that the Poisson structure is multiplied by $-1$ under the action
of generators of the braid group. 
Recall also that a Poisson structure on a space with coordinates $x_1,\dots,x_n$ is log-canonical if
$\{x_i,x_j\}= a_{ij} x_ix_j$ for all $i,j$, where $a_{i,j}$ are constants. Our log-canonical Poisson structure has 
$a_{i,j}=\pm 1, 0$.

\subsubsection{} 
The space $\C^3$ considered by Dubrovin is actually identified with the group $U_3$ of unipotent upper triangular matrices (the Stokes matrices of three dimensional Frobenius manifolds). Standing on such an identification, M.\,Ugaglia generalized the construction of Dubrovin's Poisson structure to all groups $U_n$, see \cite{Ug99} for the explicit equations. Remarkably enough, the same braid invariant Poisson structure on $U_n$ was found independently also in \cite{Bo01,Bon99} from two completely different perspectives. Let $B_\pm$ be the groups of upper and lower triangular $n\times n$ matrices. In \cite{Bo01}, P.\,Boalch proved that $U_n$ is the stable locus of a Poisson involution of the Poisson-Lie group $B_+*B_-$, and that the standard Poisson structure of $B_+*B_-$ induces the braid invariant Poisson structure on $U_n$. The construction in \cite{Bon99}, is based on the identification of the group $U_n$ with the space of Gram matrices $(\chi(E_i,E_j))_{i,j}$ for exceptional collections $(E_1,\dots, E_n)$ in triangulated categories\footnote{Notice that the two identification of $U_n$ as Stokes matrices or Gram matrices of the $\chi$-pairing should coincide, at least for quantum cohomologies, according to a conjecture of Dubrovin, see \cite{Du98,CDG}.}. A.\,Bondal discovered a symplectic groupoid whose space of objects is $U_n$: the existence of a braid invariant Poisson structure on $U_n$ is then deduced from the general theory of symplectic groupoids. The quantization of the Poisson structure on $U_n$ is also known as Nelson-Regge algebra in 2+1 quantum gravity \cite{NR89,NRZ90}, and as Fock-Rosly bracket in Chern-Simons theory \cite{FR97}. Furthermore, L.\,Chekhov and M.\,Mazzocco generalized the construction of the Dubrovin Poisson structures to the space of bilinear forms with block-upper-triangular Gram matrix, they also extensively studied the related Poisson algebras, their quantization and affinization, see \cite{CM11,CM13}. See very interesting short paper \cite{CF00} by L.O.\,Chekhov and V.V.\,Fock.

\vsk.2>

It would be interesting to see the $*$-analogs of these considerations.

\subsection{$*$-Analogs of the Horowitz theorem}

 In \cite{H}, R.D.\,Horowitz proved the following result, 
 characterizing the  Markov group as a subgroup of the group of ring automorphisms of $\Z[a,b,c]$.

\begin{thm}[{\cite[Theorem 2]{H}}]
\label{thm Hor-c}

The group 
 of ring automorphisms of $\Z[a,b,c]$, which preserve the polynomial
\[
H=a^2+b^2+c^2-abc,
\]
 is isomorphic to the  Markov group.
\end{thm}

As the third topic of this paper we develop $*$-analogs of the Horowitz theorem, see Section
\ref{sec *Hor} and Appendix \ref{appendix}.

\subsection{Exposition of material} 

In Section \ref{sec 2} we introduce the $*$-Markov equation and  evaluation morphism.
The $*$-Markov group and its subgroups, in particular,
 the important $*$-Vi\`ete subgroup, are defined in Section \ref{sec 3}. 
The Markov and extended Markov trees are introduced in Section \ref{sec 4}.
In Section \ref{sec 5} we introduce the notion of a distinguished representative of a Markov triple and 
show that the $*$-Vi\`ete group  acts freely and transitively on the set of distinguished representatives.

\vsk.2>
In Section \ref{sec 6} we introduce the notion of an admissible triple of Laurent polynomials and the notion
of a reduced polynomial presentation of a Markov triple. 
One of the main theorems of the paper, 
Theorem \ref{thm 4.14} says that a Markov triple has a unique reduced polynomial presentation. 
We also introduce the notion of a $*$-Markov polynomial.

\vsk.2>
In Section \ref{sec 7} six decorated  infinite planar binary trees are defined.
They are  the  $*$-Markov polynomial tree,  $2$-vector tree, 
 matrix tree,  deviation tree,  Markov tree,  Euclid tree.
We discuss the interrelations between the trees. An interesting problem is to study the 
asymptotics of the decorations along the infinite paths from the root of the tree to infinity.

\vsk.2>
In Sections \ref{sec Fib} and \ref{sec Pell}
we introduce the odd $*$-Fibonacci and odd $*$-Pell polynomials and discuss their properties.

\vsk.2>
In Section \ref{sec *M act} we construct actions of the $*$-Markov group on the spaces $\C^6$ and $\C^5$ and
a map $F:\C^6\to\C^5$ commuting with the actions.  Using these objects we construct
 equivariant Poisson structures on  $\C^6$ and $\C^5$ in Section \ref{sec NPo}.

\vsk.2>
In Section \ref{sec *Hor} we establish $*$-analogs of the Horowitz theorem on $\C^6$ and $\C^5$.
In Appendix \ref{appendix} we discuss more analogs of the Horowitz theorem.

\vsk.2>
In Appendix  \ref{app B} we  discuss briefly the $*$-equations for $\Bbb P^3$ and associated Poisson structures
on $\C^{12}$.

\vsk.4>

The authors thank P.\,Etingof, M.\,Mazzocco, V.\,Ovsienko, V.\,Rubtsov,  V.\,Schechtman,
M.\,Shapiro, L.\,Takhtajan, 
 A.\,Veselov,   A.\,Zorich for useful discussions.  The authors thank HIM and MPI in Bonn,
 Germany for hospitality
 in 2019.

\vsk.2>
The first author was supported by MPIM in Bonn, Germany, the EPSRC Research Grant EP/P021913/2, and by the FCT Project PTDC/MAT-PUR/ 30234/2017 ``Irregular connections on algebraic curves and Quantum Field Theory''.
The second author was supported in part by NSF grants DMS-1665239, DMS-1954266.

\section{$*$-Markov equation}
\label{sec 2}

\subsection{$*$-Involution}

Denote $\bm z = (z_1,z_2,z_3)$, $\bs s=(s_1,s_2,s_3)$. Let $\Zsym$ be the ring of 
symmetric Laurent polynomials in $\bm z$ with integer coefficients.
We define an isomorphism 
$\Zsym\cong\mathbb Z[s_1,s_2,s_3^{\pm 1}]$ by sending
\bea
(z_1+z_2+z_3,\
z_1z_2+z_1z_3+z_2z_3,\
z_1z_2z_3)\quad \to\quad (s_1,s_2,s_3).
\eea
Define the involution 
\beq\notag
(-)^*\colon \Zsym\to\Zsym,\quad f\mapsto f^*,
\eeq 
where
\beq\notag
f^*(\bm z):=
f\left( \frac{1}{z_1}, \frac 1{z_2},\frac1{z_3}\right),\quad f\in\Zsym.
\eeq
This induces a $*$-involution on $\Zs$.
\bea
f^*(s_1,s_2,s_3) =  f\left(\frac{s_2}{s_3}, \frac{s_1}{s_3}, \frac{1}{s_3}\right).
\eea
Denote
$\bs s_o:=(3,3,1)$.
Define the evaluation morphisms
\bea
&&
{\evo}\colon \Zs\to \Z, \qquad \quad \quad f(\bm s)\mapsto f(\bm s_o),
\\
&&
\Evo\colon \left(\Zs\right)^3\to\Z^3,\qquad (a,b,c)\mapsto \left(a(\bm s_o),b(\bm s_o),c(\bm s_o)\right).
\eea
The evaluation morphism corresponds to the evaluation of a Laurent polynomial $f(z_1,z_2,z_3)$ at
$z_1=z_2=z_3=1$.

\subsection{Evaluation morphism} 

The {\it $*$-Markov equation} is the equation
\beq
\label{ME}
aa^*+bb^*+cc^*-abc= \frac{3s_1s_2-s_1^3}{s_3},
\eeq
where $a,b,c\in\Zs$. The solution
\beq\label{ES}
I
=\left(s_1,\frac{s_1}{s_3},s_1\right) 
\eeq
is call the {\it initial solution}.
 
We have
\bea
\evo\left( \frac{3s_1s_2-s_1^3}{s_3} \right)=0\,.
\eea

\begin{prop}\label{Evo}
If $f=(f_1,f_2,f_3)$ is a solution of the $*$-Markov equation \eqref{ME}, then 
$\Evo (f)$   is a solution of the  Markov equation \eqref{Me}.
\qed
\end{prop}

 For example, the evaluation of the  initial solution $I$  gives the triple (3,3,3).

\begin{rem}
The $*$-Markov equation \eqref{ME} can be studied  by looking for solutions $(a,b,c)$ 
in $\Bbb A^3$, where $\Bbb A$ is a ring  more general  than $\Zsym\cong\mathbb Z[s_1,s_2,s_3^{\pm 1}]$.
 
\vsk.2>
For instance, if we look for solutions of the form $(\al s_1,\beta\frac{s_1}{s_3},\ga s_1)$, where
$\al,\beta,\ga\in \C$, then
\bea
\al^2+\beta^2+\ga^2 = 3, \qquad \al\beta\ga =1.
\eea
This curve has infinitely many algebraic points, for example
\bea
\al =\sqrt{-1}\,,\qquad \beta=   \sqrt{2+\sqrt{5}}\,,\qquad 
\ga=
\frac{1}{\sqrt{-2-\sqrt{5}}}\, .
\eea
\end{rem}

\section{Groups of symmetries}
\label{sec 3}

\subsection{Symmetries of  Markov equation} 
\label{sec sym Me}

Consider the following three groups of transformations of $\Z^3$:
\vskip 2mm
{\bf Type I. } The group $\Gc{1}$ generated by transformations
\beq\notag
\lambda_{i,j}^c\colon (a,b,c)\mapsto \left((-1)^i a,(-1)^{i+j} b, (-1)^j c\right),\qquad i,j\in\mathbb Z_2.
\eeq
\vskip 2mm
{\bf Type II. } The group $\Gc{2}$ generated  by transformations
\beq\notag
\sigma_1^c:(a,b,c) \mapsto (b,a,c),\quad \si_2^c:(a,b,c)\mapsto (a,c,b).
\eeq
\vskip 2mm
{\bf Type III. }The group $\Gc{3}$ generated by  transformations 
\bea
&&
\tau_1^c\colon (a,b,c)\mapsto   (-a,\quad c,\quad b-ac),
\\
&&
\tau_2^c\colon (a,b,c)\mapsto (b,\quad a-bc,\quad -c).
\eea
We have  $\tau_1^c\tau_2^c\tau_1^c=\tau_2^c\tau_1^c\tau_2^c$.

\vsk.2>
In these notations the superscript $c$ stays for the word classical.

\begin{rem}
\label{clB3}
Let $\mathcal B_3$ be the braid group with three strands, and $\beta_1,\beta_2$  its 
standard generators (elementary braids) with $\beta_1\beta_2\beta_1=\beta_2\beta_1\beta_2$.
 There is a group epimorphism 
\beq
\notag
\pho\colon\mathcal B_3\to\ \!\! \Gc{3},\quad \beta_i\mapsto \tau_i^c,\quad i=1,2.
\eeq
The center $\mathcal Z(\mathcal B_3)=\langle(\beta_1\beta_2)^3\rangle$ is contained in $\ker\pho$. Thus, the group 
\beq
\notag
\mathcal B_3\big/\mathcal Z(\mathcal B_3)\cong \on{PSL}(2,\Z)
\eeq
acts on the set of solutions of \eqref{Me}.
\end{rem}

\begin{prop}
The set of nonzero Markov triples is invariant under the action of each of the groups $\Gc{1},\Gc{2},\Gc{3}$.\qed
\end{prop}

\subsection{Markov and Vi\`ete  groups}

Define the \emph{Markov group} $\mg$ as the group of transformations of $\Z^3$ generated by $\Gc{1},\ \!\!\Gc{2},\ \!\!\Gc{3}$,
\beq
\mg:=\langle \Gc{1},\ \!\! \Gc{2},\ \!\!\Gc{3}\rangle.
\eeq
Define the \emph{Vi\`ete involutions} $v_1^c,v_2^c,v_3^c\in\ \!\! \mg$ by the formulas
\begin{align}
\notag
v_1^c\colon(a,b,c)&\mapsto(bc-a,b,c),
\\
\notag
v_2^c\colon(a,b,c)&\mapsto(a,ac-b,c),
\\
\notag
v_3^c\colon(a,b,c)&\mapsto(a,b,ab-c).
\end{align}
Define the \emph{Vi\`ete group} $\vg$ as the group generated by the Vi\`ete
involutions $v_1^c,v_2^c,v_3^c$, 
\beq
\vg:=\langle v_1^c,v_2^c,v_3^c\rangle.
\eeq
We have
\bean
\label{vtau1cl}
v_1^c=\lambda_{1,1}^c\ \sigma_1^c\ \tau_2^c,\qquad
v_2^c=\lambda_{1,0}^c\ \sigma_2^c\ \tau_1^c,\qquad
v_3^c=\lambda_{1,1}^c\ \tau_1^c\ \sigma_2^c.
\eean

\begin{thm}[{\cite[Theorem 1]{EH}}]\label{EH}
The group $\vg$ is freely generated by $v_1^c, v_2^c, v_3^c$, that is,
 $\vg\cong\Z_2*\Z_2*\Z_2$.\qed
\end{thm}

\begin{prop}\label{identcl}
We have the following identities:
\bea
\si_1^c\la_{k,l}^c\si_1^c=\la_{k+l,l}^c,
\qquad
\si_2^c\la_{k,l}^c\si_2^c=\la_{l, k+l}^c,
\eea
\begin{alignat}{2}
\notag
\si_1^c v_1^c\si_1^c&=v_2^c,\qquad \si_2^c v_1^c\si_2^c&=v_1^c,\\
\notag
\si_1^c v_2^c\si_1^c&= v_1^c,\qquad \si_2^c v_2^c\si_2^c&=v_3^c,\\
\si_1^c v_3^c\si_1^c&=v_3^c,\qquad \si_2^c v_3^c\si_2^c&=v_2^c,
\notag
\end{alignat}
\beq
\notag
\la_{k,l}^c v_i^c\la_{k,l}^c =v_i^c,\qquad i=1,2,3.
\eeq
\qed
\end{prop}

\begin{cor}
We have $\mg=\langle\vg,\Gc{1},\Gc{2}\rangle$. Moreover, $\vg$ is a normal subgroup of $\mg$.
\end{cor}
\proof
The inclusion $\mg\supseteq\langle\vg,\Gc{1},\Gc{2}\rangle$ is clear. 
We have $\Gc{3}\subseteq\langle\vg,\Gc{1},\Gc{2}\rangle$,
by equations \eqref{vtau1cl}. 
Hence $\mg=\langle\vg,\Gc{1},\Gc{2}\rangle$. 
We have  $gv_i^cg^{-1}\in\vg$ for any $g\in\Gc{1},\Gc{2}$ by 
Proposition \ref{identcl}.
\endproof

\begin{prop}\label{uniqcl}
We have $\vg\cap\langle\Gc{1},\Gc{2}\rangle=\left\{{\rm id}\right\}$.
\end{prop}
\proof
Any element of $\vg$ fixes the triple $(2,2,2)$. The only elements of $\langle\Gc{1},\Gc{2}\rangle$ which fix $(2,2,2)$ are the elements of $\Gc{2}$.

Extend the action of both $\vg$ and $\Gc{2}$ to the space $\C^3$. The point $(0,0,0)$ is a fixed point for both actions. The Jacobian matrices at $(0,0,0)$ of the Vi\`ete transformations $v_1^c,v_2^c,v_3^c$ are
\[
\begin{pmatrix}
-1&0&0\\
0&1&0\\
0&0&1
\end{pmatrix},\quad
\begin{pmatrix}
1&0&0\\
0&-1&0\\
0&0&1
\end{pmatrix},\quad
\begin{pmatrix}
1&0&0\\
0&1&0\\
0&0&-1
\end{pmatrix},
\]respectively. Hence, any element of $\vg$ has diagonal Jacobian matrix at $(0,0,0)$. The only transformation of $\Gc{2}$ which can be represented by a diagonal matrix is the identity.
\endproof

\begin{cor}\label{declv12}For any element $g\in\ \!\!\!\mg$, there exist unique $v\in\ \!\!\!\vg$ and $h\in \langle\Gc{1},\ \!\!\Gc{2}\rangle$ such that $g=vh$.
This implies that $\mg=\ \vg\rtimes \langle\Gc{1},\ \!\!\Gc{2}\rangle$.
\end{cor}
\proof
Since $\mg=\langle \vg,\ \!\! \Gc{1},\ \!\!\Gc{2}\rangle$, any $g$ can be expressed as a product 
\beq
\notag
g=v_{i_1}a_{i_1}v_{i_2}a_{i_2}\dots v_{i_k}a_{i_k},\quad a_{i_j}\in\langle \Gc{1},\ \!\!\Gc{2}\rangle.
\eeq
We can factor $g$ as 
\beq
\label{decl}
g=v_{i_1}v_{i_2}\dots v_{i_k}a'_{i_1}a'_{i_2}\dots a'_{i_k},\quad a'_{i_j}\in\langle \Gc{1},\ \!\!\Gc{2}\rangle,
\eeq
by using the commutation rules described in Proposition \ref{identcl}.
The decomposition in \eqref{decl} is unique, by Proposition \ref{uniqcl}.
\endproof

\subsection{Symmetries of $*$-Markov equation} 
\label{sec sym *Me}

Consider the 
following four groups of transformations of the space
$\left(\Zs\right)^3$.

\vskip 2mm
{\bf Type I. } The group $\Gs{1}$ generated by transformations 
\beq\notag
\lambda_{i,j}\colon (a,b,c)\mapsto \left((-1)^i a,(-1)^{i+j} b, (-1)^j c\right),\qquad i,j\in\mathbb Z_2.
\eeq

\vskip 2mm
{\bf Type II. } The  group $\Gs{2}$ generated by transformations
\beq\notag
\sigma_1 :(a,b,c)\mapsto (b,a,c),\qquad \si_2: (a,b,c)\mapsto (a,c,b).
\eeq

{\bf Type III. } The  group $\Gs{3}$ generated by  transformations  
\bean
\label{tau1}
&&
\tau_1\colon (a,b,c)\mapsto   (-a^*,\quad c^*,\quad b^*-ac),
\\
\notag
&&
\tau_2\colon (a,b,c)\mapsto (b^*,\quad a^*-bc,\quad -c^*).
\eean
We have $\tau_1\tau_2\tau_1=\tau_2\tau_1\tau_2$.
\vsk.2>
{\bf Type IV. } The group $\Gs{4}$ generated by transformations
\beq
\notag
\mu_{i,j}\colon (a,b,c) \mapsto
(s_3^i a,  s_3^{-i-j} b,    s_3^j c),
\qquad i,j 
\in\mathbb Z.
\eeq

\begin{prop}
\label{prop inv}
The set of all solutions of the
$*$-Markov equation \eqref{ME} is invariant under the action of each of the groups $\Gs{1},\Gs{2},\Gs{3},\Gs{4}$.\qed
\end{prop}

As in the  case of the Markov equation \eqref{Me}, we have the action of $\mc B_3\big/\mc Z(\mc B_3)\cong \on{PSL}(2,\Z)$ on the 
set of all solutions of the $*$-Markov equation \eqref{ME}. See Remark \ref{clB3}.

\subsection{$*$-Markov and $*$-Vi\`ete groups}

Define the  $*$-\emph{Markov group} $\smg$ as the group of transformations of
$\left(\Zs\right)^3$ generated by $\Gs{1},\ \!\!\Gs{2},\ \!\!\Gs{3},\ \!\!\Gs{4}$,
\beq
\label{*MG}
\smg:=\langle \Gs{1},\ \!\! \Gs{2},\ \!\!\Gs{3},\ \!\!\Gs{4}\rangle.
\eeq
Define the $*$-\emph{Vi\`ete involutions} $v_1,v_2,v_3\in\ \!\! \smg$ 
by the formulas
\bea
v_1\colon(a,b,c)&\mapsto(bc-a^*,b^*,c^*),\\
v_2\colon(a,b,c)&\mapsto(a^*,ac-b^*,c^*),\\
v_3\colon(a,b,c)&\mapsto(a^*,b^*,ab-c^*).
\eea
Define the $*$-\emph{Vi\`ete group} $\svg$  as the group generated by 
Vi\`ete involutions $v_1,v_2,v_3$, 
\beq\notag
\svg:=\langle v_1,v_2,v_3\rangle.
\eeq
We have
\bean
\label{vtau1}
v_1 =\lambda_{1,1} \sigma_1 \tau_2,\qquad
v_2=\lambda_{1,0} \sigma_2\tau_1,\qquad
v_3=\lambda_{1,1} \tau_1\sigma_2.
\eean

\begin{prop}
\label{ident}

We have the following identities,
\bea
\si_1\la_{k,l}\si_1=\la_{k+l,l},
\qquad
\si_2\la_{k,l}\si_2=\la_{k,k+l},
\eea
\begin{alignat}{2}
\notag
\si_1 v_1\si_1&=v_2,\qquad \si_2 v_1\si_2&=v_1,\\
\notag
\si_1 v_2\si_1&= v_1,\qquad \si_2 v_2\si_2&=v_3,\\
\notag
\si_1 v_3\si_1&=v_3,\qquad \si_2 v_3\si_2&=v_2,
\end{alignat}
\bea
\la_{k,l} v_i\la_{k,l} =v_i,\qquad i=1,2,3,
\eea
\begin{alignat}{2}
\notag
\la_{k,l}\,\mu_{i,j}\,\la_{k,l}&=\mu_{i,j}\,,\qquad k,l\in\Z_2,\qquad i,j \in\Z,
\\
\notag
\si_1\mu_{i,j}\si_1&=\mu_{-i-j,j}\,,\qquad \si_2\mu_{i,j}\si_2=\mu_{i,-i-j}\,,
\\
\notag
v_k\,\mu_{i,j}\,v_k&=\mu_{-i-j}\,,\qquad k=1,2,3.
\end{alignat}
\end{prop}

\proof
These identities are proved by straightforward computations.
\endproof

\begin{cor}\label{dec124}
For any element $g\in\langle G_1,G_2,G_4\rangle$, there exist unique $g_1\in G_1, g_2\in G_2, g_4\in G_4$ such that
\beq\label{124}
g=g_4g_1g_2.
\eeq
\end{cor}

\proof
Any $g\in\langle G_1,G_2,G_4\rangle$ can be put in the form \eqref{124} by Proposition \ref{ident}.
 The uniqueness follows from the identities
\[
\puqed
G_4\cap\langle G_1,G_2\rangle=\left\{\rm id\right\},\quad G_1\cap G_2=\left\{\rm id\right\}.\makeatletter\qed
\poqed
\]

\begin{cor} We have $\smg=\langle \svg,\ \!\!\Gs{1},\ \!\!\Gs{2},\ \!\!\Gs{4}\rangle$. Moreover 
$\svg$ is a normal subgroup of $\smg$.
\end{cor}

\proof The inclusion $\smg\supseteq\langle\Gs{1},\ \!\!\Gs{2},\ \!\!\Gs{4},\ \!\!\svg\rangle$ is clear. 
We have $G_3\subseteq\langle\Gs{1},\ \!\!\Gs{2},\ \!\!\Gs{4},\ \!\!\svg\rangle$,
by equations \eqref{vtau1}.
Hence $\smg=\langle\Gs{1},\ \!\!\Gs{2},\ \!\!\Gs{4},\ \!\!\svg\rangle$.
It follows that $g v_i g^{-1}\in\ \!\!\!\svg$ for any $g\in\ \!\!\Gs{1},\Gs{2},\Gs{4}$, 
by Proposition \ref{ident}.
\endproof

\begin{prop}\label{uniq}
We have $\svg\cap \langle \Gs{1},\ \!\!\Gs{2},\ \!\!\Gs{4}\rangle=\left\{\rm id\right\}$.
\end{prop}

\proof Let $g\in\svg\cap \langle \Gs{1},\ \!\!\Gs{2},\ \!\!\Gs{4}\rangle$. We have $g\in\ker\phi_M$
by Proposition \ref{uniqcl}. 
The only elements of the form \eqref{124} which are in $\ker\phi_M$ are the elements of $G_4$.
Any element of $\svg$ fixes the triple of constant polynomials $(2,2,2)$. The only element of $G_4$ which fixes $(2,2,2)$ is the identity. 
\endproof

\begin{cor}\label{decv124}For any element $g\in\ \!\!\!\smg$, there exist unique $v\in\ \!\!\!\svg$ and $h\in \langle\Gs{1},\ \!\!\Gs{2},\ \!\!\Gs{4}\rangle$ such that $g=vh$.
This implies that $\smg=\ \svg\rtimes \langle\Gs{1},\ \!\!\Gs{2},\ \!\!\Gs{4}\rangle$.
\end{cor}
\proof
Since $\smg=\langle \svg,\ \!\! \Gs{1},\ \!\!\Gs{2},\ \!\!\Gs{4}\rangle$, any $g$ can be expressed as a product 
\beq
\notag
g=v_{i_1}a_{i_1}v_{i_2}a_{i_2}\dots v_{i_k}a_{i_k},\quad a_{i_j}\in\langle \Gs{1},\ \!\!\Gs{2},\ \!\!\Gs{4}\rangle.
\eeq
We can factor $g$ as 
\beq\label{dec}
g=v_{i_1}v_{i_2}\dots v_{i_k}a'_{i_1}a'_{i_2}\dots a'_{i_k},\quad a'_{i_j}\in\langle \Gs{1},\ \!\!\Gs{2},\ \!\!\Gs{4}\rangle
\eeq
by using the commutation rules described in Proposition  \ref{ident}.
The decomposition in \eqref{dec} is unique, by Proposition \ref{uniq}.
\endproof

Let $g\in\smg$. Consider its restriction $g|_{\Z^3}$ to the subset $\Z^3\subseteq \left(\Zs\right)^3$. Define the transformation $\phi_M(g)\colon \Z^3\to \Z^3$ as the composition $\Evo\circ g|_{\Z^3}$, i.e.
\[
\xymatrix{
\Z^3\ar[rr]^{g|_{\Z^3}\quad}&&\left(\Zs\right)^3\ar[rr]^{\quad\Evo}&&\Z^3.
}
\]

\begin{prop}
\label{surjmg}
We have a group epimorphism
\beq 
\notag
\phi_M\colon\smg\twoheadrightarrow\mg,
\eeq 
which acts on the generators as 
\beq	\label{phim}
\lambda_{\al,\bt}\mapsto\la^c_{\al,\bt},\quad \si_i\mapsto\si_i^c,\quad \tau_i\mapsto \tau_i^c,\quad \mu_{\al,\bt}\mapsto {\rm id}.
\eeq
\end{prop}

\proof
Identities \eqref{phim} are easily checked. Let $g,h\in\smg$. From the commutative diagram
\[
\xymatrix{
\left(\Zs\right)^3\ar[r]^{g}&\left(\Zs\right)^3\ar[r]^h\ar@/^/[d]^{\Evo}&\left(\Zs\right)^3\ar[d]^{\Evo}\\
\Z^3\ar[u]^i\ar[r]_{\phi_M(g)}&\Z^3\ar@/^/[u]^i\ar[r]_{\phi_M(h)}&\Z^3
}
\]it readily follows that $\phi_M(hg)=\phi_M(h)\phi_M(g)$.
\endproof

\begin{prop}\label{kerphiM}
We have $\ker\phi_M=G_4$, so that $\mg\cong\smg/G_4$.
\end{prop}

\proof
Let $g\in\ker\phi_M$. By Corollaries \ref{decv124} and \ref{dec124}, there exist unique elements $v\in\svg, g_1\in G_1,g_2\in G_2, g_4\in G_4$ such that
\[g=vg_4g_1g_2.
\]We have
\[\phi_M(g)=\phi_M(v)\phi_M(g_1)\phi_M(g_2)=\rm id.
\]
By Corollary \ref{declv12}, together with $G_1^c\cap G_2^c=\left\{\rm id\right\}$, we have 
\[\phi_M(v)={\rm id},\quad \phi_M(g_1)={\rm id},\quad \phi_M(g_2)={\rm id}.
\]
This clearly implies that $g_1=\rm id$ and $g_2=\rm id$.
The element $v$ is of the form $v=\prod_{j=1}^nv_{i_j}$ with $i_j\in\left\{1,2,3\right\}$, so that
\[{\rm id}=\phi_M(v)=\prod_{j=1}^n\phi_M(v_{i_j})=\prod_{j=1}^nv_{i_j}^c.
\]Since $v_i^c$ freely generate $\vg$ (by Theorem \ref{EH}), we necessarily have $n=0$, and $v={\rm id}$. This shows that $\ker\phi_M\subseteq G_4$. The opposite inclusion is obvious.
\endproof

\begin{lem}
The morphism $\phi_M$ defines isomorphisms between the group $G_i$ and $G_i^c$ for $i=1,2,3$, and between the group $\svg$ and $\vg$.\qed
\end{lem}

\begin{lem}\label{equivEvo}The evaluation morphism $\Evo$ is $\phi_M$-equivariant, i.e.
\[
\Evo(g\cdot \bm a)=\phi_M(g)\cdot \Evo(\bm a),\quad g\in\smg,\quad \bm a\in\left(\Zs\right)^3.
\makeatletter\displaymath@qed
\]
\end{lem}

\section{Markov trees}
\label{sec 4}

\subsection{Decomposition of $\Mcl$}

The set $\mc M^c$ of all nonzero solutions of the  Markov equation \eqref{Me}
admits a partition in four subsets,
\beq
\Mcl=\Mclp\cup\Mclmone\cup\Mclmtwo\cup\Mclmthr,
\eeq
where $\Mclp$ consists of all 
positive triples and $\mc M_{ij}^c$ consists of all
triples with negative entries in the $i$-th and $j$-th position.

We have a  projection $\pi: \Mcl\to \Mclp$,  forgetting the minuses. 

\begin{lem}
The action of the Vi\`ete group $\vg$ on $\mc M^c$ preserves each of $\Mclp,\Mclmone,\Mclmtwo,\Mclmthr$. 
The action of the group $\langle \Gc{2}, \vg\rangle$ on $\Mcl$ preserves
the sets $\Mcl_{+}$ and $\Mclmone\cup\Mclmtwo\cup\Mclmthr$ and
commutes with the projection $\pi: \Mcl\to \Mclp$.
\qed
\end{lem}

\subsection{Markov tree} 
Solutions of the  Markov equation \eqref{Me} can be arranged in a graph, called the \emph{Markov tree}. 
\vskip 2mm

Define
 $ L:=\sigma_2^cv_3^c, \ R:=\sigma_1^cv_1^c \in\mg$.  Given $(x,y,z)\in\mc M^c_+$, we have
\beq
\label{optree}
\xymatrix@R-1pc{
(x,xy-z,y)&&(y,yz-x,z)\\
&(x,y,z)\ar[ul]^L\ar[ur]_R&
}
\eeq
The Markov tree $\mc T$ is the infinite graph obtained by iterating 
the operations \eqref{optree} starting from the initial solution $(3,3,3)$. 
{\small\[
\xymatrix@R-1.5pc@C-2pc{
&&&\dots&&&\\
(3, 102,39)&&(39, 582,15)&&(15,1299,87)&&(87, 507,6)\\
&(3,39,15)\ar[ul]\ar[ur]&&&&(15, 87,6)\ar[ul]\ar[ur]&\\
&&&(3,15,6)\ar[ull]\ar[urr]&&&\\
&&&(3,6,3)\ar[u]&&&\\
&&&(3,3,3)\ar[u]&&&\\
}
\]
}

\begin{thm}[{\cite{Mar1,Mar2}\cite[Theorem 3.3]{Aig}}]\label{thmAig}
Up to permutations in $G_2^c$, all the elements of $\mc M^c_+$ appear exactly once in the Markov tree $\mc T$. \qed
\end{thm}

\begin{cor}\label{trancmg}
The  group $\mg$ acts transitively on the set $\Mcl$.
\end{cor}
\proof
Let $\bm x,\bm y\in\Mcl$. There exist $\gamma_1,\gamma_2\in\langle G_1^c, G_2^c\rangle$ such that $\gamma_1\bm x,\gamma_2\bm y$ are vertices of $\mc T$, by Theorem \ref{thmAig}. So, there exist $\delta_1,\delta_2\in\mg$ such that $\delta_1(3,3,3)=\gamma_1\bm x$ and $\delta_2(3,3,3)=\gamma_2\bm y$. 
We have
\[
\gamma_2^{-1}\delta_2\delta_1^{-1}\gamma_1\bm x=\bm y. \makeatletter\displaymath@qed
\]

\begin{thm}
[{\cite[Lemma 3.1]{Aig}}]
\label{lemAig}

The triples $(3,3,3)$ and $(3,6,3)$ are the only vertices of $\mc T$ with repeated numbers.
\qed
\end{thm}

\subsection{Extended Markov tree}

Define the \emph{extended Markov graph} as the infinite graph $\mc {T}^{\on{ext}}$, with vertex set  $\Mclp$.
We connect two vertices $(a,b,c)$, $(a\rq{},b\rq{},c\rq{})$ of $\mc {T}^{\on{ext}}$  by an edge if
$(a,b,c)= v_i^c(a\rq{},b\rq{},c\rq{})$ for some $i\in\{1,2,3\}$, where $v_i^c$ are Vi\`ete involutions.

\[
\xymatrix@R-1.5pc@C-1pc{
&&\dots&&\\
&(15,3,6)&&(3,15,6)&\\
\iddots&&(3,3,6)\ar[ul]\ar[ur]&&\ddots\\
(6,15,3)&&(3,3,3)\ar[u]\ar[dl]\ar[dr]&&(15,6,3)\\
&(6,3,3)\ar[ul]\ar[d]&&(3,6,3)\ar[ur]\ar[d]&\\
\ddots&(6,3,15)&&(3,6,15)&\iddots\\
&&\dots&&
}
\]

\begin{thm}
\label{thmemt}

The Vi\`ete group acts freely on the vertex set $\Mclp$ of the extended Markov graph
with one orbit,  $\Mclp = \vg(3,3,3)$. Moreover $\mc {T}^{\on{ext}}$ is a tree.
\end{thm}

The graph $\mc {T}^{\on{ext}}$  is called  the {\it extended Markov tree}.
\vskip1mm
The proof of Theorem \ref{thmemt} requires the following lemma.
Define the function $m\colon\Mclp\to\N$,
 which assigns to a triple $(x,y,z)$ its maximal entry. It is known that  
\beq\notag
\min m(\mc M^{c}_+)=3,
\eeq
and that such a minimum is achieved at $(3,3,3)$.

\begin{lem}\label{uv}
For any $(x,y,z)\in \mc M^{c}_+\setminus\{(3,3,3)\}$ there exists a \emph{unique} 
Vi\`ete transformation $v_i^c$, with $i\in\left\{1,2,3\right\}$, such that $m(v_i^c(x,y,z))<m(x,y,z)$.
\end{lem}
\proof
We have 
\bea
v_1^c(x,y,z)=(yz-x,y,z),
\qquad
v_2^c(x,y,z)=(x,xz-y,z),
\qquad
v_3^c(x,y,z)=(x,y,xy-z).
\eea
We claim that if $m(x,y,z)=x$, then the transformation is $v_1^c$;\
if $m(x,y,z)=y$, then the transformation is $v_2^c$;\
if $m(x,y,z)=z$, then the transformation is $v_3^c$.

To prove the first case we need to show that $yz-x\le x$,
 $xz-y>x$, $xy-z>x$.
 
 We may  assume that $z\leq y<x$.
  Consider the function $\phi\colon\R\to\R$ defined by
\beq\notag
\phi(t):=t^2+y^2+z^2-tyz.
\eeq
We have $\phi(x)=\phi(yz-x)=0$, so that
\beq\notag
\phi(t)=(t-x)(t-(yz-x)).
\eeq
If $yz-x\geq x$, so that $\phi(t)>0$ for all $t<x$. Then on the one hand we have $y<x$, but
on the other hand we have
\begin{align*}
z^2\leq y^2\quad \Rightarrow \quad 2y^2+z^2\leq 3y^2\leq zy^2
\quad
\Rightarrow\quad \phi(y)=2y^2+z^2-y^2z\leq 0.
\end{align*}
This shows that the assumption $yz-x\geq x$ is contradictory. 
We also have 
\begin{align*}
xz-y>3x-y>2x>x,\\
xy-z>3x-z>2x>x.
\end{align*}
This completes the proof in the first case. The other two cases are proved similarly.
\endproof

\begin{cor}
Any $(x,y,z)\in\mc M^{c}_+$ can be transformed to $(3,3,3)$ by an element of the Vi\`ete group $\Gm_V^c$. 
Consequently, $\mc M^c_+ = \Ga^c_V(3,3,3)$.\qed
\end{cor}

\proof[Proof of Theorem \ref{thmemt}]

It is sufficient to prove that if $v(3,3,3)=(3,3,3)$ for some $v\in\Ga_V^c$, 
 then $v=\rm id$. Any element $v\in\Gm_V^c$ is of the form
\beq
\label{genv}
v=\prod_{k=1}^nv_{i_{k}}^c\quad i_k=1,2,3.
\eeq 
Define
\[m_j:=m\Big[\Big(\prod_{k=1}^jv_{i_k}^c\Big)(3,3,3)\Big],\quad j=1,\dots,n.
\]
Define
\beq\notag
M:=\max_{j=1,\dots,n} m_j,\quad J:=\min\left\{j\colon m_j=M\right\}.
\eeq
We claim that $v^c_{i_{J+1}}=v^c_{i_J}$. Indeed, the assumption  $v_{i_{J+1}}\neq v_{i_J}$
would imply  that $m_{J+1}>m_J=M$, which is impossible. Hence, we can decrease 
the number of factors in \eqref{genv} by two. 
By repeating the argument, we prove that all the factors in \eqref{genv} cancel. 

The same argument shows that  the graph $\mc{T}^{\on{ext}}$ has no loops.
\endproof

\begin{cor}
For any $i,j$ the Vi\`ete group $\Gm_V^c$ acts freely on the set $\mc M^c_{ij}$ with one orbit.\qed
\end{cor}

\section{Distinguished representatives}
\label{sec 5}

\subsection{$*$-Markov group orbit of initial solution}

Let $\GI$ be the orbit of the initial solution $I$ of the
$*$-Markov equation \eqref{ME} under the action of the $*$-Markov group $\smg$.
  Any element of  $\GI$   is a solution of the $*$-Markov equation \eqref{ME}, 
  see Proposition \ref{prop inv}.

\begin{prop}

The evaluation morphism $\Evo$ maps the set $\GI$ 
onto the set  $\mc M^c$ of all nonzero solutions of the Markov equation \eqref{Me}.
\end{prop}

\proof
If $\bm a\in\GI$, then $\Evo(\bm a)\in\Mcl$, by Proposition \ref{Evo}. We check surjectivity. Let $\bm x\in\Mcl$. There exists $\gamma\in\mg$ such that $\bm x=\gamma\cdot (3,3,3)$, by Corollary \ref{trancmg}. There exists $\tilde\gamma\in\smg$ such that $\phi_M(\tilde \gamma)=\gamma$, by Proposition \ref{surjmg}. We have that $\Evo(\tilde\gamma\cdot (s_1,s_2^*,s_1))=\bm x$, by Lemma \ref{equivEvo}.
\endproof

\subsection{Initial solution and $*$-Vi\`ete group}

Let $p=(a,b,c) \in \mc M^c_+$. Let $v^p \in \Ga_V^c$ be the unique element 
of the Vi\`ete group such that
$v^p(3,3,3) =(a,b,c)$. Define the {\it distinguished element } $f^p \in \GI$ by the formula
\bea
f^p = v^p I, 
\eea
where $v^p$ is considered as an element of the $*$-Vi\`ete group $\Ga_V$.  Notice that
$\evo(f^p) = (a,b,c)$.

\begin{lem}
For $p', p\in \mc M^c_+$,  let $v^{p',p}\in \Ga_V^c$ be the unique element such that
 $v^{p',p}p=p\rq{}$. Then
$ v^{p',p} f^p = f^{p'}$\,,
where $v^{p',p}$ is considered as an element of  $\Ga_V$.
\end{lem}

\begin{proof} We have
$v^{p',p} f^p = v^{p',p} v^p I = v^{p\rq{}} I\,.$
\end{proof}

\begin{thm}
\label{thm perm}

Let $p=(a,b,c)$,\,  $ p\rq{}=(a\rq{},b\rq{},c\rq{})  \in \mc M^c_+$ be  such that the triple
$p\rq{}$ is a permutation of the triple $p$. Then
$f^{p\rq{}}$ is obtained from $f^p$ by 
 the same permutation of coordinates of $f^p$ composed with a transformation from
 the  group $G_4$.

\end{thm}

\proof 

Let $f^p=(f^p_1, f^p_2,f^p_3)$ and $f^{p'}=(f^{p'}_1, f^{p'}_2,f^{p'}_3)$. Let $\om$
be the permutation such that the evaluation of $\om f^{p'}$ is $(a,b,c)$, the same as the evaluation of $f^p$. 

The triple $\om f^{p'}$ lies in the orbit $\Ga_MI$. So 
\bean
\label{1}
\om f^{p'} = vg_4g_1g_2 I = vg_4g_2I = v\tilde{g}_4I,
\eean
where $v\in \Ga_V$, $g_j\in G_j$, $\tilde g_4\in G_4$.
Here we may conclude that $g_1=1$, since $(a,b,c)$ are positive. 
We may also conclude that 
$g_4g_2I = \tilde{g}_4I$ for some 
$\tilde g_4\in G_4$ since any permutation of coordinates of the initial solution $(s_1,s_1/s_3,s_1)$ 
can be performed by a transformation from $G_4$.
On the other hand, we also have 
\bean
\label{2}
f^p = v I,
\eean
where $v$ in \eqref{2} is the same as in \eqref{1}. We also know that $v\tilde g_4 =  g_4'v$,  for some $g_4'\in G_4$, 
by Proposition \ref{ident}.
Hence
$\om f^{p'} = v \tilde g_4I =  g_4'vI=  g_4' f^p.$
This proves the theorem.
\endproof

\begin{thm}
\label{thm stat}
Let $p=(a,b,c) \in \mc M^c_+$. Let $p\rq{}=(a\rq{},b\rq{},c\rq{}) \in \mc M^c$ be such that
$(a\rq{},b\rq{},c\rq{})$ is obtained from $(a,b,c)$ by a permutation and possibly also by 
change of sign of two coordinates. Let $f\rq{} \in \GI$ be an element, whose evaluation is
$p\rq{}$.  Then $f\rq{}$ is obtained from $f^p$ by an element of $\langle G_1,G_2,G_4\rangle$.

\end{thm}

\proof We have $f\rq{} = g v I$, where $g\in \langle G_1,G_2,G_4\rangle$
and $v\in \Ga_V$. The evaluation of $vI$ has to be a permutation 
of $(a,b,c)$, 
\bea
\Evo(vI) = \si (a,b,c), \qquad \si \in G_2\,.
\eea
Hence $vI=f^{\si(a,b,c)}$. By Theorem \ref{thm perm},
$f^{\si(a,b,c)} =  \mu \si f^{(a,b,c)}$, $\mu\in G_4$. 
Hence
$f\rq{} = gvI = g \mu \si  f^{(a,b,c)}\,,$
that proves the theorem.
\endproof

\section{Reduced polynomials solutions and $*$-Markov polynomials}
\label{sec 6}

\subsection{Degrees of a polynomial}
\label{sec def f}

Let $f(s_1,s_2,s_3)$ be a polynomial. We consider
 two degrees of $f$:\ the 
{\it homogeneous degree} $d:={\rm deg} f $ with respect to weights $(1,1,1)$ 
and the {\it quasi-homogeneous degree} $q:={\rm Deg} f$  with respect to weights $(1,2,3)$.

\begin{lem}
\label{lem degg}  
Let $f(s_1,s_2,s_3)$ be a polynomial of homogeneous degree $d$ not divisible by $s_3$, then
\bean
\label{fg}
g(s_1,s_2,s_3):=s_3^{d} \,f\Big(\frac{s_2}{s_3}, \frac{s_1}{s_3}, \frac{1}{s_3}\Big)
\eean
is a polynomial of homogeneous degree $d$ not divisible by $s_3$. If additionally
$f(s_1,s_2,s_3)$ is a quasi-homogeneous polynomial of quasi-homogeneous degree $q$, then
$g(s_1,s_2,s_3)$ is a quasi-homogeneous polynomial of quasi-homogeneous degree $3d-q$.

\end{lem}

\begin{proof}
If $s_1^{a_1} s_2^{a_2}s_3^{a_3}$     is a monomial entering the polynomial $f$ with a nonzero coefficient, then
$s_2^{a_1} s_1^{a_2}s_3^{d-(a_1+a_2+a_3)}$ is a monomial entering  $g$ with a nonzero coefficient.
Hence $g$ is a polynomial.

The homogeneous degree of $s_2^{a_1} s_1^{a_2} s_3^{d-(a_1+a_2+a_3)}$ equals $d-a_3$. Hence
${\rm deg}\, g\leq d$. 

Since $f$ is not divisible by $s_3$,  there is a monomial $s_1^{a_1} s_2^{a_2}$ entering $f$. Hence the monomial
$s_2^{a_1} s_1^{a_2}s_3^{d-(a_1+a_2)}$ enters $g$ and has homogeneous degree $d$. Hence
${\rm deg}\, g = d$. 

Since ${\rm deg}\, f=d$, there is a monomial $s_1^{a_1} s_2^{a_2}s_3^{a_3}$ entering $f$ such that 
$a_1+a_2+a_3=d$. Then the monomial
$s_2^{a_1} s_1^{a_2}s_3^{d-(a_1+a_2+a_3)}= s_2^{a_1} s_1^{a_2}$
enters $g$ and hence $g$ is not divisible by $s_3$.

If additionally all monomials $s_1^{a_1} s_2^{a_2}s_3^{a_3}$ of $f$ have the property
$a_1+2a_2+3a_3 = q$, then the corresponding monomials 
$s_2^{a_1} s_1^{a_2}s_3^{d-(a_1+a_2+a_3)}$ of $g$ have the property $2a_1+a_2 + 3(d -(a_1+a_2+a_3)) = 3d - q$.
\end{proof}

The polynomial $g$ will be denoted by $\mu (f)$. Clearly
\bean
\label{barbar}
\evo(f) = \evo(g),
\qquad
\mu^2(f) = f .
\eean
The polynomials $f, g$ are called {\it dual}. The {\it bi-degree vectors} of dual polynomials are
\bean
\label{bid}
(d,q), \qquad (d, 3d-q).
\eean
The linear transformation
\bea
 \Z^2\to \Z^2,\qquad (d,q) \mapsto (d, 3d-q),
\eea
is an involution with invariant vector $(2,3)$ and anti-invariant vector $(0,1)$.

\vsk.2>
It is convenient to assign to the polynomial $f$ the {\it  $2\times 2$ degree matrix}
\bean
\label{mat}
M_f  =  \begin{pmatrix}
          d   & d   \\
           q  & 3d-q 
                             \end{pmatrix} ,
\eean                    
whose  columns are the bi-degrees of $f$ and  $\mu(f)$.
Then
\bean
\label{mat}
M_{\mu(f)}  =  \begin{pmatrix}
 d   & d   
 \\
  3d-q  & q 
\end{pmatrix} 
=
 \begin{pmatrix}
          d   & d   \\
           q  & 3d-q 
                             \end{pmatrix}
 \begin{pmatrix}
          0   & 1   \\
           1  &  0 
                             \end{pmatrix} = M_f P,
\eean                    
where $P$ is the permutation matrix.

\subsection{Transformations of triples of polynomials}

Consider a triple $f=(f_1,f_2,f_3)$ of polynomials $f_1,f_2,f_3$
 in $s_1,s_2,s_3$ such that
\begin{enumerate}
\item each  $f_j$ is not divisible by $s_3$,
\item
each $f_j$ is a quasi-homogeneous polynomial with respect to weights (1,2,3),

\item denote by $(d_j,q_j)$ the bi-degree vector of $f_j$,
 then
\bean
\label{ad bi}
(d_1,q_1) +(d_3,q_3) = (d_2,q_2).
\eean

\end{enumerate}
Such a triple $(f_1,f_2,f_3)$ is called an {\it admissible} triple.
\vskip2mm
Equation \eqref{ad bi} is equivalent to the equation
\bean
\label{ad mat}
M_{f_1} + M_{f_3} = M_{f_2}\,.
\eean

\vsk.2>
Define new triples 
\begin{align}
\label{L tra}
Lf = (\mu(f_1),  \mu( f_1)f_2 - s_3^{d_1}f_3, f_2),
\\
\label{R tra}
Rf = (f_2,  f_2\mu( f_3) - s_3^{d_3}f_1, \mu(f_3)).
\end{align}
The transformation $f\mapsto Lf$ is called the {\it left} transformation of an admissible triple $f$, 
because of the new first and second terms of $Lf$ are
on the left from the surviving term $f_2$. Similarly the transformation
 $f\mapsto Rf$ is called the {\it right} transformation, because of
the new second and third  terms of $Rf$ are
on the right from the surviving term $f_2$.

\begin{thm}
Let $f=(f_1,f_2,f_3)$ be an admissible triple of polynomials with bi-degree vectors
$((d_1,q_1), (d_2,q_2), (d_3,q_3))$. Then the triples
$Lf$ and $Rf$ are admissible. The bi-degree vectors of $Lf$ are
\bean
\label{beL}
((d_1,3d_1-q_1), (d_1+d_2,3d_1-q_1+q_2), (d_2,q_2))
\eean
and the bi-degree vectors of $Rf$ are
\bean
\label{beL}
((d_2,q_2), (d_2+d_3,q_2+3d_3-q_3), (d_3,3d_3-q_3)).
\eean
\end{thm}

\proof
Clearly the polynomials 
$\mu( f_1)f_2 - s_3^{d_1}f_3$, $ f_2\mu( f_3) - s_3^{d_3}f_1$ are nonconstant and 
are not divisible by $s_3$.
The homogeneous degrees of $Lf$ are $(d_1,d_1+d_2,d_2)$. This follows from Lemma \ref{lem degg} and admissibility of the triple $f$. For the quasi-homogeneous degrees we have
\bea
{\rm Deg} (\mu( f_1)f_2) = 3d_1- q_1 +q_2= 3d_1+q_3 = {\rm Deg} (s_3^{d_1} f_2),
\eea
by Lemma \ref{lem degg}.
Hence $\mu( f_1)f_2 - s_3^{d_1}f_3$ is a quasi-homogeneous polynomial of quasi-homogeneous
degree $3d_1- q_1 +q_2$. The quasi-homogeneous degree of $\mu( f_1)$ is  $ 3d_1-q_1$.
 This proves the statement for $Lf$. The argument for $Rf$ is similar.
\endproof

\begin{cor}
\label{cor 2x_2}
Let $f=(f_1,f_2,f_3)$ be an admissible triple of polynomials with degree matrices
$(M_1, M_2, M_3)$. Then the degree matrices of 
$Lf$ and $Rf$ are 
\bean
\label{22ad}
(M_1P, \,M_1P+M_2,\, M_2),\qquad
(M_2,\, M_2 + M_3P,\, M_3P),
\eean
where  $P$ is the permutation matrix.
\qed

\end{cor}

\subsection{Reduced polynomial solutions}

Any solution of the $*$-Markov equation \eqref{ME} can be written in the form $(f_1^*, f_2,f_3^*)$, where 
$f_1,f_2,f_3$ are Laurent polynomials. For any ${m_1},m_3\in \Z$, the triple
\bea
((s_3^{m_1}f_1)^*, s_3^{m_1+m_3}f_2, (s_3^{m_3}f_3)^*)
\eea
 is also a solution. 
Given a solution $(f_1^*,f_2,f_3^*)$ there exist unique  $m_1,m_3\in\Z$ such that  
$s_3^{m_1}f_1$, $s_3^{m_3}f_3$ are polynomials and 
each of $s_3^{m_1}f_1$, $s_3^{m_3}f_3$ is not divisible by $s_3$.

\vsk.2>

A solution $(f_1^*, f_2,f_3^*)$ of \eqref{ME} is called a {\it reduced polynomial solution} if each of $f_1, f_2, f_3$ is
a nonconstant polynomial in $s_1,s_2,s_3$  not divisible by $s_3$.
In this case we say that $(f_1^*, f_2,f_3^*)$ is a {\it reduced polynomial presentation of the Markov triple}
$(a,b,c):= \Evo(f_1^*, f_2,f_3^*)$.

\vsk.2>
For example,
\bean
\label{exa sol}
(s_1^*,s_1^2-s_2, s_1^*) ,
\qquad
(s_2^*,  s_2(s_1^2-s_2)-s_3s_1, (s_1^2-s_2)^*)
\eean
are reduced polynomial presentations of the Markov triples $(3,6,3)$ 
and $(3,15,6)$.

\begin{thm}
\label{thm 4.14}

${}$

\begin{enumerate}
\item[(i)]

Let $(a,b,c)$ be a Markov triple,  $0<a < b$, $0< c < b$, $6\leq b$.
Then there exists a \emph{unique} reduced polynomial solution  $(f_1^*,f_2 , f_3^*)\in \GI$, 
such that $\Evo(f_1^*, f_2, f_3^*) = (a,b,c)$. Moreover, for that
reduced polynomial solution $(f_1^*,f_2 , f_3^*)$ the triple $f=(f_1,f_2,f_3)$ is admissible.

\item[(ii)]

Let  $(f_1^*,f_2 , f_3^*)$ be the reduced polynomial presentation of  a Markov triple  $(a, b, c)$ with
$0<a < b$, $0< c < b$, $6\leq b$. Denote $f=(f_1,f_2,f_3)$. Let 
$Lf = (\mu( f_1),  \mu( f_1)f_2 - s_3^{d_1}f_3, f_2)$ and
$Rf = (f_2,  f_2\mu( f_3) - s_3^{d_3}f_1, \mu( f_3))$ be the left and right transformations of $f$.
Then
\bean
\label{lrt}
(\mu( f_1)^*, \, \mu( f_1)f_2 - s_3^{d_1}f_3, \,f_2^*)
\eean
is the reduced polynomial presentation of the Markov triple $(a, ab-c, b)$
and 
\bean
\label{rrt}
(f_2^*, \, f_2\mu( f_3) - s_3^{d_3}f_1,\, \mu( f_3)^*)
\eean
is the reduced polynomial presentation of the Markov triple $(b, bc-a, c)$.

\end{enumerate}

\end{thm}

\proof 

First we prove the existence. The proof is by induction on the distance in the Markov tree from
$(a,b,c)$ to $(3,3,3)$.

Let us find the reduced polynomial presentations
in $\Ga_MI$ for the Markov triples $(3,6,3)$, $(3,15,6)$.
We  transform the initial solution as follows,
\bea
&&
I=(s_1, s_2^*,s_1) \mapsto
(s_1^*, s_1^2-s_2, s_1^*) = (s_2/s_3, s_1^2-s_2, s_1^*)
\\
&&
\phantom{aaa}
\mapsto (s_2, s_1^2-s_2, (s_3s_1)^*) \mapsto
(s_2^*, s_2(s_1^2-s_2)-s_3s_1,(s_1^2-s_2)^*).
\eea
The triples
\bean
\label{rpp}
(s_1^*,s_1^2-s_2, s_1^*) ,
\qquad
(s_2^*,  s_2(s_1^2-s_2)-s_3s_1, (s_1^2-s_2)^*)
\eean
are desired reduced polynomial presentations of $(3,6,3)$ 
and $(3,15,6)$. For example, the polynomials 
$s_2, s_2(s_1^2-s_2)-s_3s_1,s_1^2-s_2$ are quasi-homogeneous of quasi-homogeneous 
degrees $(2, 4, 2)$ with $2+2=4$ as predicted and of homogeneous degrees $(1,3,2)$ with $1+2=3$.
These three polynomials form an admissible triple.

\vsk.2>
Now assume that a Markov triple 
$(a,b,c)$,  $0<a < b$, $0< c < b$, $6\leq b$,
 has a reduced polynomial presentation
$(f_1^*,f_2 , f_3^*)$, where $(f_1,f_2,f_3)$ is an admissible triple.
Then
\bean
\label{pp1}
(f_1^*,f_2 , f_3^*)= (\mu( f_1)/s_3^{d_1}, f_2,f_3^*) 
\mapsto
(\mu( f_1), f_2, (s_3^{d_1}f_3)^*) \mapsto 
(\mu( f_1)^*,  \mu( f_1)f_2-s_3^{d_1}f_3, f_2^*)
\phantom{aaa}
\eean
and
\bean
\label{pp2}
(f_1^*,f_2 , f_3^*)= (f_1^*, f_2, \mu( f_3)/s_3^{d_3}) \mapsto
((s_3^{d_3}f_1)^*, f_2, \mu( f_3)) \mapsto 
(f_2^*,  \mu( f_3)f_2 -s_3^{d_3}f_1, \mu( f_3)^*)
\phantom{aaa}
\eean
are transformations by elements of the Markov group $\Ga_M$.
The triple $(\mu( f_1)^*,  \mu( f_1)f_2-s_3^{d_1}f_3, f_2^*)$ presents the Markov triple $(a, ab-c, b)$, and
the triple $(f_2^*,  \mu( f_3)f_2 -s_3^{d_3}f_1, \mu( f_3)^*)$ presents the Markov triple
$(b, bc-a, c)$.  
These two triples satisfy the requirements of part (ii) of the theorem.

\vsk.2> 
Let us prove the uniqueness. Let $(f_1^*, f_2, f_3^*)$ and
$(h_1^*,h_2, h_3^*)$ be two reduced polynomial presentations of 
a Markov triple $(a,b,c)$ with   $0<a < b$, $0< c < b$, $6\leq b$.
By Theorem \ref{thm stat} $(h_1^*,h_2, h_3^*)$ is obtained from $(f_1^*,f_2, f_3^*)$ by  a transformation
of the form $g_2g_1g_4$, where $g_i\in G_i$. It is clear that a transformation $g_4$ cannot be used because it will destroy the property of $f_1f_2f_3$ to be not divisible by $s_3$. We also cannot use $g_1$ because it will destroy the fact that
 $(f_1^*,f_2, f_3^*)$  represents a positive triple $(a,b,c)$. If the numbers $a,b,c$ are all distinct,
 we cannot use $g_2$.  
If $(a,b,c)=(3,6,3)$, the presentation $(s_1^*,s_1^2-s_2, s_1^*)$ is symmetric with respect 
to the permutation of the first and third coordinates. The theorem is proved.
\endproof

\subsection{$*$-Markov polynomials}

We say that a polynomial $P(s_1,s_2,s_3)$ is a {\it $*$-Markov polynomial} if 
there exists a Markov triple $(a,b,c)$,  $0<a < b$, $0< c < b$, $6\leq b$, with reduced polynomial
presentation $(f_1^*,f_2,f_3^*)\in \Ga_MI$, such that $P=f_2$. 

\vsk.2>
In particular, this means that $P$ is quasi-homogeneous and is not divisible by $s_3$.

\vsk.2>
The polynomial $s_2$ will also be called a {\it $*$-Markov polynomial}.

\vsk.2>
We say that a polynomial $Q(s_1,s_2,s_3)$ is a {\it dual $*$-Markov polynomial} if 
$Q$ is not divisible by $s_3$ and $\mu( Q)$ is a $*$-Markov polynomial. 
\vsk.2>

In particular this means that $Q$ is quasi-homogeneous.

\vsk.2>

For example, $s_1^2-s_2$, $s_2(s_1^2-s_2)-s_3s_1$ are $*$-Markov polynomials, 
since they appear as the middle terms in the reduced polynomial presentations in 
\eqref{exa sol} and $s_2^2-s_1s_3$,  $s_1(s_2^2-s_1s_3) - s_3s_2$ are
the corresponding dual $*$-Markov polynomials.

\begin{cor}
\label{cor coord}
Let $(a,b,c)$ be a Markov triple with $0<a < b$, $0< c < b$, $6\leq b$.
Let $(f_1^*,f_2,f_3^*)\in \Ga_MI$ be the reduced polynomial presentation of 
$(a,b,c)$. Then each of $f_1, f_3$ is either a $*$-Markov polynomial
or a dual $*$-Markov polynomial.
Moreover, if $(g_1,g_2,g_3) \in \Ga_MI$ is any presentation of $(a,b,c)$, then
\bean
\label{any g}
g_1 = s_3^{k_1} \mu(f_1),
\qquad
g_2 = s_3^{k_2} f_2,
\qquad
g_3 = s_3^{k_3} \mu(f_3),
\eean
for some $k_1,k_2,k_3\in\Z$,
and hence each of $g_1,g_3$ is either a $*$-Markov polynomial
or a dual $*$-Markov polynomial multiplied by a power of $s_3$, and
$g_2$ is a $*$-Markov polynomial multiplied by a power of $s_3$.

\end{cor}

\begin{proof}
The first statement follows from Theorem \ref{thm 4.14} and the second statement follows from
Theorem \ref{thm stat}.
\end{proof}

\begin{rem}
Consider the three  versions of the $*$-Markov equation,
\bean
\label{1e}
&&
aa^*+bb^*+cc^*-abc = \frac{3s_1s_2-s_1^3}{s_3},
\\
\label{2e}
&&
aa^*+bb^*+cc^*-ab^*c = \frac{3s_1s_2-s_1^3}{s_3},
\\
\label{3e}
&&
aa^*+bb^*+cc^*-a^*bc^* = \frac{3s_1s_2-s_1^3}{s_3}.
\eean
The first of them is the $*$-Markov equation \eqref{ME}, the second  was considered in the introduction,
see \eqref{mme} and \cite{CV}. The third  is a new one. 
All of the equations are obtained  one from another
by an obvious change of variables.
For example, the third equation is obtained from the first  by the change $(a,b,c) \to (a^*,b,c^*)$.
That is, if $(f_1^*,f_2,f_3^*)$ is a solution of the $*$-Markov equation \eqref{1e}, then
$(f_1,f_2,f_3)$ is a solution of equation \eqref{3e}. 

Hence, by Theorem \ref{thm 4.14},
for any Markov triple $(a,b,c)$,  $0<a < b$, $0< c < b$, $6\leq b$, there exists a polynomial solution
$(f_1,f_2,f_3)$  of equation \eqref{3e},
 such that $\Evo(f_1,f_2,f_3) = (a,b,c)$.

\end{rem}

\section{Decorated planar binary trees}
\label{sec 7}

\subsection{Sets with involution and transformations}
\label{sec sets}

A {\it set with involution and transformations} is a set $S$   with an involution
$\tau:S\to S$, $\tau^2=\id_S$, a subset $T\subset S\times  S\times S$ with a marked point 
$t^0=(t^0_1,t^0_2,t^0_3)\in T$
and two maps
\bean
\label{lr def}
&&
L : T\to T, \qquad (t_1,t_2,t_3) \mapsto (\tau (t_1), L_2(t_1,t_2,t_3), t_2),
\\
\notag
&&
R:T\to T,\qquad
 (t_1,t_2,t_3) \mapsto (t_2, R_2(t_1,t_2,t_3), \tau(t_3)),
\eean
where $L_2, R_2 : T\to S$ are some functions.

\vsk.2>
A {\it morphism} $\phi : (S,T, t^0,\tau, L, R)\to (S',T',t^{0'}, \tau', L', R')$ is a map $S\to S'$, which
commutes with involutions and  induces a map
$(T, t^0)\to (T', t^{0'})$ commuting with transformations.

\vsk.2>
Here are  examples.

\subsubsection{}
\label{exa1}

Let $ S$ be the set of all polynomials in $\Z[s_1,s_2,s_3]$ not divisible by $s_3$
and $T\subset S^3$ the subset of all admissible triples. 
Let
\bean
\label{t01}
t^0 
&=&
(s_2,  s_2(s_1^2-s_2)-s_3s_1, s_1^2-s_2),
\\
\label{ex_1t}
\tau 
&:&
 S\to S, \qquad  f\mapsto \mu( f),
\\
\label{ex_1tL}
L 
&:&
T\to T,\qquad (f_1,f_2,f_3) \mapsto (\mu( f_1),  \mu( f_1)f_2 - s_3^{d_1}f_3, f_2),
\\
\label{ex_1tR}
R
&:&
T\to T,\qquad (f_1,f_2,f_3) \mapsto  (f_2,  f_2\mu( f_3) - s_3^{d_3}f_1, \mu( f_3)),
\eean
where $\mu$ is defined in Section \ref{sec def f}.

\subsubsection{}
\label{exa2}

Let $ S=\C^2$ 
and $T\subset \C^2\times  \C^2\times \C^2$ 
the subset of all triples of vectors $w_1,w_2,w_3$ such that $w_1+w_3=w_2$.
Let
\bean
\label{t02}
t^{0} 
&=&
((1,2),  (3,4), (2,2)),
\\
\label{ex_2t}
\tau 
&:&
 \C^2\to \C^2, \qquad  (d,q)\mapsto (d, 3d-q),
\\
L 
&:&
T\to T,\qquad (w_1,w_2,w_3) \mapsto (\tau(w_1),  \tau(w_1)+w_2, w_2),
\\
R
&:&
T\to T,\qquad (w_1,w_2,w_3) \mapsto (w_2,  w_2+\tau(w_3), \tau(w_3)).
\eean

\subsubsection{}
\label{exa2'}

Let $ S$ be the set $\on{Mat}(2,\C)$ of all $2\times 2$-matrices with complex entries
and $T\subset \on{Mat}(2,\C)^3$ 
the subset of all triples of matrices $M_1,M_2,M_3$ such that $M_1+M_3=M_2$.
Let
\bean
\label{t02'}
t^{0} 
&=&
\Big(
\begin{pmatrix}
 1   & 1   
 \\
 2  &  1
\end{pmatrix} ,
\
\begin{pmatrix}
 3   & 3   
 \\
 4  &  5
\end{pmatrix} ,
\
\begin{pmatrix}
 2   & 2   
 \\
 2  &  4
\end{pmatrix}\Big),
\\
\label{ex_2t'}
\tau 
&:&
  \on{Mat}(2,\C) \to  \on{Mat}(2,\C), \qquad  M\mapsto MP,
\\
L 
&:&
T\to T,\qquad 
(M_1,M_2,M_3) \mapsto (M_1P,  M_1P+M_2, M_2),
\\
R
&:&
T\to T,\qquad 
(M_1,M_2,M_3) \mapsto (M_2, M_2+ M_3P, M_3P).
\eean

\subsubsection{}
\label{exa2''}

Let $ S=\C$ 
and $T\subset \C^3$ 
the subset of all triples  $(w_1,w_2,w_3)$ such that $w_1+w_3=w_2$.
Let
\bean
\label{t02''}
t^{0} 
&=&
(1,-1,-2),
\\
\label{ex_2t''}
\tau 
&:&
\C\to \C,\qquad w\mapsto-w,
\\
\label{Ll}
L 
&:&
T\to T,\qquad 
(w_1,w_2,w_3)  \mapsto (-w_1,-w_1+w_2,w_2),
\\
\label{Rr}
R
&:&
T\to T,\qquad 
(w_1,w_2,w_3) \mapsto 
(w_2,w_2-w_3,-w_3) .
\eean

\subsubsection{}
\label{exa3}

Let $ S=\C$  and $T = \C^3$.
Let
\bean
\label{t03}
t^{0} 
&=&
(3,15,6),
\\
\label{ex_3t}
\tau 
&=&
\id_\C
\\
L
&:&
T\to T,\qquad (a,b,c) \mapsto (a,ab-c,b),
\\
R
&:&
T\to T,\qquad (a,b,c) \mapsto (b,bc-a,c).
\eean

\subsubsection{}
\label{exa4}

Let $ S=\C$  and $T = \C^3$.
Let
\bean
\label{t04}
t^{0} 
&=&
(1,3,2),
\\
\label{ex_4t}
\tau 
&=&
\id_\C
\\
L
&:&
T\to T,\qquad (a,b,c) \mapsto (a,a+b,b),
\\
R
&:&
T\to T,\qquad (a,b,c) \mapsto (b,b+c,c).
\eean

\subsubsection{De-quantization}
\label{exa5} Let $S$ be the set with involution and transformations
in Example \ref{exa1} and $S'$ the set with involution and transformations
in Example \ref{exa2}.
The map 
\bea
\phi : S\to S', \qquad f\mapsto (\deg(f), \on{Deg} (f)),
\eea
defines a morphism of the sets with involution and transformations.

\vsk.2>

We may think of  that $\phi : S\to S'$ is a {\it de-quantization} of the set $S$ with 
involution and transformations  as explained in Section \ref{1.4.3}. Namely,
Let $s_1= c_1e^{\alpha + \beta}$,  $s_2=c_2e^{\alpha + 2\beta}$,
$s_3=c_3e^{\alpha + 3\beta}$, 
where $\al, \beta$ are real parameters which tend to $ +\infty$ and
$c_1,c_2,c_3$ are fixed generic real numbers.   
If $f(s_1,s_2,s_3)$  is a quasi-homogeneous polynomial of bi-degree $(d,q)$, then
$\ln f(c_1e^{\alpha + \beta}, c_1e^{\alpha + \beta},c_1e^{\alpha + \beta})$
has {\it leading term}  $d\al + q\beta$ independent of the choice of $c_1,c_2,c_3$, which may be considered as a vector
$(d,q)$.

\vsk.2>
Taking the leading terms of all quasi-homogeneous polynomials in formulas of Example \ref{exa1}
 we obtain the 2-vectors in formulas of Example \ref{exa2}.
  For instance, the triple of leading terms
of the triple $(s_2,  s_2(s_1^2-s_2)-s_3s_1, s_1^2-s_2)$ is the triple
$(\al+2\beta, \, 3\al+4\beta,\, 2\al+2\beta)$, cf. \eqref{t01} and \eqref{t02}.

\subsubsection{}
\label{exa6}

Let $S$ be the set  with involution and transformations
 in Example \ref{exa1} and $S'$ the set
 with involution and transformations in Example \ref{exa2'}.
The map 
\bea
\phi : S\to S', \qquad f\mapsto M_f,
\eea
where $M_f$ see in \eqref{mat},
defines a morphism of the sets with involution and transformations.

\subsubsection{}
\label{exa6'}

Let $S$ be the set  with involution and transformations
 in Example \ref{exa2'} and $S'$ the set
 with involution and transformations in Example \ref{exa2''}.
The map 
\bea
\phi : S\to S', \qquad \begin{pmatrix}
 a_{11}   & a_{12}   
 \\
 a_{21}  &  a_{22}
\end{pmatrix} \mapsto a_{21}-a_{22},
\eea
defines a morphism of the sets with involution and transformations.

\subsubsection{}
\label{exa6''}

Let $S$ be the set  with involution and transformations
 in Example \ref{exa1} and $S'$ the set
 with involution and transformations in Example \ref{exa3}.
The map 
\bea
\phi : S\to S', \qquad f\mapsto \evo(f),
\eea
defines a morphism of the sets with involution and transformations.

\subsubsection{}
\label{exa7}
Let $S$ be the set  with involution and transformations
 in Example \ref{exa1} and $S'$ the set
 with involution and transformations in Example \ref{exa4}.
The map 
\bea
\phi : S\to S', \qquad f\mapsto \deg (f),
\eea
defines a morphism of the sets with involution and transformations.

\subsection{Planar binary tree}
Consider the {\it oriented binary planar tree}, growing from floor,
 and the domains of its complement, see Figure \ref{fig1&2}.
The boundary of any domain of the complement has a distinguished vertex 
with shortest number of steps to the root along the tree.

\vsk.2> 
There are two initial domains, which touch the floor. In Figure \ref{fig1&2} they are 
$D_1$ and $D_3$. The root of the tree is the distinguished
vertex of the two initial  domains.

\begin{figure}
\centering
\def\svgscale{.7}
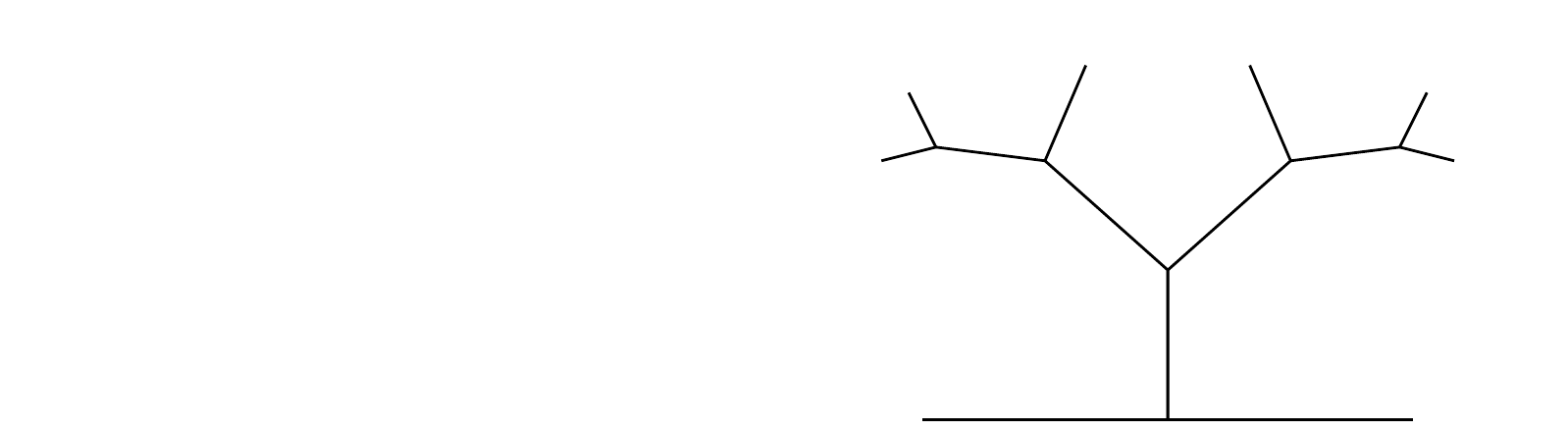
\caption{}
\label{fig1&2}
\end{figure}

\vsk.2>
The boundary  of the left initial domain  $D_1$  consists of the left half-floor
and the infinite sequence of edges $l_1, l_2, \dots$, see Figure \ref{fig1&2}.
In the notation
 $l_k$, the letter $l$ means that the domain $D_1$ is on the left from the edge,
 when we move from the root to this edge along the tree,
and  $k$
means that it is the $k$-th edge counted from the root of the tree.

The boundary  of the right initial domain  $D_3$  consists of the right half-floor
and the infinite sequence of edges $r_1, r_2, \dots$, see Figure \ref{fig1&2}.

 The boundary of any other domain  consists of two infinite sequences
of edges $r_1, r_2, \dots$ and $l_1, l_2, \dots$, see Figure \ref{fig3&4}.

\begin{figure}
\centering
\def\svgscale{1}
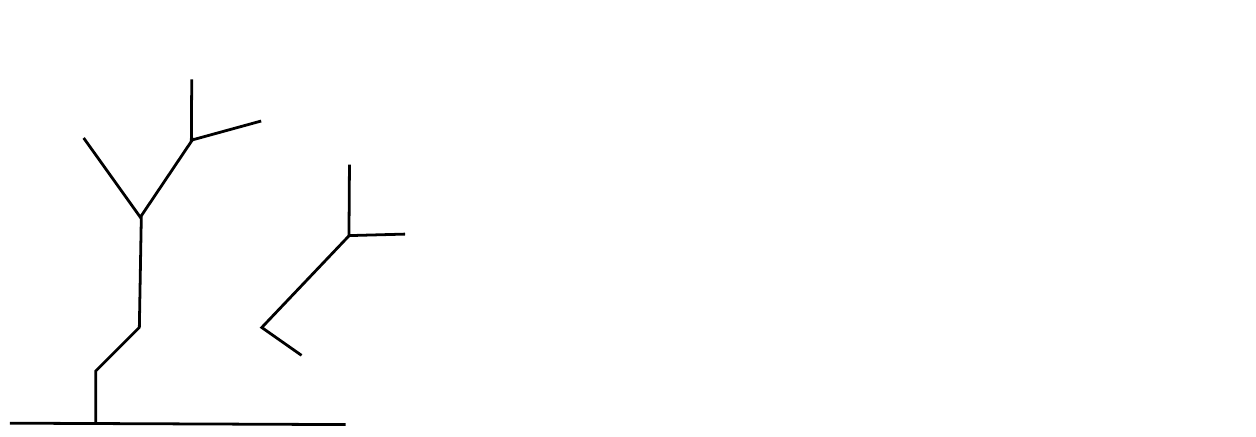
\caption{}
\label{fig3&4}
\end{figure}

 \vsk.2>
Every edge of the tree gets two labels, a label $l_a$ from the left and a label 
$r_b$ from the right. We denote such an edge with labels by $l_a|r_b$.
The first edge of the tree has labels $l_1|r_1$. All other edges of the tree have labels
\bea
l_1|r_k \qquad \on{or} \qquad l_k|r_1\qquad \on{with} \qquad k>1,
\eea
see Figure \ref{fig3&4}.

\subsection{Decorations} 

Let $(S,T, t^0,\tau, L, R)$ be a set with 
involution and transformations.  First we assign an element of the set $T$ to every vertex of the
planar binary tree different from the root vertex, and then assign an element of the 
set $S$ to every domain of the complement. Thus the decoration procedure consists of two step.

\vsk.2>
Denote by $v_1$ the vertex of the tree  surrounded by
the domains $D_1,D_2,D_3$ in Figure \ref{fig1&2}.
We assign  to the vertex $v_1$ the marked triple $t^0=(t^0_1, t^0_2, t^0_3)$. 

 Let $v_2$ be any other vertex of the tree different from the root.
Let $p$ be the path connecting $v_1$ and $v_2$ in the tree. The path is a sequence of turns
$p_np_{n-1}\dots p_2p_1$, where $p_j$ is the turn to the left or right on the way from $v_1$ to $v_2$.
We assign to $v_2$ the element $t\in T$ obtained from $t^0$ by the
 application of the sequence of transformations $L$ and $R$, where we apply 
$L$ if $p_j$ is the turn to the left and apply $R$ if $p_j$
is the turn  to the right. For example, the element $LRt^0$ is 
assigned to the vertex  $v_2$ in Figure \ref{fig5&6}.
\vsk.2>

This is the end of the first step of the decoration.

\begin{figure}
\centering 
\def\svgscale{1}
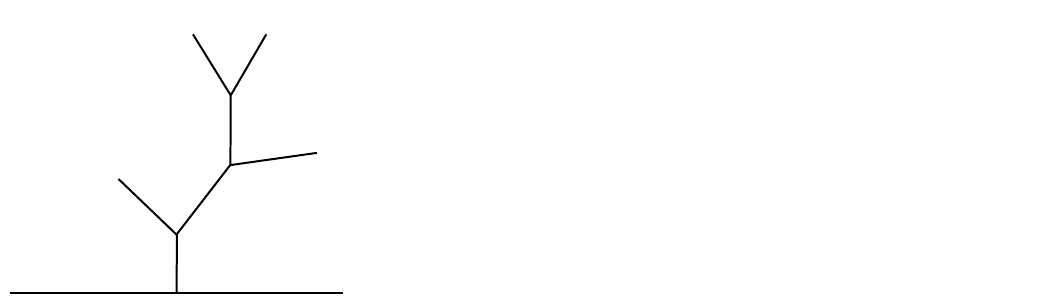
\caption{}
\label{fig5&6}
\end{figure}

\vsk.2>

At the second step we
 assign to the initial  domains $D_1, D_3$ in Figure \ref{fig1&2}  
 the elements $t^0_1,$ $ t^0_3$, respectively, where $t^0_1$, $t^0_3$
 are the first and third coordinates  of the initial triple $(t^0_1, t^0_2, t^0_3)$.

\vsk.2>
Let $C$ be any domain of the complement different from $D_1$, $D_3$.
Let $v$ be the distinguished vertex of the domain $C$,  and $t=(t_1,t_2,t_3)$ the element of $T$ assigned to $v$.
We assign to  $C$  the element $t_2$.

\vsk.2>
For example we assign the element $t^0_2$ to the domain $D_2$ in Figure \ref{fig1&2}.

\vsk.2>
This is the end of the decoration procedure.

\vsk.2>
The decoration associated with $(S,T, t^0,\tau, L, R)$ is functorial with respect to morphisms of sets with 
involution and transformations.

\vsk.2>
Let us describe how to recover the element of $T$ assigned to a vertex from the elements of $S$ assigned to the domains
of the complement.

\begin{thm}
\label{thm StoT}
Let $v$ be a vertex surrounded by domains $C_1,C_2,C_3$ as in Figure \ref{fig5&6}. Let $t_1,t_2,t_3$ be elements of $S$ assigned to
$C_1,C_2,C_3$, respectively, at the second step of the decoration. Let the edge entering the vertex $v$ has labels $l_a|r_b$. 
Then the element  $(\tau^{a-1}(t_1), t_2,\tau^{b-1}(t_3))$ is an element of the set\, $T$ and that element 
was assigned to $v$ at the first step of the decoration.

\end{thm}

\begin{proof} 
The proof is by induction on the distance from $v$ to the root.
\end{proof}

\subsection{Examples}
\subsubsection{}

Let $(S,T, t^0,\tau, L, R)$ be the set of Example \ref{exa1}. Then the domains of the complement to the binary tree are labeled by
$*$-Markov polynomials. The resulting decorated tree is  called the {\it $*$-Markov polynomial tree}, see Figure \ref{fig8}.

\begin{figure}[H]
\centering
\def\svgscale{1}
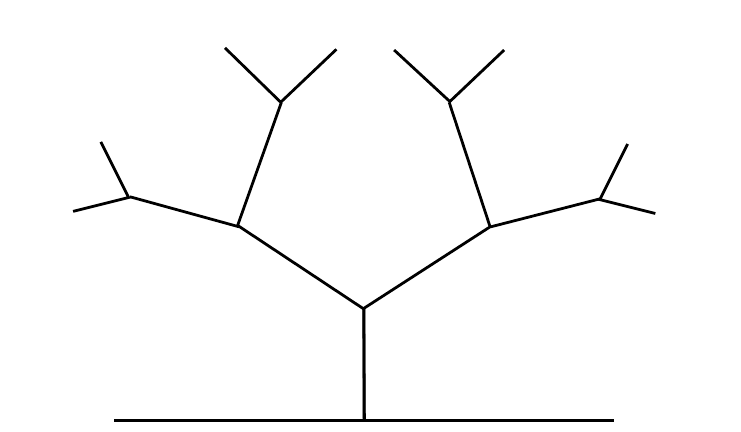
\caption{}
\label{fig8}
\end{figure}

The polynomials $A_i(\bm s)$ are given by the formulas
\begin{align*}
A_1(\bm s)&=s_2,\\
 A_2(\bm s)&=s_2 (s_1^2 - s_2) - s_1 s_3,\\
A_3(\bm s)&=s_1^2 - s_2,\\
 A_4(\bm s)&=s_1^3 s_2 - s_1 s_2^2 - 2 s_1^2 s_3 + s_2 s_3,\\
A_5(\bm s)&=s_1^2 s_2^3 - s_2^4 - s_1^3 s_2 s_3 + s_1^2 s_3^2 - s_2 s_3^2,\\
A_6(\bm s)&=s_1^3 s_2^2 - s_1 s_2^3 - 3 s_1^2 s_2 s_3 + 2 s_2^2 s_3 + s_1 s_3^2,\\
 A_7(\bm s)&=s_1^4 s_2^3 - 2 s_1^2 s_2^4 + s_2^5 - s_1^5 s_2 s_3 + s_1^3 s_2^2 s_3 + 
 s_1^4 s_3^2 - 3 s_1^2 s_2 s_3^2 + 2 s_2^2 s_3^2 + s_1 s_3^3,\\
A_8(\bm s)&=s_1^4 s_2^3 - s_1^2 s_2^4 - s_1^5 s_2 s_3 - 2 s_1^3 s_2^2 s_3 + 2 s_1 s_2^3 s_3 + 
 2 s_1^4 s_3^2 + s_1^2 s_2 s_3^2 - s_2^2 s_3^2 - s_1 s_3^3,\\
     A_9(\bm s)&=s_1^3 s_2^5 - s_1 s_2^6 - 2 s_1^4 s_2^3 s_3 + s_2^5 s_3 + s_1^5 s_2 s_3^2 + 
 2 s_1^3 s_2^2 s_3^2 - s_1 s_2^3 s_3^2 - s_1^4 s_3^3 + s_1 s_3^4.
\end{align*}

Let $v$ be any vertex. It enters the boundary of three domains, which we denote by
 $C_1,C_2,C_3$ as in Figure \ref{fig5&6}. Let $f_1,f_2,f_3$ be 
the $*$-Markov polynomials, assigned to the domains
$C_1,C_2,C_3$, respectively at the second step of the decoration. Let the edge entering the vertex $v$ have labels $l_a|r_b$. 
Then the triple of polynomials  $(\tau^{a-1}(f_1), f_2,\tau^{b-1}(f_3))$ is  assigned to $v$ at the first step
 of decoration, and the triple
of polynomials
\bea
((\tau^{a-1}(f_1))^*, f_2,(\tau^{b-1}(f_3))^*)
\eea
is a reduced polynomial solution of the $*$-Markov equation \eqref{ME}.

\subsubsection{}
 Let $(S,T, t^0,\tau, L, R)$ be the set of Example \ref{exa2}. Then the domains of 
 the complement to the binary tree are labeled by
2-vectors with positive integer coordinates. The resulting decorated tree is called the {\it 2-vector  tree}, see Figure \ref{tree1}.

\subsubsection{}
 Let $(S,T, t^0,\tau, L, R)$ be the set of Example \ref{exa2'}. Then the domains of 
 the complement to the binary tree are labeled by
$2\times 2$-matrices with positive integer coordinates. The resulting decorated tree is called the {\it matrix  tree}, see Figure \ref{matree}.

\subsubsection{}
 Let $(S,T, t^0,\tau, L, R)$ be the set of Example \ref{exa2''}. Then the domains of 
 the complement to the binary tree are labeled by integers. 
 The resulting decorated tree is called the {\it deviation  tree}, see Figure \ref{tree1}.

\subsubsection{}
 Let $(S,T, t^0,\tau, L, R)$ be the set of Example \ref{exa3}. Then the domains of 
 the complement to the binary tree are labeled by Markov numbers. The resulting decorated tree is  called
 the {\it Markov tree}, see the left picture in Figure \ref{tree2}.

 \subsubsection{}
 Let $(S,T, t^0,\tau, L, R)$ be the set of Example \ref{exa4}. Then the domains of 
 the complement to the binary tree are labeled by positive integers. 
 The resulting decorated tree is called the {\it Euclid tree}, see the right picture in Figure \ref{tree2}.

\begin{figure}
\centering
\def\svgscale{1}
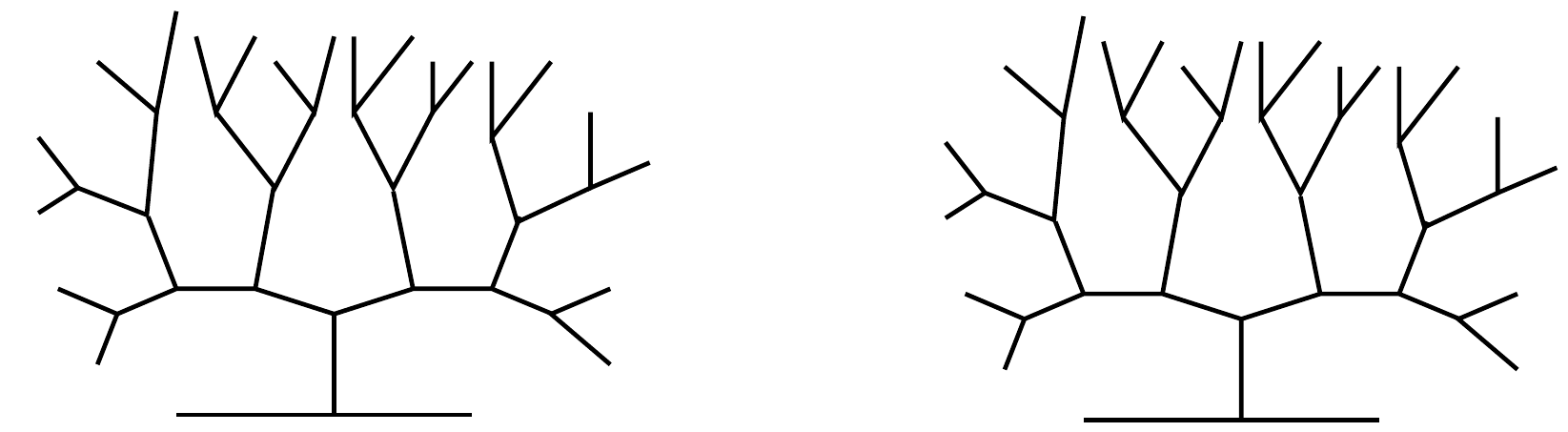
\caption{}
\label{tree1}
\end{figure}

\begin{figure}
\centering
\def\svgscale{.8}
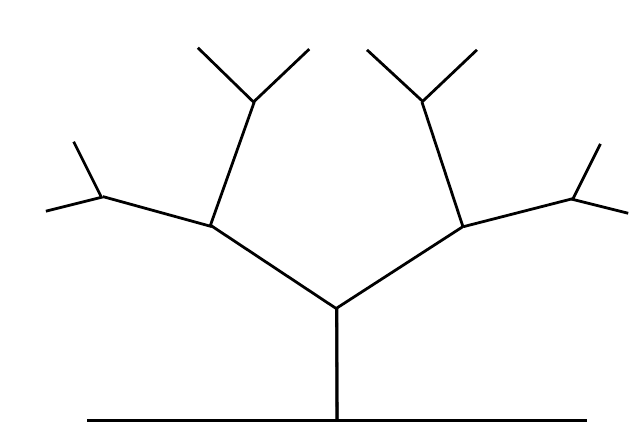
\caption{}
\label{matree}
\end{figure}

\begin{figure}
\centering
\def\svgscale{1}
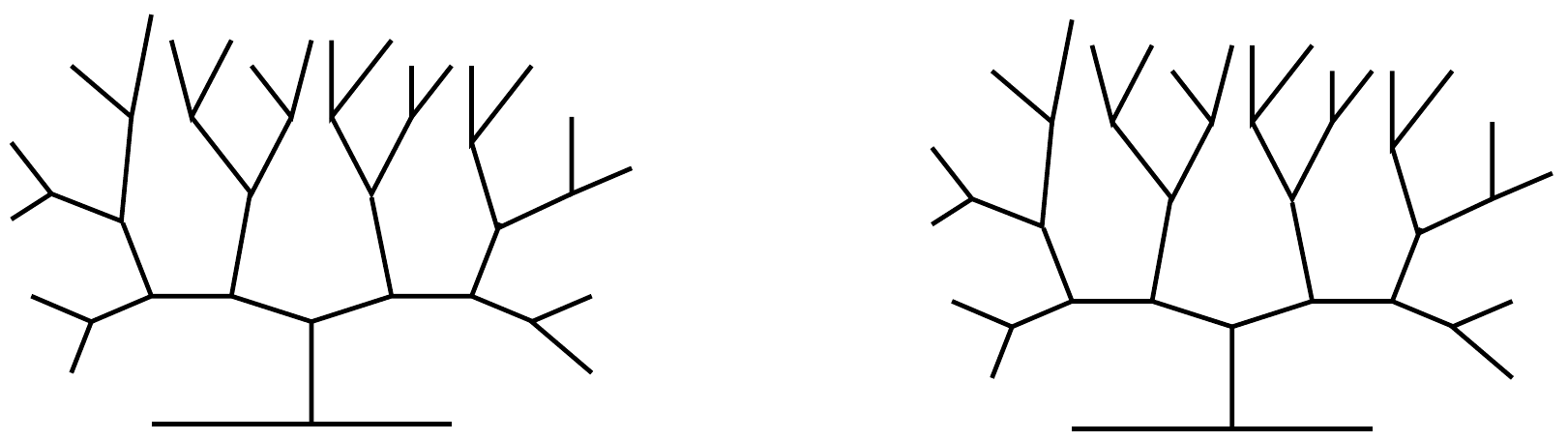
\caption{}
\label{tree2}
\end{figure}

\vsk.2>
 The decorated trees in Figures \ref{tree1}-\ref{tree2} can be obtained 
 from the $*$-Markov polynomial tree in Figure \ref{fig8}. Namely the 2-vector tree is obtained by taking
  the bi-degree vectors of $*$-Markov polynomials;
   the matrix tree is obtained by taking
  the degree matrices  of $*$-Markov polynomials;
   the deviation tree is obtained by assigning to a $*$-Markov polynomial with bi-degree $(d,q)$
     the number 
\bea
q-(3d-q)=2q-3d\,;
\eea
     the Markov tree is obtained by applying the evaluation map $\evo$; 
   the Euclid tree is obtained by taking the homogeneous degrees of  $*$-Markov polynomials.
 
 \subsection{Do asymptotics exist?} 
 \label{sec OP}

 Having a decorated tree it  would 
be interesting to study asymptotics of the triples assigned to vertices along the 
infinite paths in the tree going from  root to infinity. 
In \cite{SoV19, SpV17, SpV18} the Markov and Euclid trees were considered. For any such a
path the Lyapunov exponent was defined. The Lyapunov function on the space of paths was studied.
Relations with hyperbolic dynamics were established.

\vsk.2>
The interrelations of the triples assigned to vertices of the Markov and Euclid  trees 
were analyzed in \cite{Za82} to study the growth of Markov numbers 
ordered in the increasing order. More precisely, if $(u,v,w)$ is a Euclid triple with $u+w=v$,
then the triple
\bean
\label{cosh f}
a=2\on{cosh}\,u,
\qquad 
b=2\on{cosh}\,v,
\qquad
c=2\on{cosh}\,w
\eean
is a solution of the modification of the Markov equation 
\bean
\label{mMS}
a^2+b^2+c^2-abc = 4,
\eean
considered by Mordell \cite{Mo53}. This observation was used in \cite{Za82} to evaluate 
asymptotics of Markov numbers in terms of asymptotics of Euclid numbers, see \cite{SpV17}.

\vsk.2>
Combining these remarks we observe a full circle of relations. We started 
with Markov triples and upgraded them to triples of $*$-Markov polynomials;
taking the homogeneous degrees of  $*$-Markov polynomials we obtained the Euclid triples;
formulas \eqref{cosh f} send us to triples solving  the modified Markov
equation \eqref{mMS}; and the triples solving  equation \eqref{mMS}
approximate the true Markov triples. 
This circle of relations is a combination of ``quantizations'' and ``de-quatizations''.

\subsection{Values of $2q-3d$}

\begin{thm}
\label{thm rebi}
Let  $P$ be a $*$-Markov polynomial of bi-degree $(d,q)$. Then
$|2q-3d|=1$ if $d$ is odd and $|2q-3d|=2$ if $d$ is even. Moreover,
 the  only triples of integers attached to vertices of the deviation tree are the elements of the set
\bea
T^0=\big\{(1,-1,-2), \ (-1,1,2), \ (-2,-1,1),\ (2,1,-1),
\
(1,2,1), \ (-1,-2,-1)\big\}\,.
\eea

\end{thm}

\proof 
The statement is true for the triple $t^0=(1,-1,-2)$ assigned to the first vertex  of the deviation tree,
see \eqref{t02''}.
It is easy to check that the set
$T^0$ is preserved by the $L$ and $R$ transformations in formulas \eqref{Ll} and \eqref{Rr}.
This proves the theorem.
\endproof

\begin{cor}
\label{cor ldq}
We have
\bean
\label{ldq}
 q = \frac{3d}2 + \mc O (1)\quad\on{as}\quad d\to\infty.
\eean
\qed
\end{cor}

\subsection{Newton polygons}
\label{sec NPP}

Let $P$ be a $*$-Markov polynomial  of bi-degree $(d,q)$.
Let $N_P$ be the {\it Newton} polytope of $P$.
Recall that for  each monomial $s_1^{a_1}s_2^{a_2}s_3^{a_3}$,
 entering $P$ with nonzero coefficient, we
mark the point $(a_1,a_2,a_3)\in \R^3$, and  the Newton polytope is the convex hull of  marked points.

\vsk.2>
Since $P$ is a quasi-homogeneous polynomial of degree $q$, the Newton polytope is  a two-dimensional
convex polygon, lying inside  the {\it bounding polygon} $N_{d,q}$\,,
\bean
\label{qua}
N_{d,q} = \{ (a_1,a_2,a_3)\in \R^3\ |\ a_1+2a_2+3a_3=q; \  0 \leq a_1,a_2,a_3\leq d\}.
\eean
We divide all coordinates by $d$ and obtain the {\it normalized Newton} polygon $\bar N_P$ inside
the {\it normalized bounding} polygon  $\bar N_{d,q}$,
\bean
\label{nqua}
\bar N_{d,q} = \{ (a_1,a_2,a_3)\in \R^3\ |\ a_1+2a_2+3a_3=q/d; \  0 \leq a_1,a_2,a_3\leq 1\}.
\eean
It is convenient to project  the  polygons $\bar N_P$ and
 $\bar N_{d,q}$ along  the $a_3$-axis to $\R^2$ with coordinates $a_1,a_2$
and obtain the {\it projected normalized Newton polygon} $\tilde N_P$ inside the
{\it projected normalized bounding polygon} $\tilde N_{d,q}$.

\subsection{Limit $d\to \infty$}
\label{sec dti}

The Euclid tree shows the distribution of the
homogeneous degrees of $*$-Markov polynomials.
The homogeneous degree $d$ tends to infinity along the paths of the planar binary tree 
from  root to infinity.
Along these paths we have $q\to 3d/2$. 
In this limit the 
 normalized bounding polygon $\bar N_{d,q}$ turns into the {\it  quadrilateral} $\bar N_\infty$,
\bean
\label{iqua}
\bar N_\infty = \{ (a_1,a_2,a_3)\in \R^3\ |\ a_1+2a_2+3a_3=3/2; \  0 \leq a_1,a_2,a_3\leq 1\},
\eean
and the projected normalized bounding polygon $\tilde N_{d,q}$ turns into the {\it projected quadrilateral }
$\tilde N_\infty$,  the convex quadrilateral with vertices
$(0,0), (3/4,0), (1/2, 1/2), (0,3/4)$. See the pictures 
of $\bar N_\infty$ and $\tilde N_\infty$ in Figure \ref{quad}.
\begin{figure}
\centering
\def\svgscale{1}
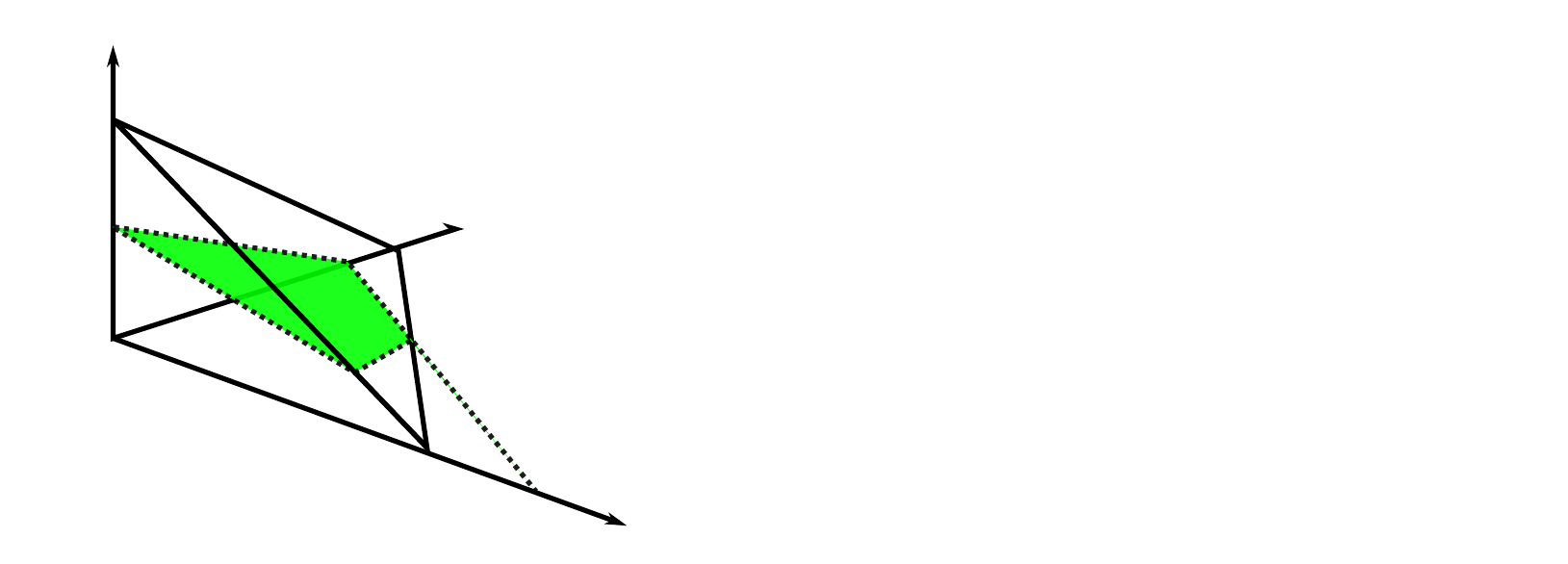
\caption{}
\label{quad}
\end{figure}

\medskip
\noindent
{\bf Question.}  Could it be that for any infinite path  from root to infinity, the projected normalized 
Newton polygon $\tilde N_P$ tends in an appropriate sense to a limiting shape inside the projected quadrilateral $\tilde N_\infty$?
\medskip

We show that this is indeed  so in the two examples of the left and right paths of the planar binary tree, 
which are related to the $*$-Fibonacci and  $*$-Pell polynomials discussed in Sections \ref{sec Fib} and \ref{sec Pell}.
In the first case the limiting shape is the interval with vertices $(0,0)$, $(1/2,1/2)$, see Section \ref{sec NP}.
 In the second case the limiting shape is 
the whole  projected quadrilateral $\tilde N_\infty$, see Section \ref{sec NPe}.

\subsubsection{}  Any $*$-Markov polynomial $P$ has a monomial  
of the form $s_1s_3^{a_3}$ or $s_2s_3^{a_3}$ entering $P$ with a nonzero coefficient and has no monomials of the form
$s_3^{a_3}$. 
This  easily follows by induction. Hence for any infinite path from root to infinity the point $(0,0)$ is a limiting point of the 
  projected normalized  Newton polygon $\tilde N_P$.  

\subsubsection{}
Elementary computer experiments show that the expected limiting shape of the polygon $\tilde N_P$
 along an infinite path is a 6-gon   like in Figure \ref{quad2},  with width monotonically increasing from 0,
 for the $*$-Fibonacci polynomials, to the maximal value, for the $*$-Pell polynomials, when the path
 changes from the leftmost 
 to the rightmost. This $6$-gon is  symmetric with respect to the diagonal
 $a_1=a_2$, and hence its width completely determines the 6-gon. It looks like the speed of convergence to the limiting shape
 increases if the path has many changes of direction from left to right and back.

\begin{figure}[H]
\centering
\def\svgscale{1}
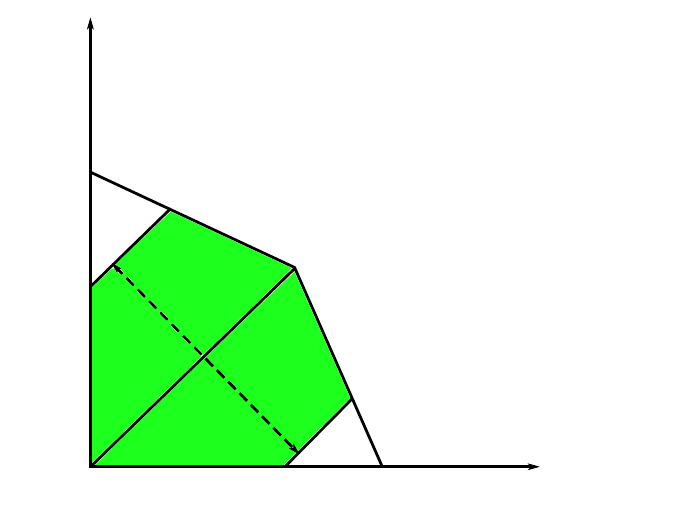
\caption{}
\label{quad2}
\end{figure}

\subsection{Planar binary tree decorated by convex sets}
\label{sec CS}

The study of the limiting shapes of Newton polygons is closely
related to the following decorated planar binary tree.
\vsk.2>

\subsubsection{}
Consider $\R^2$ with coordinates $a_1,a_2$. For a subset $A\subset \R^2$ we denote by 
$\on{conv}[A]$ the convex hull of $A$. We denote by $\mu(A)$ the subset  $A$  reflected with respect to the diagonal
$a_1=a_2$. For $d\in\R_{>0}$, we denote 
\bea
dA = \{ dv\ |\ v\in A\}.
\eea
For subsets $A, B\subset \R^2$ we denote by $A+B$ the Minkowski sum,
\bea
A+B = \{a +b\ |\ a\in A,\ b\in B\}.
\eea

\subsubsection{}

Define a set $(S,T, t^0,\tau, L, R)$ with involution and transformations.  Let $S$ be the set of all pairs
$(A,d)$, where $A$ is a convex subset of $\R^2$
 and $d$ a positive number.
Define the involution $\tau$ by the formula
\bea
\tau : S \to S, \qquad (A, d) \mapsto (\mu(A), d).
\eea
Let $T \subset S^3$ be the subset of all triples
$((A_1,d_1), (A_2,d_2), (A_3,d_3))$ such that $d_2=d_1+d_3$.
We fix the initial triple 
\bean
\label{t0co}
t^0=((A_1^0,1), (A_2^0,3), (A_3^0, 2))\in T,
\eean
where  $A^0_1$ is the point $(0,1)$, $A^0_3$ is the interval with vertices $(1,0), (0,1/2)$,
and $A^0_2$ is the triangle  with vertices $(1/3,0),(2/3,1/3),(0,2/3)$.

Define the left and right transformations by the formulas
\bean
\label{ex_1tLc}
L 
&:&
T\to T,\ ((A_1,d_1), (A_2, d_2), (A_3, d_3)) \mapsto ((\mu( A_1), d_1), (L_2, d_1+d_2), (A_2,d_2)),
\phantom{aaa}
\\
\label{ex_1tRc}
R
&:&
T\to T,\ ((A_1,d_1), (A_2,d_2), (A_3,d_3)) \mapsto  ((A_2,d_2),  (R_2, d_2+d_3), (\mu( A_3),d_3)),
\phantom{aaa}
\eean
where
\bean
\label{LRc}
L_2 &=& \on{conv}\left[\left(\frac {d_1}{d_1+d_2}\mu(A_1) + \frac {d_2}{d_1+d_2}A_2\right)\cup 
\frac {d_3}{d_1+d_2}A_3   \right],
\\
\notag
R_2 &=& \on{conv}\left[\left(\frac {d_2}{d_2+d_3}A_2+\frac {d_3}{d_2+d_3}\mu(A_3)\right)\cup
 \frac {d_1}{d_2+d_3}A_1   \right].
\eean
Cf. \eqref{ex_1tL} and \eqref{ex_1tR}.

\vsk.2>
Having this set with involution and  transformations we may consider the associated decorated planar binary tree.
The problem is to study the asymptotics of  triples of convex subsets of 
$\R^2$ 
along the paths of the tree. Such asymptotics reflect the asymptotics of the
Newton polygons of the  $*$-Markov polynomials.

\begin{rem}
The triple of convex sets in \eqref{t0co} are  the projected normalized Newton polygons of the triple
$(s_2,  s_2(s_1^2-s_2)-s_3s_1, s_1^2-s_2)$. 

The $L$ and $R$ transformations in \eqref{ex_1tLc} and \eqref{ex_1tRc} are just the reformulations  of the 
$L$ and $R$ transformations of polynomials in \eqref{ex_1tL} and \eqref{ex_1tR} in the language 
of the their projected normalized  Newton polygons.

\end{rem}

\section{Odd $*$-Fibonacci Polynomials}
\label{sec Fib}

\subsection{Definition of odd $*$-Fibonacci polynomials} 

The left boundary path of the Markov tree
corresponds to the sequence of Markov triples 
$(3, 15, 6), (3,39,15), (3, 102, 19)$, \dots\,,  with general term
$(3, 3\phi_{2n+1}, 3\phi_{2n-1})$,
where $\phi_{2n+1}, \phi_{2n-1}$ are odd Fibonacci numbers, 
 \bea
\phi_1=1,\quad \phi_3=2,\quad \phi_5=5, \quad \phi_7=13,\quad \phi_9 = 34,\quad
 \dots \quad,
\eea
with  recurrence relation
\bean
\label{reF}
\phi_{2n+3}= 3 \phi_{2n+1} - \phi_{2n-1}.
\eean
We define the {\it odd $*$-Fibonacci polynomials}  recursively by the formula
\bean
\label{recF}
F_1(\bs s) 
&=&
s_1, \quad \quad \
F_3(\bs s)=s_1^2-s_2,
\\
\label{ORR}
 F_{2n+3}(\bs s) 
 &=&
  g_n F_{2n+1}(\bs s) - s_3 F_{2n-1}(\bs s),
\eean
where $g_n=s_2$ if $n$ is odd, and $g_n=s_1$ if $n$ is even. In other words we have
\bean
\label{RecR}
&&
F_{4n+3} = s_1 F_{4n+1} - s_3 F_{4n-1}\,,
\\
\notag
&&
F_{4n+5} = s_2 F_{4n+3} - s_3 F_{4n+1}\,.
\eean
\begin{lem}
We have  $\evo(F_{2n+1}) = 3\phi_{2n+1}$.
\qed
\end{lem}

The first odd $*$-Fibonacci polynomials are
\begin{align*}
F_1(\bs s)&=s_1,\\
F_3(\bs s)&=s_1^2-s_2,\\
F_5(\bs s)&=s_1^2s_2 - s_1s_3-s_2^2,\\
F_7(\bs s)&= s_1^3s_2-2  s_1^2s_3- s_1s_2^2+s_2 s_3,\\
F_9(\bs s)&=s_1^3s_2^2 -3 s_1^2s_2 s_3 -s_1s_2^3 + s_1s_3^2+2 s_2^2 s_3,\\
F_{11}(\bs s)&=s_1^4s_2^2 -4  s_1^3s_2 s_3-s_1^2s_2^3 +3  s_1^2s_3^2+3  s_1s_2^2 s_3-s_2 s_3^2,\\
F_{13}(\bs s)&=s_1^4s_2^3 -5 s_1^3s_2^2 s_3 -s_1^2s_2^4 +6 s_1^2s_2 s_3^2 -s_1s_3^3 +4 s_1s_2^3 s_3 -3 s_2^2 s_3^2\,.
\end{align*}

\begin{thm}
\label{thm Fib} 

For $n>1$ the triple
\bean
\label{Fib 3}
(g_{n-1}^*, F_{2n+1}, F_{2n-1}^*)
\eean
is the reduced polynomial presentation of the Markov triple 
$(3, 3\phi_{2n+1}, 3\phi_{2n-1})$.

\end{thm}

\begin{proof} The proof is by induction on $n$.
The statement is true for $n=2$, since
\bea
(g_{1}^*, F_{5}, F_{3}^*) = (s_2^*,  s_2(s_1^2-s_2)-s_3s_1, (s_1^2-s_2)^*)
\eea
is the reduced polynomial presentation of the Markov triple $(3,15,6)$,
see \eqref{exa sol}.

Assume that $(g_{n-1}^*, F_{2n+1}, F_{2n-1}^*)$
is the reduced polynomial presentation of the Markov triple 
$(3, 3\phi_{2n+1}, 3\phi_{2n-1})$. Denote 
$f=(f_1,f_2,f_3):=(g_{n-1}, F_{2n+1}, F_{2n-1})$. Let $Lf$ be the triple
defined in Theorem \ref{thm 4.14}. By Theorem \ref{thm 4.14}
the triple
\bea
Lf 
&=& (\mu(g_{n-1}), \mu(g_{n-1}) F_{2n+1} -s_3 F_{2n-1}, F_{2n+1})
\\
&=& (g_{n}, g_{n} F_{2n+1} -s_3 F_{2n-1}, F_{2n+1})
\\
&=&
(g_{n}, F_{2n+3}, F_{2n+1})
\eea
is such that the triple $(g_{n}^*, F_{2n+3}, F_{2n+1}^*)$
is the reduced polynomial presentation of the Markov
triple 
\bea
(3,9\phi_{2n+1} - 3\phi_{2n-1}, 3\phi_{2n+1})
= (3, 3( 3\phi_{2n+1} - \phi_{2n-1}), 3\phi_{2n+1})=
(3, 3\phi_{2n+3}, 3\phi_{2n+1}).
\eea
This proves the theorem.
\end{proof}

\begin{cor}
The odd $*$-Fibonacci polynomials are $*$-Markov polynomials.
\qed
\end{cor}

\begin{rem}
There are many $q$-deformations of (odd) Fibonacci numbers. For example,
 S.\,Morier-Genoud and V.\,Ovsienko \cite{MO20} consider
the odd Fibonacci polynomials $f_{2k+1}(q)$,
defined by the  relations
\bean
\label{OvI}
f_1(q) 
&=&
q^{-1}, 
\qquad \ 
f_3(q)=1+q,
\\
\label{OvR}
f_{2n+3} 
&=&
 (1+q+q^2)f_{2n+1} - q^2f_{2n-1}.
\eean
As  V.\,Ovsienko informed us,  our recurrence relation \eqref{ORR}
turns into  relation \eqref{OvR} under the specification
$s_1=s_2=1 + q+q^2, s_3=q^2$. Our initial conditions \eqref{recF}
turn into $1+q+q^2$, $(1+q+q^2)(q+q^2)$. Hence for any $k$
the odd $*$-Fibonacci polynomials $F_{2k+1}(\bs s)$ evaluated at
$s_1=s_2=1 + q+q^2, s_3=q^2$ equals 
$(q+q^2) f_{2k+1}(q)$.

\end{rem}

%!!!

\subsection{Formula for odd $*$-Fibonacci polynomials}

\begin{thm}
For $n\geq 0$, we have
\bean
\label{F_{4k+1}}
&&
{}
\\
&&
\notag
F_{4n+1}(\bs s) = 
s_1\sum_{i=0}^n \binom{2n-i}{i}(-s_3)^i(s_1s_2)^{n-i}
-s_2\sum_{i=0}^{n-1} \binom{2n-1-i}{i}(-s_3)^is_1^{n-1-i}s_2^{n-i}\,,
\phantom{aaaa}
\eean
\bean
\label{F_{4k+3}}
&&
{}
\\
&&
\notag
F_{4n+3}(\bs s)= 
s_1\sum_{i=0}^n \binom{2n+1-i}{i}(-s_3)^is_1^{n+1-i}s_2^{n-i}
-s_2\sum_{i=0}^{n} \binom{2n-i}{i}(-s_3)^i(s_1s_2)^{n-i}\,.
\phantom{aaaa}
\eean

\end{thm}

\proof
The  proof is by induction. The  formulas correctly reproduce $F_1, F_3$.
Then  $s_1 F_{4n+1} - s_3 F_{4n-1}$ equals
\bea
&&
s_1\Big(s_1\sum_{i=0}^n \binom{2n-i}{i}(-s_3)^i(s_1s_2)^{n-i}
\,-\,s_2\sum_{i=0}^{n-1} \binom{2n-1-i}{i}(-s_3)^is_1^{n-1-i}s_2^{n-i}\Big)
\\
&&
-s_3\Big(s_1\sum_{i=0}^{n-1} \binom{2n-1-i}{i}(-s_3)^is_1^{n-i}s_2^{n-1-i}
\,-\,s_2\sum_{i=0}^{n-1} \binom{2n-2-i}{i}(-s_3)^i(s_1s_2)^{n-1-i}\Big).
\eea
We have
\bea
&&
s_1^2\sum_{i=0}^n \binom{2n-i}{i}(-s_3)^i(s_1s_2)^{n-i}
-s_3s_1\sum_{i=0}^{n-1} \binom{2n-1-i}{i}(-s_3)^is_1^{n-i}s_2^{n-1-i}
\\
&&
=\,s_1\sum_{i=0}^n \binom{2n+1-i}{i}(-s_3)^is_1^{n+1-i}s_2^{n-i}
\eea
and
\bea
&&
s_1 s_2\sum_{i=0}^{n-1} \binom{2n-1-i}{i}(-s_3)^is_1^{n-1-i}s_2^{n-i}
-s_3s_2\sum_{i=0}^{n-1} \binom{2n-2-i}{i}(-s_3)^i(s_1s_2)^{n-1-i}
\\
&&
=\,s_2\sum_{i=0}^{n} \binom{2n-i}{i}(-s_3)^i(s_1s_2)^{n-i}\,.
\eea
Hence,  $s_1 F_{4n+1} - s_3 F_{4n-1}= F_{4n+3}$. The other identity is proved similarly.
\endproof

\begin{cor} For the ordinary Fibonacci integers we have formulae

\begin{align}
\label{F_{4k+1}i}
\phi_{4n+1}\,
=&
\, 
\sum_{i=0}^n \binom{2n-i}{i} (-1)^i 3^{2n-2i}
\,-\,\sum_{i=0}^{n-1} \binom{2n-1-i}{i}(-1)^i 3^{2n-1-2i}\,,\\
\label{F_{4k+2}i}
\phi_{4n+2}\,
=&
\, 
\sum_{i=0}^n \binom{2n-i}{i}(-1)^i3^{2n+1-2i}\,-\,2\sum_{i=0}^{n} \binom{2n-i}{i}(-1)^i 3^{2n-2i}\,,\\
\label{F_{4k+3}i}
\phi_{4n+3}\,
=&
\, 
\sum_{i=0}^n \binom{2n+1-i}{i}(-1)^i3^{2n+1-2i}
\,-\,\sum_{i=0}^{n} \binom{2n-i}{i}(-1)^i 3^{2n-2i}\,,\\
\label{F_{4k+4}i}
\phi_{4n+4}\,=&\,
\sum_{i=0}^n\binom{2n+1-i}{i}(-1)^i3^{2n+1-2i}\,.
\end{align}
\end{cor}

\proof
Formulae \eqref{F_{4k+1}i} and \eqref{F_{4k+3}i} follow from \eqref{F_{4k+1}} and \eqref{F_{4k+3}}. Formulae \eqref{F_{4k+2}i} and \eqref{F_{4k+4}i} easily follow from the identities 
\[
\phi_{4n+2}=\phi_{4n+3}-\phi_{4n+1},\quad \phi_{4n+4}=\phi_{4n+3}+\phi_{4n+2}.\makeatletter\displaymath@qed
\]

\subsection{Newton polygons of odd $*$-Fibonacci polynomials}
\label{sec NP}

\begin{lem}

The odd $*$-Fibonacci polynomials  $F_{4n+1}$ and $F_{4n+3}$ are of bi-degree
$(2n+1, 3n+1)$ and $(2n+2, 3n+2)$, respectively.
\qed

\end{lem}

The Newton polygon $N_{F_{4n+1}}$  of $F_{4n+1}$  is the convex hull
of four points   $(n+1,n, 0), (1,0,n),(n-1,n+1, 0), (0,2,n-1)$. The projected normalized
Newton polygon $\tilde N_{F_{4n+1}}$ is the convex hull of four points 
\bea
\Big(\frac{n+1}{2n+1},\frac{n}{2n+1}\Big), \Big(\frac{1}{2n+1},0\Big),\Big(\frac{n-1}{2n+1},
\frac{n+1}{2n+1}\Big), \Big(0,\frac{2}{2n+1}\Big) .
\eea
The limit  of  $\tilde N_{F_{4n+1}}$ as $n\to\infty$
is the interval with vertices $(0,0)$ and $(1/2,1/2)$.

\vsk.2>
The Newton polygon $N_{F_{4n+3}}$ is the convex hull
of four points 
 $(n+2,n, 0), (2,0,n), (n,n+1, 0), (0,1,n)$. The projected normalized
Newton polygon $\tilde N_{F_{4n+3}}$ is the convex hull of four points 
\bea
\Big(\frac{n+2}{2n+2},\frac{n}{2n+2}\Big), \Big(\frac{2}{2n+2},0\Big),\Big(\frac{n}{2n+2},
\frac{n+1}{2n+2}\Big), \Big(0,\frac{1}{2n+2}\Big) .
\eea
The limit  of  $\tilde N_{F_{4n+3}}$ as $n\to\infty$
is the interval with vertices $(0,0)$ and $(1/2,1/2)$, see Section \ref{sec dti}.

\subsection{Generating function}
Introduce the generating power series of odd $*$-Fibonacci polynomials,
\bean
\mathcal F(\bs s, t):=\sum_{n=0}^\infty F_{2n+1}(\bs s)t^{2n+1}.
\eean

\begin{thm}
We have
\bean\label{genF}
\mathcal F(\bs s, t)=\frac{s_2 s_3 t^7+t^5 \left(s_2^2-s_1 s_3\right)+t^3 \left(s_2-s_1^2\right)-s_1 t}{-s_3^2t^8-(2s_3-s_1s_2)t^4-1}.
\eean
\end{thm}
\proof
Split the series $\mathcal F(\bs s,t)$ as follows
\bean\label{splitF}
\mathcal F(\bs s, t)=\underbrace{\sum_{k=0}^\infty F_{4k+1}(\bs s)t^{4k+1}}_{\mathcal F_1(\bs s, t)}+\underbrace{\sum_{n=0}^\infty F_{4k+3}(\bs s)t^{4k+3}}_{\mathcal F_2(\bs s, t)}.
\eean
From the recursive relation \eqref{recF}, we deduce
\bea
\underbrace{(1+s_3 t^4)\mathcal F_1(\bs s,t)}_{*}-t^2s_1\mathcal F_1(\bs s,t)+(1+s_3 t^4)\mathcal F_2(\bs s,t)\underbrace{-t^2s_2\mathcal F_2(\bs s,t)}_{*}=\underbrace{F_1t}_{*}-F_{-1}s_3t^3,
\eea
where $F_{-1}=s_2/s_3,\ F_1=s_1$. 
The terms marked by $*$ have only powers $t^{4k+1}$, with $k\geq 0$. The remaining terms have only powers $t^{4k+3}$, with $k\geq 0$. We have a linear system
\bean
\begin{pmatrix}
1+s_3t^4&-s_2t^2\\
-s_1t^2&1+s_3t^4
\end{pmatrix}
\begin{pmatrix}
\mathcal F_1(\bs s,t)\\
\mathcal F_2(\bs s,t)
\end{pmatrix}=
\begin{pmatrix}
F_1t\\
-F_{-1}s_3t^3
\end{pmatrix}.
\eean
Hence
\bean\label{F1F2}
\mathcal F_1(\bs s,t)=
\frac{(s_1 s_3-s_2^2) t^5+s_1t}{\left(s_3 t^4+1\right)^2-s_1 s_2 t^4},
\quad 
\mathcal F_2(\bs s,t)=\frac{s_2s_3t^7+(s_2-s_1^2)t^3}{s_1 s_2 t^4-\left(s_3 t^4+1\right)^2}.
\eean
Equation \eqref{genF} follows from \eqref{splitF} and \eqref{F1F2}.
\endproof

\begin{cor}
For any $n\geq 0$, we have
\bean
F_{2n+1}(\bs s)=\frac{1}{(2n+1)!}\left.\frac{\partial^{2n+1}}{\partial t^{2n+1}}\right|_{t=0}\mathcal F(\bs s,t) .
\eean
\qed
\end{cor}

\begin{rem}
At $\bs s=\bs s_o$, the generating function $\mathcal F(\bs s,t)$ reduces to
\bea
\mathcal F(\bs s_o,t)=\frac{3(t-t^3)}{t^4-3t^2+1}=3\left(t+2 t^3+5 t^5+13 t^7+34 t^9+\dots\right),
\eea
 the generating function of odd Fibonacci numbers multiplied by 3.
\end{rem}

\subsection{Binet formula for odd $*$-Fibonacci polynomials}
In this section we consider the generating function $\mathcal F(\bs s,t)$ as a rational function of $t$
with coefficients depending on the parameters $\bs s$ varying in a neighborhood of 
$\bs s_o$.

\vsk.2>
The poles of  $\mathcal F(\bs s,t)$ 
are  the roots of the polynomial  $s_3^2t^8+(2s_3-s_1s_2)t^4+1$.

We will use the roots of $s_3^2t^2+(2s_3-s_1s_2)t+1$,
\bean
a_\pm(\bs s):=\frac{s_1s_2-2s_3\pm\left(s_1^2s_2^2-4s_1s_2s_3\right)^\frac{1}{2}}{2s_3^2}
\eean
with $a_\pm(\bs s_0)=\frac{7\pm 3\sqrt{5}}{2}$.

\vsk.2>
Introduce $\al_0,\al_1,\beta_0,\beta_1, \ga_1,\ga_2$ by the formulas
\bea
&&
t((s_1 s_3-s_2^2) t^4+s_1) = \al_1t^5+\al_0t,
\\
&&
t^3 (-s_2+s_1^2-s_2 s_3 t^4) = \beta_1t^7 + \beta_0t^3,
\\
&&
s_3^2t^8+(2s_3-s_1s_2)t^4+1 = \ga_2 t^8 + \ga_1t^4+1.
\eea
Then
\bea
\mc F_1(\bs s,t) = \frac{\al_1t^5+\al_0t}{\ga_2 t^8 + \ga_1t^4+1},
\qquad
\mc F_2(\bs s,t) = \frac{\beta_1t^7 + \beta_0t^3}{\ga_2 t^8 + \ga_1t^4+1}.
\eea

\begin{thm}
\label{thm Binet}
We have
\bean
&&
\mc F_1(\bs s,t) = 
-\sum_{k=0}^\infty\Big(
\frac{\al_1a_++\al_0}{a_+^{k+1}(2\ga_2 a_+ + \ga_1)}
+
\frac{\al_1a_-+\al_0}{a_-^{k+1}(2\ga_2 a_- + \ga_1)}\Big)t^{4k+1},
\phantom{aaaaaaaaaaa}
\\
\notag
&&
\mc F_2(\bs s,t) = 
-\sum_{k=0}^\infty\Big(
\frac{\beta_1a_++\beta_0}{a_+^{k+1}(2\ga_2 a_+ + \ga_1)}
+
\frac{\beta_1a_-+\beta_0}{a_-^{k+1}(2\ga_2 a_- + \ga_1)}\Big)t^{4k+3}.
\phantom{aaaaaaaaaaa}
\eean

\end{thm}

\begin{proof}
We prove the first formula. The  proof of the second is similar.

 The roots of $s_3^2t^8+(2s_3-s_1s_2)t^4+1$ are
\begin{align}
b_n(\bs s): =\om^na_+(\bs s)^\frac{1}{4},\quad 
b_{4+n}(\bs s):=\om^n a_-(\bs s)^\frac{1}{4},\quad n=1,2,3,4,
\end{align}
where $\om = e^{\pi\sqrt{-1}/2}$. We have
$\mathcal F_1(\bs s, t)=\sum_{n=1}^8\frac{A_n(\bs s)}{t-b_n(\bs s)}$,
where
\bea
A_n(\bs s)={\rm res}_{t=b_n(\bs s)}\mathcal F_1(\bs s,t)
=\frac{\al_1b_n(\bs s)^4+\al_0}{b_n(\bs s)^{2}(8\ga_2 b_n(\bs s)^4 + 4\ga_1)}\,.
\eea
Hence
\bea
\label{F1res}
&&
\mathcal F_1(\bs s,t)=\sum_{n=1}^8\frac{A_n(\bs s)}{t-b_n(\bs s)}=-\sum_{n=1}^8\frac{A_n(\bs s)}{b_n(\bs s)}\frac{1}{1-\frac{t}{b_n(\bs s)}}
\\
&&
\phantom{aaaaaa}
=
-\sum_{m=0}^\infty\sum_{n=1}^8\frac{A_n(\bs s)}{b_n(\bs s)^{m+1}}t^{m}
 = -\sum_{m=0}^\infty\sum_{n=1}^8
\frac{\al_1b_n(\bs s)^4+\al_0}{b_n(\bs s)^{m+3}(8\ga_2 b_n(\bs s)^4 + 4\ga_1)}t^{m}
\\
&&
\phantom{aaaaaa}
=
-\sum_{k=0}^\infty\Big(
\frac{\al_1a_+(\bs s)+\al_0}{a_+(\bs s)^{k+1}(2\ga_2 a_+(\bs s) + \ga_1)}
+
\frac{\al_1a_-(\bs s)+\al_0}{a_-(\bs s)^{k+1}(2\ga_2 a_-(\bs s) + \ga_1)}\Big)t^{4k+1}.
\eea
\end{proof}

\begin{cor}
\label{cor BiF}
We have
\bean
&&
F_{4k+1}(\bs s) = 
-
\frac{\al_1a_++\al_0}{a_+^{k+1}(2\ga_2 a_+ + \ga_1)}
-
\frac{\al_1a_-+\al_0}{a_-^{k+1}(2\ga_2 a_- + \ga_1)}\,,
\phantom{aaaaaaaaaaa}
\\
\notag
&&
F_{4k+3}(\bs s) = 
-
\frac{\beta_1a_++\beta_0}{a_+^{k+1}(2\ga_2 a_+ + \ga_1)}
-
\frac{\beta_1a_-+\beta_0}{a_-^{k+1}(2\ga_2 a_- + \ga_1)}.
\phantom{aaaaaaaaaaa}
\eean
If $\bs s$ is in a small neighborhood of $\bs s_o$, then $|a_+(\bs s)| > |a_-(\bs s)|$ and
\bean
\label{as F}
&&
F_{4k+1}(\bs s) \sim
-
\frac{\al_1a_-+\al_0}{a_-^{k+1}(2\ga_2 a_- + \ga_1)},
\quad
F_{4k+3}(\bs s) \sim - \frac{\beta_1a_-+\beta_0}{a_-^{k+1}(2\ga_2 a_- + \ga_1)},
\\
\label{as F/F}
&&
s_2\,\frac{F_{4k+3}(\bs s)}{F_{4k+1}(\bs s)} \sim
s_2\,\frac{\beta_1a_-+\beta_0}{\al_1a_-+\al_0},
\quad\phantom{aaa}
s_1\,\frac{F_{4k+5}(\bs s)}{F_{4k+3}(\bs s)} \sim
s_1\,\frac{\al_1a_-+\al_0}{a_-(\beta_1a_-+\beta_0)},
\eean
as $n\to\infty$.
\end{cor}

\begin{lem}
\label{lem 10.2}
We have
\bean
\label{1=1}
s_2\,\frac{\beta_1a_\pm+\beta_0}{\al_1a_\pm+\al_0}
=
s_1\,\frac{\al_1a_\pm+\al_0}{a_\pm(\beta_1a_\pm+\beta_0)}.
\eean

\end{lem}

\begin{proof} 
The proof is by direct verification.

\end{proof}

\subsection{Odd $*$-Fibonacci polynomials with negative indices}
The relations
\bea
&&
F_{4n+3} = s_1 F_{4n+1} - s_3 F_{4n-1}\,,
\\
&&
F_{4n+5} = s_2 F_{4n+3} - s_3 F_{4n+1}\,.
\eea
can be reversed and written as
\bea
&&
F_{-(4n+3)}\, = \,\frac{s_2}{s_3}\, F_{-(4n+1)}\, -\, \frac 1{s_3} \,F_{-(4n-1)}\,,
\\
&&
F_{-(4n+5)} \,=\, \frac{s_1}{s_3}\, F_{-(4n+3)}\, -\, \frac1{s_3} \,F_{-(4n+1)}\,.
\eea
This allows us to define the $*$-Fibonacci Laurent polynomials with negative indices.

\begin{thm}
For any $n$ we have
$F_{-2n-1}= (F_{2n+1})^*\,.$
\qed

\end{thm}

For example,  $F_1=s_1$, $F_{-1}=(F_1)^*=\frac{s_2}{s_3}\,.$

\subsection{Cassini identity for odd $*$-Fibonacci polynomials}

The odd $*$-Fibonacci numbers satisfy the following identities:
\bea
\phi_{2n+3}^2 - \phi_{2n+5}\phi_{2n+1} = 
\phi_{2n+1}^2 - \phi_{2n+3}\phi_{2n-1} = \dots = -1\,.
\eea
Indeed, we have
\bea
&&
\phi_{2n+3}^2 - \phi_{2n+5}\phi_{2n+1} = 
\phi_{2n+3}(3\phi_{2n+1}-\phi_{2n-1}) - \phi_{2n+5}\phi_{2n+1} 
\\
&&
\phantom{aa}
=\phi_{2n+1}(3\phi_{2n+3}- \phi_{2n+5}) -\phi_{2n+3}\phi_{2n-1}
=\phi_{2n+1}^2 - \phi_{2n+3}\phi_{2n-1}\,.
\eea

\begin{thm}
The odd $*$-Fibonacci polynomials satisfy the following  identities:
\bea
g_{n+1}F_{2n+3}^2 - g_n F_{2n+5}F_{2n+1} = 
s_3(g_{n}F_{2n+1}^2 - g_{n-1} F_{2n+3}F_{2n-1}) =
s_3^{n-1} ( s_1^3s_3+s_2^3-s_1^2s_2^2)\,.
\eea

\end{thm}

\proof
We have
\bea
&&
g_{n+1}F_{2n+3}^2 - g_n F_{2n+5}F_{2n+1}
= 
g_{n+1}F_{2n+3}(g_n F_{2n+1}-s_3F_{2n-1}) - g_nF_{2n+5}F_{2n+1} 
\\
&&
\phantom{aaa}
=g_n F_{2n+1}(g_{n+1} F_{2n+3} -  F_{2n+5}) -s_3g_{n+1}F_{2n+3}F_{2n-1}
\\
&&
\phantom{aaaaaa}
=s_3(g_n F_{2n+1}^2 - g_{n-1}F_{2n+3}F_{2n-1})\,
\eea
and
\[
g_0F_1^2-g_{-1}F_{-1}F_3 = s_3^{-1}( s_1^3s_3+s_2^3-s_1^2s_2^2)\,.\makeatletter\displaymath@qed
\]

\begin{cor}
\label{cor rat F}
We have
\bean
\label{F/F}
g_n\,\frac{F_{2n+5}}{F_{2n+3}} \,-\, g_{n-1}\,\frac{F_{2n+3}}{F_{2n+1}}
\, =\, s_3^{n-1}\,\frac{ s_1^2s_2^2- s_1^3s_3-s_2^3}{F_{2n+3}F_{2n+1}}\,.
\eean
\qed
\end{cor}

 Identity \eqref{F/F}  evaluated at $\bs s=\bs s_o$ 
  takes the form
\bea
\frac{\phi_{2n+5}}{\phi_{2n+3}} - \frac{\phi_{2n+3}}{\phi_{2n+1}}
=\frac1{\phi_{2n+5}\phi_{2n+3}}\,.
\eea

If $\bs s$ lies in a small neighborhood of $\bs s_o$, then
the right-hand and left-hand sides of formula \eqref{F/F} tend to zero 
as $n\to\infty$, see \eqref{as F}, \eqref{as F/F}, \eqref{1=1}.

\subsection{Continued fractions for odd $*$-Fibonacci  polynomials}

Consider the field 
$\Q(\bs s)$ of rational functions in variables $\bs s$ with rational coefficients.
Consider  a continued fraction  of the following form
\bea
\Big[a_0;\frac{b_1}{a_1},\dots, \frac{b_5}{a_5}\Big]:=
a_0 + \frac{b_1}{a_1+\frac {b_2}{a_2+\frac {b_3}{a_3 + \frac {b_4}{a_4+\frac {b_5}{a_5}}}}}\ \in \ \Q(\bs s)\,,
\eea 
where $a_0,\dots,a_5 \in \Z[s_1,s_2, s_3]$ and each of $b_1,\dots, b_5$ is of the form
\bea
s_1^{a_1}s_2^{a_2}s_3^{a_3}, \qquad a_1,a_2\in\Z_{\geq 0}, \ a_3\in\Z\,.
\eea
Similarly we define continued fractions
$\Big[a_0;\frac{b_1}{a_1},\dots,\frac{b_n}{a_n}\Big]$\,
for any positive integer $n$.

\vsk.2>
For example,
\bea
\big[s_2;\frac{s_1}{s_2},\frac{s_1}{s_3}\big] = s_2 +\frac {s_1}{s_2 +\frac{s_1}{s_3}} = \frac{s_2^2s_3 +s_1s_2+s_1s_3}{s_2s_3+s_1}\,.
\eea

\begin{thm}
\label{thm 10.16}
For $n\geq 1$ we have
\bea
\frac{F_{2n+3}}{F_{2n+1}} = 
\big[g_n;\frac{-s_3}{g_{n-1}},\frac{-s_3}{g_{n-2}},\dots,
\frac{-s_3}{g_{1}},\frac{-s_3}{g_{0}},\frac{-s_2}{s_1}\big]\,.
\eea

\end{thm}

\begin{proof}
The formula follows from the recurrence relations for the $*$-Fibonacci polynomials.
\end{proof}

For example,
\bea
&&
\frac{F_{5}}{F_{3}} = 
\big[g_1;\frac{-s_3}{g_0},\frac{-s_2}{s_1}\big]\, =\, g_1 -\frac{s_3}{g_0- \frac{s_2}{s_1}}\, =\,
s_2 -\frac{s_3}{s_1-\frac{s_2}{s_1}}\,,
\\
&&
\frac{F_{7}}{F_{5}} = 
\big[g_2;\frac{-s_3}{g_1}, \frac{-s_3}{g_0},\frac{-s_2}{s_1}\big]\, =\, 
g_2 -\frac{s_3}{g_1-\frac{s_3}{g_0- \frac{s_2}{s_1}}}\, 
=\,
s_1 -\frac{s_3}{s_2-\frac{s_3}{s_1- \frac{s_2}{s_1}}}\,.
\eea

\begin{rem}
Formulas \eqref{as F/F} show that the continued fraction of Theorem \ref{thm 10.16} converges
to a an element of a quadratic extension of the field 
$\Q(\bs s)$.

\end{rem}

\section{Odd $*$-Pell polynomials} 
\label{sec Pell}

\subsection{Definition of odd $*$-Pell polynomials}

The right boundary path of the Markov tree
corresponds to the sequence of Markov triples $(3,15, 6), (15, 87, 6), (87, 507, 6)$, \dots\,, with general term
$( 3\psi_{2n-1},3\psi_{2n+1}, 6)$, where $\psi_{2n+1}$, $\psi_{2n-1}$ are odd Pell numbers,
 \bea
\psi_1=1,\qquad \psi_3=5,\quad \psi_5=29, \quad \psi_7=169,\quad
 \dots \quad ,
\eea
with the recurrence relation
\bean
\label{reP}
\psi_{2n+3}= 6 \psi_{2n+1} - \psi_{2n-1}\,.
\eean

We define the {\it odd $*$-Pell polynomials}  recursively by the formula
\bean
\label{recP}
&&
P_1(\bs s)
=
s_2, \qquad
P_3(\bs s) =s_1^2s_2 - s_1s_3-s_2^2,
 \\
 \notag
&&  P_{2n+3} (\bs s)
  =
  h_n P_{2n+1}(\bs s) - s_3^2 P_{2n-1}(\bs s)\,,
\eean
where $h_n=s_2^2-s_1s_3$ if $n$ is odd and $h_n=s_1^2-s_2$ if $n$ is even. In other words we have
\bea
&&
P_{4n+3} = (s_1^2-s_2) P_{4n+1} - s_3^2 P_{4n-1}\,,
\\
&&
P_{4n+5} = (s_2^2-s_1s_3) P_{4n+3} - s_3^2 P_{4n+1}\,.
\eea
\begin{lem}
We have  $\evo(P_{2n+1}) = 3\psi_{2n+1}$.
\qed
\end{lem}

The first odd $*$-Pell polynomials are
\begin{align*}
P_1(\bs s)&=s_2,\\
P_3(\bs s)&=s_1^2s_2 - s_1s_3-s_2^2,\\
P_5(\bs s)&= s_1^2 s_2^3-s_1^3  s_2s_3-s_2s_3^2 +s_1^2 s_3^2-s_2^4,\\
P_7(\bs s)&=s_1^4s_2^3 + s_1^4s_3^2-s_1^5s_2 s_3 + s_1^3s_2^2 s_3-2  s_1^2s_2^4-3 s_1^2s_2 s_3^2 +s_1s_3^3 +s_2^5+2 s_2^2 s_3^2,\\
P_9(\bs s)&=s_1^4 s_2^5-2 s_1^2 s_2^6+s_2^7-s_1  s_2^5s_3+3 s_2^4s_3^2 +3 s_1^3 s_2^4s_3 -4 s_1^2 s_2^3s_3^2 -2 s_1^5s_2^3 s_3 \\
&\ \ \ -s_1 s_2^2s_3^3 +s_2s_3^4 +4 s_1^3 s_2s_3^3 +s_1^6 s_2s_3^2 -2 s_1^2 s_3^4-s_1^5 s_3^3\,.
\end{align*}

\begin{thm}
\label{thm Pe} 
For $n>0$ the triple 
\bean
\label{Pe 3}
(P_{2n-1}^*, P_{2n+1}, h_{n-1}^*)
\eean
is the reduced polynomial presentation of the Markov triple 
$(3\psi_{2n-1}, 3\psi_{2n+1}, 6)$.
\end{thm}

\begin{proof} The proof is by induction on $n$.
The statement is true for $n=1$, since
\bea
(P_{1}^*, P_{3}, h_0^*) = (s_2^*,  s_2(s_1^2-s_2)-s_3s_1, (s_1^2-s_2)^*)
\eea
is the reduced polynomial presentation of the Markov triple $(3,15,6)$,
see \eqref{exa sol}.

Assume that 
$(P_{2n-1}^*, P_{2n+1}, h_{n-1}^*)$ 
is the reduced polynomial presentation of the Markov triple 
$(3\psi_{2n-1}, 3\psi_{2n+1}, 6)$.
 Denote 
$f=(f_1,f_2,f_3):=(P_{2n-1}, P_{2n+1}, h_{n-1})$. Let $Rf$ be the triple
defined in Theorem \ref{thm 4.14}. By Theorem \ref{thm 4.14}
the triple
\bea
Rf 
&=& (P_{2n+1},P_{2n+1}\mu(h_{n-1})-s_3^2P_{2n-1}, \mu(h_{n-1}))
\\
&=& 
(P_{2n+1},P_{2n+1}h_{n}-s_3^2P_{2n-1}, h_{n})
\\
&=&
(P_{2n+1},P_{2n+3}, h_{n})
\eea
is such that the triple $(P_{2n+1}^*, P_{2n+3}, h_{n}^*)$
is the reduced polynomial presentation of the Markov
triple 
\bea
(3\psi_{2n+1}, 18\psi_{2n+1}-3\psi_{2n-1}, 6)
=
(3\psi_{2n+1}, 3(6\psi_{2n+1}-\psi_{2n-1}), 6)
=
(3\psi_{2n+1}, 3\psi_{2n+3}, 6).
\eea
This proves the theorem.
\end{proof}

\begin{cor}
The odd $*$-Pell polynomials are $*$-Markov polynomials.
\qed
\end{cor}

\subsection{Formula for odd $*$-Pell polynomials}

\begin{thm}
For $n\geq 0$, we have
\bea
P_{4n+1}(\bs s) = s_2\sum_{i=0}^n(-1)^i\binom{2n-i}{i}(h_0h_1)^{n-i}s_3^{2i}
-s_1\sum_{i=0}^{n-1}(-1)^i\binom{2n-1-i}{i}h_1(h_0h_1)^{n-i-1}s_3^{2i+1}\,,
\eea
\bea
P_{4n+3}(\bs s)
=s_2\sum_{i=0}^{n}(-1)^i\binom{2n+1-i}{i}h_0(h_0h_1)^{n-i}s_3^{2i}
-s_1\sum_{i=0}^n(-1)^i\binom{2n-i}{i}(h_0h_1)^{n-i}s_3^{2i+1}\,.
\eea

\end{thm}

\proof
The  proof is by induction. First one checks that the formulas correctly reproduce  $P_1, P_3$. Then 
$h_0 P_{4n+1} - s_3^2 P_{4n-1}$ equals
{\small
\bea
&&
s_2\sum_{i=0}^n(-1)^i\binom{2n-i}{i}h_0(h_0h_1)^{n-i}s_3^{2i}
-s_1\sum_{i=0}^{n-1}(-1)^i\binom{2n-1-i}{i}
(h_0h_1)^{n-i}s_3^{2i+1}
\\
&&
-s_2\sum_{i=0}^{n-1}(-1)^i\binom{2n-1-i}{i}h_0(h_0h_1)^{n-i-1}s_3^{2i+2}
+s_1\sum_{i=0}^{n-1}(-1)^i\binom{2n-i-2}{i}(h_0h_1)^{n-i-1}s_3^{2i+3}.
\eea
}
We have
{\small
\bea
&&
s_2\left(\sum_{i=0}^n(-1)^i\binom{2n-i}{i}h_0(h_0h_1)^{n-i}
s_3^{2i}+\sum_{i=0}^{n-1}(-1)^{i+1}\binom{2n-1-i}{i}h_0(h_0h_1)^{n-i-1}s_3^{2i+2}\right)
\\
&&
=s_2\left(\sum_{i=0}^{n}(-1)^i\binom{2n+1-i}{i}h_0(h_0h_1)^{n-i}s_3^{2i}\right),
\eea
\bea
&&
s_1\left(\sum_{i=0}^{n-1}(-1)^{i+1}\binom{2n-1-i}{i}
(h_0h_1)^{n-i}s_3^{2i+1}+\sum_{i=0}^{n-1}(-1)^i\binom{2n-i-2}{i}(h_0h_1)^{n-i-1}s_3^{2i+3}\right)
\\
&&
=-s_1\left(\sum_{i=0}^n(-1)^i\binom{2n-i}{i}(h_0h_1)^{n-i}s_3^{2i+1}\right)\,.
\eea}
Hence, $h_0P_{4n+1}-s_3^2P_{4n-1}=P_{4n+3}$. The other identity is proved similarly.
\endproof

\begin{cor}
For the ordinary Pell numbers, we have
\begin{align}
\label{pe1}
\psi_{4n+1}=&\ \sum _{i=0}^n (-1)^i \binom{2 n-i}{i}6^{2 n-2 i} -\sum _{i=0}^{n-1} (-1)^{i} \binom{2 n-i-1}{i} 6^{2 n-2i-1},\\
\label{pe2}
\psi_{4n+2}=&\ \frac{1}{2} \sum _{i=0}^n (-1)^i \binom{2 n-i}{i} 6^{2 n-2i+1}-\sum _{i=0}^n (-1)^i  \binom{2 n-i}{i}6^{2 n-2 i},\\
\label{pe3}
\psi_{4n+3}=&\ \sum _{i=0}^n (-1)^i  \binom{2n-i+1}{i}6^{2 n-2i+1}-\sum _{i=0}^n (-1)^i  \binom{2 n-i}{i}6^{2 n-2 i},\\
\label{pe4}
\psi_{4n+4}=&\ 2 \sum _{i=0}^n (-1)^i \binom{2 n-i+1}{i}6^{2 n-2i+1} +\frac{1}{2} \sum _{i=0}^n (-1)^i \binom{2 n-i}{i}6^{2 n-2i+1} \\
\notag&\ -3 \sum _{i=0}^n (-1)^i\binom{2 n-i}{i} 6^{2 n-2 i}.
\end{align}
\end{cor}

\proof
 Formulas \eqref{pe2}, \eqref{pe4} follow from 
the recurrence relation for Pell numbers
\[
\psi_{n+1}=2\psi_n+\psi_{n-1},\quad n\geq 1.\makeatletter\displaymath@qed
\]

\subsection{Limiting Newton polygons of odd $*$-Pell polynomials}
\label{sec NPe}

\begin{lem}

The odd $*$-Pell polynomials  $P_{4n+1}$ and $P_{4n+3}$ are of bi-degree
$(4n+1, 6n+2)$ and $(4n+3, 6n+4)$, respectively.
\qed

\end{lem}

The Newton polygon $N_{P_{4n+1}}$  of $P_{4n+1}$  contains
the points
$(0,1, 2n), (2n, 2n+1, 0), (0, 3n+1,0), (3n, 1, n)$.
Hence the 
 limit  of  $\tilde N_{P_{4n+1}}$ as $n\to\infty$
contains the points $(0,0)$, $(1/2, 1/2)$, $(0, 3/4)$, $(3/4,0)$.
Therefore the limit of $\tilde N_{P_{4n+1}}$  is the 
projected quadrilateral $\tilde N_\infty$.

\vsk.2>
Similarly one checks that the limit of $\tilde N_{P_{4n+3}}$ as $n\to\infty$ is the 
projected quadrilateral $\tilde N_\infty$, 
 see Section \ref{sec dti}.

\subsection{Generating function}
Introduce the generating power series of odd $*$-Pell polynomials
\bean
\mathcal P(\bs s,t):=\sum_{n=0}^\infty P_{2n+1}(\bs s)t^{2n+1}.
\eean

\begin{thm}
We have
\bean
\label{genP}
\mathcal P(\bs s, t)=\frac{-s_1 s_3^3 t^7+\left(\left(s_1^2+s_2\right) s_3^2-s_1 s_2^2 s_3\right) t^5+\left(\left(s_1^2-s_2\right) s_2-s_1 s_3\right) t^3+s_2 t}{s_3^4 t^8+\left(s_1^3s_3 -s_1s_2 s_3 +s_2^3-s_1^2 s_2^2+2 s_3^2\right) t^4+1}.
\eean
\end{thm}
\proof
The proof is similar to the proof of the corresponding theorem on the odd Fibonacci polynomials.
\endproof

\begin{cor}
For any $n\geq 0$, we have
\bean
P_{2n+1}(\bs s)=\frac{1}{(2n+1)!}\left.\frac{\partial^{2n+1}}{\partial t^{2n+1}}\right|_{t=0}\mathcal P(\bs s,t),
\eean
where $\mc P(\bm s,t)$ is given by \eqref{genP}.
\end{cor}

\begin{rem}
At $\bs s=\bs s_o$, the generating function $\mathcal P(\bs s,t)$ reduces to
\bean
\mathcal P(\bs s_o,t)=\frac{3 \left(t-t^3\right)}{t^4-6 t^2+1}=3(t+5 t^3+29 t^5+169 t^7+985 t^9+\dots),
\eean
namely the generating series of odd Pell numbers multiplied by 3.
\end{rem}

\subsection{Other properties of $*$-Pell polynomials}

The $*$-Pell polynomials have properties similar to the properties of $*$-Fibonacci polynomials discussed in 
Section \ref{sec Fib}.  In particular one easily obtains a Binet-type formula 
like in Corollary \ref{cor BiF}. 

As examples of properties of $*$-Pell polynomials we formulate  the continued fraction property and an analog of the Cassini
identity.

\begin{thm}
For $n\geq 1$ we have
\bea
\frac{P_{2n+3}}{P_{2n+1}} = 
\big[h_n;\frac{-s_3^2}{h_{n-1}},\frac{-s^2_3}{h_{n-2}},\dots,
\frac{-s^2_3}{h_{1}},\frac{-s^2_3}{h_{0}},\frac{-s_1s_3}{s_2}\big]\,.
\eea
\qed

\end{thm}

\begin{thm}
The odd $*$-Pell polynomials satisfy the following  identities:
\bea
&&
h_{n+1}P_{2n+3}^2 - h_n P_{2n+5}P_{2n+1} = 
s_3^2(h_{n}P_{2n+1}^2 - h_{n-1} P_{2n+3}P_{2n-1}) 
\\
&&
=
 s_3^{2n-1} \Big((s_1^2-s_2)s_2^2s_3\, -\, s_1\,(s_2^2-s_1s_3) \big((s_1^2-s_2)s_2 -s_1s_3\big)\Big)\,.
\eea
\qed
\end{thm}

\section{$*$-Markov group actions}
\label{sec *M act}

In this section we study the action of the $*$-Markov group
on $\C^6$ with coordinates 
\\
$(a,a^*$, $b,b^*$, $c,c^*)$.
 It is convenient to denote these coordinates by
$(x_1,\dots,x_6)$.

\subsection{Space $\C^6$ with involution and polynomials}
\label{sec 12.1}

Consider $\C^6$ with coordinates $\bm x=(x_1,\dots,x_6)$, involution
\bea
\nu:\C^6\to\C^6, \qquad (x_1, x_2,x_3, x_4,x_5, x_6)\mapsto (x_2,x_1,x_4,x_3,x_6,x_5),
\eea
polynomials
\bea
H_1=x_1x_2+x_3x_4+x_5x_6-x_1x_3x_5,
\qquad
H_2=x_1x_2+x_3x_4+x_5x_6-x_2x_4x_6.
\eea
The $*$-Markov group $\Ga_M$ acts on $\C^6$ by the formulas
\bean\notag
\la_{i,j}\colon&& \bm x\mapsto \left((-1)^i x_1,(-1)^i x_2,(-1)^{i+j} x_3, 
(-1)^{i+j} x_4, (-1)^j x_5, (-1)^j x_6\right),
\\
\notag
\sigma_1\colon &&\bm x\mapsto(x_3,x_4,x_1,x_2,x_5,x_6),
\\
\notag
\si_2\colon &&\bm x\mapsto (x_1,x_2,x_5,x_6,x_3,x_4),\\
\label{tt1}
\tau_1\colon &&\bm x\mapsto (-x_2,\, -x_1,\, x_6,\, x_5,\, x_4-x_1x_5,\, x_3-x_2x_6), \\
\label{tt2}
\tau_2\colon &&\bm x\mapsto (x_4,\, x_3,\, x_2-x_3x_5,\, x_1-x_4 x_6,\, -x_6,\, -x_5),
\eean
and the elements $\mu_{i,j}$ act on $\C^6$ by the identity maps.

\vsk.2>
The action of the $*$-Markov group on $\C^6$ commutes with the involution $\nu$.

\begin{lem}
\label{lem dd}

The $\Ga_M$-action   preserves each of the polynomials $H_1, H_2$, and 
\bea
\nu\pb H_1=H_2, \qquad \nu\pb H_2=H_1.
\eea
Hence
 the differential forms 
$dH_1$, $dH_2$, $ dH_1\wedge dH_2$
are $\Ga_M$-invariant,  and $ dH_1\wedge dH_2$ is $\nu$ anti-invariant.
\qed
\end{lem}

\begin{lem}
\label{lem 12.3}
The holomorphic volume form 
\bea
dV:=  dx_1\wedge dx_2\wedge dx_3\wedge dx_4\wedge dx_5\wedge dx_6\,
\eea
is $\la_{i,j}, \si_1, \si_2$ invariant and $\tau_1,\tau_2,\nu$ 
anti-invariant.
\qed
\end{lem}

\begin{lem}
\label{lem OM}
The differential 4-form
\bean
\label{oM}
\Om 
&=& \ \
x_1x_3\, dx_2\wedge dx_4\wedge dx_5\wedge dx_6
\ +\
x_1x_4\, dx_2\wedge dx_3\wedge dx_5\wedge dx_6
\\
\notag
&& 
- x_1x_5\, dx_2\wedge dx_3\wedge dx_4\wedge dx_6
\ -\ 
x_1x_6 \,dx_2\wedge dx_3\wedge dx_4\wedge dx_5
\\
\notag
&& 
+ x_2x_3\, dx_1\wedge dx_4\wedge dx_5\wedge dx_6
\ +\ 
x_2x_4 \,dx_1\wedge dx_3\wedge dx_5\wedge dx_6
\\
\notag
&& 
- x_2x_5\, dx_1\wedge dx_3\wedge dx_4\wedge dx_6
\ -\ 
x_2x_6 \,dx_1\wedge dx_3\wedge dx_4\wedge dx_5
\\
\notag
&& 
+ x_3x_5\, dx_1\wedge dx_2\wedge dx_4\wedge dx_6
\ +\ 
x_3x_6\, dx_1\wedge dx_2\wedge dx_4\wedge dx_5
\\
\notag
&& 
+ x_4x_5\, dx_1\wedge dx_2\wedge dx_3\wedge dx_6
\ +\ 
x_4x_6\, dx_1\wedge dx_2\wedge dx_3\wedge dx_5\,.
\eean
 is $\la_{i,j}$, $\tau_1, \tau_2$ invariant and $ \si_1, \si_2, \nu$ anti-invariant.

\end{lem}

\begin{proof}
The proof is by direct verification. 
For example, we have
\bea
&&\tau_1\Om=x_5 (x_4-x_1 x_5) dx_1\wedge dx_2\wedge dx_3\wedge dx_6
+x_6 (x_3-x_2 x_6) dx_1\wedge dx_2\wedge dx_4\wedge dx_5\\
&&+x_5 (x_3-x_2 x_6) (dx_1\wedge dx_2\wedge dx_4\wedge dx_6-x_1 dx_1\wedge dx_2\wedge dx_5\wedge dx_6)\\
&&-x_6 (x_4-x_1 x_5) (-dx_1\wedge dx_2\wedge dx_3\wedge dx_5
-x_2 dx_1\wedge dx_2\wedge dx_5\wedge dx_6)\\
&&+x_2 x_5 (-x_1 (x_6 dx_1\wedge dx_2\wedge dx_5\wedge dx_6
-dx_1\wedge dx_3\wedge dx_5\wedge dx_6)
+x_6 dx_1\wedge dx_2\wedge dx_4\wedge dx_6\\
&&-dx_1\wedge dx_3\wedge dx_4\wedge dx_6)-x_2 (x_4-x_1 x_5) (x_6 dx_1\wedge dx_2\wedge dx_5\wedge dx_6-dx_1\wedge dx_3\wedge dx_5\wedge dx_6)\\
&&+x_2 (x_3-x_2 x_6) dx_1\wedge dx_4\wedge dx_5\wedge dx_6-x_2 x_6 (-x_6 dx_1\wedge dx_2\wedge dx_4\wedge dx_5\\
&&-x_2 dx_1\wedge dx_4\wedge dx_5\wedge dx_6+dx_1\wedge dx_3\wedge dx_4\wedge dx_5)+x_1 (x_4-x_1 x_5) dx_2\wedge dx_3\wedge dx_5\wedge dx_6\\
&&+x_1 x_5 (x_5 dx_1\wedge dx_2\wedge dx_3\wedge dx_6+x_1 dx_2\wedge dx_3\wedge dx_5\wedge dx_6-dx_2\wedge dx_3\wedge dx_4\wedge dx_6)\\
&&+x_1 (x_3-x_2 x_6) (x_5 dx_1\wedge dx_2\wedge dx_5\wedge dx_6+dx_2\wedge dx_4\wedge dx_5\wedge dx_6)\\
&&-x_1 x_6 (x_5 (-dx_1\wedge dx_2\wedge dx_3\wedge dx_5-x_2 dx_1\wedge dx_2\wedge dx_5\wedge dx_6)+dx_2\wedge dx_3\wedge dx_4\wedge dx_5\\
&&-x_2 dx_2\wedge dx_4\wedge dx_5\wedge dx_6)=\Om.
\eea
See another proof of the lemma in  Corollary \ref{cor an proof}.
That other proof also provides reasons for the existence of such a 4-form $\Om$.
\end{proof}

\subsection{Casimir subalgebra}
Denote
\bean
\label{xy}
h_1=x_1x_2,\quad h_2=x_3x_4,
\quad h_3=x_5x_6,\quad h_4=x_1x_3x_5,\quad h_5=x_2x_4x_6,
\eean
and $\bs h=(h_1,\dots,h_5)$.
Then
\bea
h_1h_2h_3-h_4h_5=0
\eea
and
\[H_1=h_1+h_2+h_3-h_4,
\qquad 
H_2=h_1+h_2+h_3-h_5.
\]
Define the \emph{Casimir subalgebra} $\mc C\subset \C[\bm x]$ to be the subalgebra
generated by $h_1,\dots, h_5$.  

\begin{thm}
\label{thm CSA}
The Casimir subalgebra is $\nu$ and $\Ga_M$ invariant.
More precisely,
\bea
&& 
\nu : \bs h \mapsto (h_1,\,h_2,\,h_3,\,h_5,\,h_4),
\\
&&
\tau_1:
\bs h \mapsto (h_1,\,h_3,\,h_2+h_1h_3-h_4-h_5, \,-h_5+h_1h_3,\,-h_4+h_1h_3),
\\
&&
\tau_2:
\bs h
\mapsto (h_2,\,h_1+h_2h_3-h_4-h_5,\, h_3,\, -h_5+h_2h_3,\,-h_4+h_2h_3),
\\
&&
\si_1:
\bs h \mapsto (h_2,\,h_1,\, h_3,\, h_4,\, h_5),
\\
&&
\si_2:
\bs h \mapsto (h_1,\,h_3, \,h_2,\, h_4, \,h_5),
\eea
and the elements $\la_{i,j}, \mu_{i,j}\in \Ga_M$ fix elements of  $\mc C$ point-wise.
\qed

\end{thm}

\subsection{Space $\C^5$ with polynomials}
\label{sec an pr}

 Consider $\C^5$ with coordinates $\bm y=(y_1,\dots, y_5)$, and involution
\[
\nu\colon\C^5\to\C^5,\quad \bm (y_1,y_2,y_3,y_4,y_5)\mapsto (y_1,y_2,y_3,y_5,y_4).
\] 
The $*$-Markov group acts on $\C^5$ by the formulas of Theorem \ref{thm CSA}.
The $*$-Markov group action  on $\C^5$ commutes with the involution $\nu$.

\vskip3mm
Denote
\bea
J = y_1y_2y_3-y_4y_5,\quad 
J_1 = y_1+y_2+y_3-y_4,
\quad
J_2 = y_1+y_2+y_3-y_5,
\eea
\bea
dW = dy_1\wedge dy_2\wedge dy_3\wedge dy_4\wedge dy_5.
\eea

\begin{lem}
\label{lem 10.5}
The polynomials $J,J_1,J_2$ are $\Ga_M$-invariant.
We have
\bea
\nu\pb J = J, \qquad \nu\pb J_1=J_2,\qquad \nu\pb J_2=J_1.
\eea
The differential form $dW$ is $\tau_1, \tau_2$ invariant and $\si_1,\si_2,\nu$ anti-invariant.
The differential form $dJ\wedge dJ_1\wedge dJ_2$ is 
$\tau_1, \tau_2$, $\si_1,\si_2$ invariant and $\nu$ anti-invariant.
\qed
\end{lem}

Let $Y=\{\bm y\in  \C^5\ |\ J(\bm y)=0\}$
be the zero level hypersurface of the polynomial $J$.
The hypersurface $Y$ has a well-defined
holomorphic nonzero differential 4-form at its  nondegenerate points, 
$\om = dW/dJ$,
called the Gelfand-Leray residue form.
It is uniquely determined by the property
\bean
\label{dwo}
dW= dJ\wedge\om.
\eean
For example, if at some point $q\in Y$ we have $\frac{\der J}{\der y_1}(q)\ne 0$,
 then at a neighborhood of that point 
 \bea
 \om = 
 \frac{1}{\frac{\der J}{\der y_1}}dy_2 \wedge dy_3\wedge dy_4\wedge dy_5
 =
 \frac{1}{y_2y_3}dy_2 \wedge dy_3\wedge dy_4\wedge dy_5
 \eea
 with property \eqref{dwo}.

\begin{cor}
\label{cor 10.6}
The form
\bea
 \om = 
 \frac{1}{y_2y_3}dy_2 \wedge dy_3\wedge dy_4\wedge dy_5,
 \eea
restricted to $Y$, extends to  a  nonzero  differential 4-form on the regular part of $Y$.
That form  is 
 $\tau_1,\tau_2$ invariant and $\si_1,\si_2, \nu$ anti-invariant.
 \qed
\end{cor}

\vsk.2>

Consider the map $F:\C^6\to \C^5$ defined by the formulas
\bean
\label{hx}
y_1=x_1x_2, 
\quad
y_2=x_3x_4, 
\quad
y_3=x_5x_6, 
\quad
y_4=x_1x_3x_5,
\quad
y_5=x_2x_4x_6.
\eean

\begin{lem} We have the following statements:
\begin{enumerate}
\item[(i)]

The map $F$ commutes with the actions of the $*$-Markov group on $\C^6$ and $\C^5$.

\item[(ii)]
The Casimir subalgebra $\mc C\subset\C[\bs x]$ is the preimage of the algebra $\C[\bs y]$
under 
\\
the map $F$. 

\item[(iii)]

The image of $F$ lies in the hypersurface $Y$. 
\qed

\end{enumerate}

\end{lem}

\begin{cor}

The preimage $F\pb \om$ of the differential form $\om$ under the map $F$ is 
 $\tau_1,\tau_2$ invariant and $\si_1,\si_2, \nu$ anti-invariant.
\qed 

\end{cor}

\begin{lem}
We have $F\pb \om =\Om$, where $\Om$ is defined in \eqref{oM}.
\end{lem}

\begin{proof}
The lemma easily follows by direct verification from the formula
\bean
\label{factor}
F\pb \om
&=& 
\Big(\frac{ dx_3}{x_3} + \frac{ dx_4}{x_4}\Big)
\wedge
\Big(\frac{ dx_5}{x_5} + \frac{ dx_6}{x_6}\Big)
\\
&&
\notag
\wedge
\Big(x_3x_5 dx_1  + x_1x_5 dx_3+ x_1x_3 dx_5\Big)
\wedge
\Big(x_4x_6 dx_2  + x_2x_6 dx_4+ x_2x_4 dx_6\Big).
\eean
\end{proof}

\begin{cor}
\label{cor an proof}
The  differential form $\Om$ is
 $\tau_1,\tau_2$ invariant and $\si_1,\si_2, \nu$ anti-invariant.
\qed
\end{cor}

Cf. Lemma \ref{lem OM}.

\section{Poisson structures on $\C^6$ and $\C^5$}
\label{sec NPo}

\subsection{Nambu-Poisson manifolds}

\begin{defn}[\cite{Ta}]Let $A$ be the algebra of functions of a manifold $Y$.  The manifold $Y$ is a
{\it Nambu-Poisson manifold} of order $n$ if there exists
a multi-linear map 
\bea
\{\,,\dots,\,\} : A^{\otimes n} \to A\,,
\eea
a {\it Nambu bracket} of order $n$, satisfying the following
properties.
\begin{enumerate}
\item[(i)] Skew-symmetry,
\bea
\{f_1,\dots,f_n\} = (-1)^{\ep(\si)} \{f_{\si(1)},\dots,f_{\si(n)}\}\,, 
\eea
for all $f_1,\dots,f_n\in A$ and $\si\in S_n$.

\item[(ii)]
Leibniz rule,
\bea
\{f_1f_2, f_3, \dots,f_{n+1}\} = f_1\{f_2, f_3, \dots,f_{n+1}\} + f_2\{f_1, f_3, \dots,f_{n+1}\}\,,
\eea 
for all $f_1,\dots,f_{n+1}\in A$.

\item[(iiii)]
Fundamental Identity (FI),
\bean
\label{FI}
&&
\{ f_1,\dots,f_{n-1}, \{g_1,\dots,g_n\}\} =
\{ \{f_1,\dots,f_{n-1}, g_1\},g_2,\dots,g_n\}
\\
\notag
&&
\phantom{aaa}
+ \{ g_1,\{f_1,\dots,f_{n-1}, g_2\},g_3,\dots,g_n\} + \dots 
\\
\notag
&&
\phantom{aaa}
+
\{ g_1,\dots,g_{n-1}, \{f_1,\dots,f_{n-1}, g_n\}\}\,,
\eean
for all $f_1,\dots,f_{n-1}, g_1,\dots,g_n\in A$.

\end{enumerate}
In particular, for $n=2$ this is the standard Poisson structure.
\end{defn}

\begin{rem}
The brackets with properties (i-ii) were considered by Y.\,Nambu \cite{Na},
who was motivated by  problems of quark dynamics. 
The notion of a Nambu-Poisson manifold was introduced by L.\,Takhtajan \cite{Ta}
in order to formalize mathematically the $n$-ary generalization
of Hamiltonian mechanics proposed by Y.\,Nambu.
 The fundamental identity was 
discovered  by V.\,Filippov \cite{F} as a generalization of the Jacobi identity for an $n$-ary Lie algebra
 and then later and independently
by Takhtajan \cite{Ta} for the Nambu-Poisson setting.

 The fundamental identity is also called
the Filippov identity.
\end{rem}

The dynamics associated with the  Nambu bracket on a Nambu-Poisson  manifold 
of order $n$
is specified by $n-1$  Hamiltonians
$H_1, \dots ,H_{n-1}\in A$, and the time evolution of $f \in A$ is given by the equation
\bean
\label{fl}
\frac{df}{dt} 
= \{H_1, \dots, H_{n-1}, f\} \,.
\eean
Let $\phi_t$ be the flow associated with equation \eqref{fl} and $U_t$ the one-parameter
group  acting on $A$ by $f \mapsto U_t(f) = f \circ \phi_t$.

\begin{thm}
[\cite{Ta}] The flow preserves the Nambu bracket,
\bean
\label{pNB}
U_t(\{f_1,\dots,f_n\})
=
\{U_t(f_1),\dots,U_t(f_n)\})\,,
\eean
for all $f_1,\dots,f_n\in A$.
\qed
\end{thm}

A function $f\in A$  is called an {\it integral of motion} for the system defined by equation
\eqref{fl} if it
satisfies $\{H_1, \dots, H_{n-1}, f\} =0$.

\begin{thm}
[\cite{Ta}]
Given $H_1, \dots, H_{n-1}$,  the Nambu bracket of $n$ integrals of motion is also an integral of motion.
\qed
\end{thm}

These two theorems follow from the fundamental identity.

\subsection{Examples}
An example of a Nambu-Poisson manifold of order $n$ is $\C^n$ with 
standard coordinates $x_1,\dots,x_n$ and {\it canonical Nambu bracket} given by
\bean
\label{p1}
\{f_1,\dots,f_n\} = {\det}_{i,j=1}^n \Big(  \frac{\der f_i}{\der x_j}\Big).
\eean

This example was considered  by Nambu  \cite{Na}. Other 
examples of Nambu-Poisson manifolds see
 in \cite{CT, Ta}. See also \cite{BF, Ch, DT, GM, OR02}.

\vsk.2>

It turns out that any Nambu-Poisson manifold of order $n >2$ has  presentation \eqref{p1} locally.

\begin{thm}
Let $Y$ be an $m$-dimensional manifold which is  a Nambu-Poisson manifold of order $n$, 
$m\geq n >2$, 
with bracket $\{ \,,\dots,\,\}$. 
Let $x\in X$ be a point such that $\{ \,,\dots,\,\}$ is nonzero at $x$. Then there exists local coordinates
$x_1,\dots,x_m$ in a neighborhood of $x$ such that 
\bea
\{f_1,\dots,f_n\} = {\det}_{i,j=1}^n \Big(  \frac{\der f_i}{\der x_j}\Big)\,.
\eea

\end{thm}

This statement was conjectured by L.\,Takhtajan \cite{Ta},  proved in 
\cite{AG, Ga}. It was discovered eventually that the theorem is a consequence of an old result in \cite{W},
reproduced in the textbook by Schouten \cite{S},  Chap. II, Sections 4 and 6, formula (6.7).
  See on that in \cite{DFST}.

\subsection{Hierarchy of Nambu-Poisson structures}

A Nambu-Poisson manifold structure of order $n$ on a manifold $X$ 
 induces an infinite family of subordinated Nambu-Poisson manifold structures on $X$
  of orders $n-1$ and lower, including a family
of Poisson structures, \cite{Ta}. 
  
  \vsk.2>
  Indeed for $H_1,\dots,H_{n-k}\in A$ define the  $k$-bracket $\{\,,\dots,\,\}_H$ by the formula
  \bean
  \label{Hbr}
  \{h_1,\dots,h_k\}_H
=
  \{H_1,\dots,H_{n-k}, h_1,\dots,h_k\}\,.
  \eean
Clearly, the  bracket $\{\,,\dots,\,\}_H$ is skew-symmetric and satisfies the Leibnitz rule.
The fundamental identity for 
$\{\,,\dots,\,\}_H$ 
follows from the fundamental identity \eqref{FI} for the original bracket.

\vsk.2>

For example, for $n=6$ and $k=4$, the fundamental identity for the bracket
\bean
\label{Br64}
\{h_1,h_2,h_3,h_4\}_{H_1,H_2} :
=
  \{H_1,H_{2}, h_1,h_2,h_3,h_4\}\,
  \eean
 and 7 functions
$u_1,u_2,u_3, v_1,v_2,v_3,v_4$ 
follows from the fundamental identity for $n=6$ 
and 11 functions $f_1,f_2,f_3,f_4,f_5, g_1,g_2,g_3,g_4,g_5,g_6$ if 
\bea
(f_1,f_2,f_3,f_4,f_5, g_1,g_2,g_3,g_4,g_5,g_6) = (H_1,H_2, u_1,u_2,u_3, H_1,H_2,v_1,v_2,v_3,v_4)\,.
\eea

\vsk.2>
The family of subordinated $k$-brackets, obtained by this construction from a given $n$-bracket, satisfy the
matching conditions described in \cite{Ta}.

\begin{example} Consider $\C^3$ with coordinates $a,b,c$ and canonical Nambu bracket of 
order 3,
\bea
\{f_1,f_2,f_3\} = \frac{df_1\wedge df_2\wedge df_3}
{da\wedge db\wedge dc}\,.
\eea
The braid group $\B_3$ acts on $\C^3$ by the formulas,
\bea
&&
\tau_1: (a,b,c) \mapsto
(-a,c,b-ac)\,,
\\
&&
\tau_2: (a,b,c) \mapsto
(b,a-bc,-c)\,.
\eea
The polynomial $H=a^2+b^2+c^2-abc$
is braid group invariant.
The subordinated 2-bracket
\bea
\{h_1,h_2\}_H := \{H,h_1,h_2\}
\eea
is the braid group invariant Dubrovin Poisson structure on $\C^3$,
\bean\notag
\label{DuPo}
\{a,b\}_H = 2c-ab,
\qquad
\{b,c\}_H = 2a-bc,
\qquad
\{c,a\}_H = 2b-ac.
\eean

\end{example}

\subsection{Poisson structure on $\C^6$}
Let us return to the space $\C^6$ with involution $\nu$, two polynomials $H_1$, $H_2$,
the holomorphic volume form $dV$, differential form $\Om$,
considered in Section \ref{sec *M act}.

\vsk.2>
On $\C^6$ consider the canonical Nambu bracket of order 6,
\bean
\label{our6}
\{f_1,\dots, f_6\} = \frac{ df_1\wedge df_2\wedge df_3\wedge df_4\wedge df_5\wedge df_6}
{dV}\,,
\eean
and associated  brackets
\bean
\label{Our64}
&&
\{f_1,f_2,f_3,f_4\}_{H_1,H_2}
=
\frac{ dH_1\wedge dH_2\wedge df_1\wedge df_2\wedge df_3\wedge df_4}
{dV}\,,\\
\label{BPo}
&&
\{f_1,f_2\}_{\Om}
= \frac {\Om\wedge df_1\wedge df_2}{dV}\,.
\eean

\begin{thm}
The Nambu bracket $\{\,,\dots,\,\}_{H_1,H_2}$ defines a Nambu-Poisson manifold structure 
on $\C^6$ of order 4. The structure $\{\,,\dots,\,\}_{H_1,H_2}$ is 
$\la_{i,j}, \si_1,\si_2, \nu$ invariant and $\tau_1,\tau_2$ anti-invariant.
\end{thm}

\begin{proof}
The theorem is a corollary of Lemmas \ref{lem dd} and \ref{lem 12.3}.
\end{proof}

\begin{thm}
\label{thm NP2}

The bracket $\{\,,\,\}_{\Om}$ defines a Poisson structure 
on $\C^6$. The Poisson structure $\{\,,\,\}_{\Om}$ is $\la_{i,j}, \nu$ invariant and 
$\tau_1,\tau_2,\si_1,\si_2$ anti-invariant.
\end{thm}

\begin{proof}
By formula \eqref{factor}, the form $\Om$ is the wedge-product of four differentials. Hence
the bracket $\{\,,\,\}_{\Om}$ defines a Nambu-Poisson manifold structure on $\C^6$
of order 2. The invariance properties of it follow from  Lemmas \ref{lem 12.3} and \ref{lem OM}.
\end{proof}

\begin{lem}
\label{lem P coef}
The Poisson structure $\{\,,\,\}_{\Om}$ is log-canonical. The Poisson brackets 
$\{x_i,x_j\}_\Om$ are given by the following matrix,
 \[
\left( 
\begin{array}{cccccc}
 0 & 0 & -x_1 x_3 & x_1 x_4 & x_1 x_5 & -x_1 x_6 \\
 0 & 0 & x_2 x_3 & -x_2 x_4 & -x_2 x_5 & x_2 x_6 \\
 x_1 x_3 & -x_2 x_3 & 0 & 0 & -x_3 x_5 & x_3 x_6 \\
 -x_1 x_4 & x_2 x_4 & 0 & 0 & x_4 x_5 & -x_4 x_6 \\
 -x_1 x_5 & x_2 x_5 & x_3 x_5 & -x_4 x_5 & 0 & 0 \\
 x_1 x_6 & -x_2 x_6 & -x_3 x_6 & x_4 x_6 & 0 & 0 \\
\end{array}
\right).
\]

\end{lem}

\begin{proof}

 Explicit computations give
\begin{align*}
\Om\wedge dx_1\wedge dx_2&=0,\\
\Om\wedge dx_1\wedge dx_3&=-x_1 x_3\, dV,\\
\Om\wedge dx_1\wedge dx_4&=x_1 x_4\, dV,\\
\Om\wedge dx_1\wedge dx_5&=x_1 x_5\, dV,
\end{align*}
and so on.
\end{proof}

\begin{lem}
\label{lem Cas}
For any $f\in \C[\bs x]$ and $h\in \mc C$, we have
$\{f,h\}_\Om=0$. In particular, 
$\{f,H_1\}_\Om=\{f,H_2\}_\Om=0$ for any $f\in\C[\bs x]$.

\end{lem}

This statement justifies the name of the Casimir subalgebra for the subalgebra $\mc C\subset \C[\bs x]$.

\begin{proof} 
The equalities  $\{x_i,h_j\}_\Om =0$,  $i=1,\dots,6$, $j=1,\dots,5$,  are easily checked directly.
The statement also follows from the fact that the image of $F$ lies in $Y$ and $\om$ is a top degree form
on $Y$.
\end{proof}

\begin{lem}
The symplectic leaves of  the Poisson structure $\{\,,\,\}_\Om$ are at most two-dimensional
and lie in fibers of the map $F$.
\qed
\end{lem}

\subsection{Remarks}

\subsubsection{} The log-canonical Poisson structure $\{\,,\,\}_\Om$ can be encoded by the quiver in Figure \ref{quiver}.
\begin{figure}
\centering
\def\svgscale{0.7}
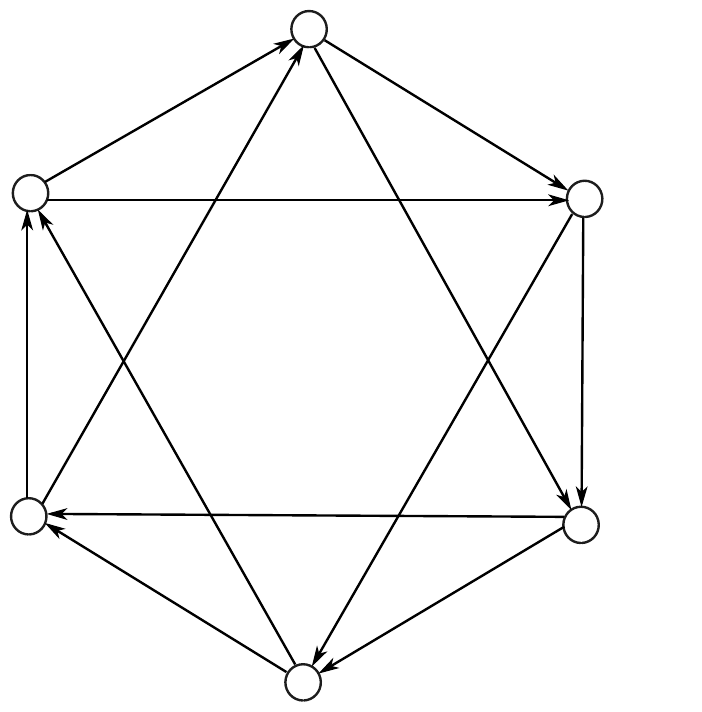
\caption{}\label{quiver}
\end{figure}
It would be interesting to determine if some of the $*$-Markov group transformations 
can be obtained as a sequence of mutations in the cluster algebra of that quiver. We were able to represent
in this way only
the action of the permutations $\si_1,\si_2$. To obtain 
$\si_1$ one needs to mutate the cluster variables at vertex 1, then
at vertex 3, then at vertex 1 and so on as in the sequence $1\,3\,1\,3\,1\,2\,4\,2\,4\,2$\ (10 mutations). The permutation $\si_2$
is obtained by the sequence of mutations $3\,5\,3\,5\,3\,4\,6\,4\,6\,4$. Cf. \cite{CS20}, where the braid group action 
was presented by   mutations for the $A_n$ quivers.

\subsubsection{}
We say that a Poisson structure on $\C^6$ is \emph{quadratic},
 if the $\{x_i,x_j\}$ are homogeneous quadratic polynomials in $\bm x$.

\begin{lem}
The Poisson structure $\{\,,\,\}_\Om$ is the {unique}, up to rescaling by a nonzero constant, 
nonzero quadratic Poisson structure on $\C^6$ having both $H_1$ and $H_2$ as Casimir elements.
\end{lem}

\proof

Let $\{x_i,x_j\}=\sum_{k,l}Q^{kl}_{ij}x_kx_l$ with unknown coefficients $Q_{ij}^{kl}\in\C$. Assume  the
skew-symmetry of $\{\cdot,\cdot\}$, and that both $H_1$ and $H_2$ are Casimir elements. 
This gives  a  system of linear equations 
for $Q_{ij}^{kl}$. A computer assistant calculation shows 
that the matrix of coefficients of that system  has rank 1, and the space of solutions is spanned by the Poisson tensor
$\{\,,\,\}_\Om$.
\endproof

\subsubsection{}

Let  $\{x_i,x_j\} =Q_{i,j}(\bs x)$ be a polynomial Poisson structure on $\C^6$ having  $H_1, H_2$ as Casimir elements.
Expand the coefficients $Q_{ij} (x)$ at the origin,
$Q_{ij}(x) = \sum_{k=1}^6 Q^k_{ij} x_k + \dots ,$ \ $Q_{ij}^k\in\C$. 
A computer assistant calculation shows that $Q^k_{i,j}=0$ for all $i,j,k$.

\subsubsection{}
A computer assistant calculation shows 
that the Poisson structure $\{\,,\,\}_\Om$ is the {unique}, up to rescaling by a nonzero constant, nonzero
log-canonical Poisson structure on $\C^6$, which remains to be log-canonical after the action on it  by any element of the
 braid group $\mc B_3$.
 
\subsection{Poisson structure on $\C^5$}

 Consider $\C^5$ with coordinates $\bm y=(y_1,\dots, y_5)$, and objects discussed in Section \ref{sec an pr}.

Consider on $\C^5$ the canonical Nambu bracket of order 5,
\bean
\label{our6}
\{f_1,\dots, f_5\} = \frac{ df_1\wedge df_2\wedge df_3\wedge df_4\wedge df_5}
{dW}\,,
\eean
and the associated  bracket
\bean
\label{BP5}
&&
\{f_1,f_2\}_{J,J_1,J_2}
= \frac {dJ\wedge dJ_1\wedge dJ_2\wedge df_1\wedge df_2}{dW}\,.
\eean

\begin{thm}
\label{thm NP2}

The bracket $\{\,,\,\}_{J,J_1,J_2}$ defines a Poisson structure 
on $\C^5$ with Casimir elements $J,J_1,J_2$.
The Poisson structure $\{\,,\,\}_{J,J_1,J_2}$ is $\la_{i,j}, \tau_1,\tau_2, \nu$ invariant and 
$\si_1,\si_2$ anti-invariant.
\end{thm}

\begin{proof}
The theorem follows from  Lemma \ref{lem 10.5}.
\end{proof}

Introduce new linear coordinates  on $\C^5$,
$(u_1,u_2,u_3,u_4,u_5):=(y_1,y_2,y_3, J_1,J_2)$.

\begin{lem}
\label{lem P5 coef}
The Poisson brackets 
$\{u_i,u_j\}_{J,J_1,J_2}$ are given by the  formulas:
\bea
&&
\{u_1,u_2\}_{J,J_1,J_2} =
u_1u_2 -2(u_1+u_2+u_3) +u_4+u_5,
\\
&&
\{u_2,u_3\}_{J,J_1,J_2} =
u_2u_3 -2(u_1+u_2+u_3) +u_4+u_5,
\\
&&
\{u_3,u_1\}_{J,J_1,J_2} =
u_3u_1 -2(u_1+u_2+u_3) +u_4+u_5,
\eea
and $\{u_i,u_4\}_{J,J_1,J_2}=\{u_i,u_5\}_{J,J_1,J_2}=0$ for all $i$.
\qed
\end{lem}

Notice  the similarity of these formulas with Dubrovin's formulas
\eqref{DuPo}. Similarly to Dubrovin's case the linear and quadratic parts
of the Poisson structure  $\{\,,\,\}_{J,J_1,J_2}$ form a  pencil of Poisson structures,
that, is any linear combination of them is a Poisson structure too.

\begin{rem}
It would be interesting to compare the Poisson structures of this Section \ref{sec NPo} with numerous examples in
\cite{OR02}.

\end{rem}

\section{$*$-Analog of Horowitz Theorem}
\label{sec *Hor}

In this section we discuss analogs of the Horowitz Theorem \ref{thm Hor-c}.

\subsection{Algebra $\mc R$ and $\nu$-endomorphisms}
\label{*endo}

Let $\bs x=(x_1,\dots,x_6)$.
Define an involution $\nu$ on the polynomial algebra $\mathcal R:=\Zs[\bm x]$
as follows.  For an element 
\[
f =\sum_{\bs a\in \N^6}
A_{\bs a}\ x_1^{a_1}
x_2^{a_2}x_3^{a_3}x_4^{a_4}x_5^{a_5}x_6^{a_6},
\qquad A_{\bm a}\in\Zs ,
\]
define
\[
\nu f =
\sum_{\bs a\in \N^6}
A_{\bs a}^*\ x_1^{a_2}
x_2^{a_1}x_3^{a_4}x_4^{a_3}x_5^{a_6}x_6^{a_5}.
\]
A $\Zs$-algebra endomorphism $\phi\colon \mathcal R\to\mathcal R$ is called a $\nu$-\emph{endomorphism} if
\[
\phi(\nu f)=\nu \phi(f),\qquad f\in \mathcal R.
\]
If $\phi$ is invertible, then $\phi$  is called a $\nu$-\emph{automorphism} of $\mc R$. 
The group of $\nu$-automorphisms of $\mc R$ is denoted by ${\rm Aut}_\nu(\mc R)$.

There is a one-to-one correspondence between $\nu$-endomorphisms of $\mc R$ and triples
 of polynomials $(P,Q,R)\in\mc R^3$. Such a triple defines a $\nu$-endomorphism
\bea
 &&x_1\mapsto P,\quad \ x_3\mapsto Q,\qquad x_5\mapsto R,\\
 && x_2\mapsto \nu P,\quad x_4\mapsto \nu Q,\quad x_6\mapsto \nu R.
\eea
In what follows we  often define a $\nu$-endomorphism by giving a triple $(P,Q,R)$.

\subsection{Markov group of $\nu$-automorphism}

Consider the following four groups of 
\\
$\nu$-automorphisms of $\mc R$:
\vskip 2mm
{\bf Type I. } The group $G_1^{\on{aut}}$ of $\nu$-automorphisms generated by
transformations
\beq
\notag
\La_{i,j}\colon x_1\mapsto (-1)^i x_1,\quad x_3\mapsto (-1)^{i+j} x_3,\quad x_5\mapsto (-1)^j x_5,
\qquad i,j\in\Z_2.
\eeq
\vskip 2mm
{\bf Type II. }
The group $G_2^{\on{aut}}$ of $\nu$-automorphisms generated by
transformations, 
\begin{align}
\notag
\Si_1\ \colon\ x_1\mapsto x_3,\quad x_3\mapsto x_1,\quad x_5\mapsto x_5,\\
\notag
\Si_2\ \colon \ x_1\mapsto x_1,\quad x_3\mapsto x_5,\quad x_5\mapsto x_3.
\end{align}
\vskip 2mm
{\bf Type III. }
The group $G_3^{\on{aut}}$ of $\nu$-automorphisms generated by
transformations
\bea
\notag
&&
{\rm T}_1\ \colon\ x_1\mapsto -x_2, \qquad 
\phantom{a}
x_3\mapsto x_6-x_1x_3, \qquad
x_5\mapsto x_4,
\\
\notag
&&
{\rm T}_2\ \colon\  x_1\mapsto x_4-x_1x_5, 
\qquad
x_3\mapsto x_2, \qquad
x_5\mapsto -x_6.
\eea
We have 
$\rm T_1 T_2T_1=T_2T_1T_2$.

\vskip 2mm
{\bf Type IV. }
The group $G_4^{\on{aut}}$ of $\nu$-automorphisms generated by
transformations
\beq
\notag
{\rm M}_{i,j}\colon x_1\mapsto s_3^{-i}x_1,\quad x_3\mapsto s_3^{i+j}x_3,\quad x_5\mapsto s_3^{-j}x_5,
\qquad i,j\in\Z.
\eeq

Define the $*$-\emph{Markov group} $\Ga_M^{\on{aut}}$ of $\nu$-automorphisms of $\mc R$
as the group generated by $G_1^{\on{aut}},G_2^{\on{aut}},G_3^{\on{aut}},G_4^{\on{aut}}$.

Define the \emph{Vi\`ete $\nu$-involutions}
$V_1,V_2,V_3\in \Ga_M^{\on{aut}}$  by the formulas
\begin{align}
\notag
V_1\colon &&x_1\mapsto x_3x_5-x_2, &&x_3\mapsto x_4, &&x_5\mapsto x_6,&\\
\notag
V_2\colon &&x_1\mapsto x_2, &&x_3\mapsto x_1x_5-x_4, &&x_5\mapsto x_6,&\\
\notag
V_3\colon &&x_1\mapsto x_2, &&x_3\mapsto x_4, &&x_5\mapsto x_1x_3-x_6.&
\end{align}
We denote by $\Ga_V^{\on{aut}}$ the group generated by $V_1,V_2,V_3$.
We have
\begin{align}
\notag
V_1=\La_{1,1}\ \Si_1\ {\rm T}_2,\qquad
\notag
V_2=\La_{1,0}\ \Si_2\ {\rm T}_1,\qquad
\notag
V_3=\La_{1,1}\ {\rm T}_1\ \Si_2.
\end{align}

\begin{thm}\label{identaut}
We have the following identities:
\begin{align}
\notag
\Si_1\La_{i,j}\Si_1&=\La_{i+j,j}\,,
\\
\notag
\Si_2\La_{i,j}\Si_2&=\La_{i, i+j}\,,
\end{align}
\begin{alignat}{2}
\notag
\Si_1 V_1\Si_1&=V_2,\qquad \Si_2 V_1\Si_2&=V_1,\\
\notag
\Si_1 V_2\Si_1&= V_1,\qquad \Si_2 V_2\Si_2&=V_3,\\
\notag
\Si_1 V_3\Si_1&=V_3,\qquad \Si_2 V_3\Si_2&=V_2,
\end{alignat}
\beq
\notag
\La_{k,l} V_i\La_{k,l} =V_i\,,\qquad i=1,2,3,
\eeq
\bea
&&
\La_{k,l}{\rm M}_{i,j}\La_{k,l}={\rm M}_{i,j}\,,\qquad
\phantom{a}
 k,l\in\Z_2,\quad i,j\in\Z,
\\
\notag
&&
\Si_1{\rm M}_{i,j}\Si_1={\rm M}_{-i-j,j}\,,\qquad \Si_2{\rm M}_{i,j}\Si_2={\rm M}_{i,-i-j}\,,
\\
\notag
&&
V_k{\rm M}_{i,j}V_k
={\rm M}_{-i,-j}\,,\qquad
\phantom{a}
 k=1,2,3.
\eea

\end{thm}

\proof
These identities are proved by straightforward computations.
\endproof

\begin{thm}
We have an epimorphism of groups $\iota\colon\smg\to\amg$ defined on generators by
\begin{align}
\label{iotagen}
\iota(\la_{i,j}):=\La_{i,j},\quad \iota(\sigma_i):=\Si_i,
\quad \iota(v_j):=V_j,\quad \iota(\mu_{i,j}):={\rm M}_{i,j}.
\end{align}
\end{thm}
\proof
Let us show that the morphism $\iota\colon\smg\to\amg$ is well defined. First, 
notice that \eqref{iotagen} uniquely extends to a group morphism on each of the groups $G_1, G_2, G_4, \svg$. Any $g\in\smg$ admits a unique decomposition $g=vg_4g_1g_2$ with $v\in\svg$, $g_1\in G_1,\ g_2\in G_2,\ g_4\in G_4$, by Corollaries \ref{dec124} and \ref{decv124}. We define
\[\iota(g):=\iota(v)\iota(g_4)\iota(g_1)\iota(g_2).
\]Given $\tilde g\in\smg$, we have to show that $\iota(g\tilde g)=\iota(g)\iota(\tilde g)$. We have 
\begin{align*}g\tilde g=vg_4g_1g_2\tilde v\tilde g_4\tilde g_1\tilde g_2
=v\tilde v'g_4\tilde g_4'g_1\tilde g_1'g_2\tilde g_2,
\end{align*}
where in the second line we use the commutations relations of Proposition \ref{ident}. The map $\iota$ preserves the commutations relations among the generators $\La_{\al,\bt},$ $\Si_i,{\rm V}_j,{\rm M}_{\al,\bt}$, by Theorem \ref{identaut}. So, we have
\begin{align*}
\iota(g\tilde g)&=\iota(v\tilde v')\iota(g_4\tilde g_4')\iota(g_1\tilde g_1')\iota(g_2\tilde g_2)\\
&=\iota(v)\iota(\tilde v')\iota(g_4)\iota(\tilde g_4')\iota(g_1)\iota(\tilde g_1')\iota(g_2)\iota(\tilde g_2)\\
&=\iota(v)\iota(g_4)\iota(g_1)\iota(g_2)\iota(\tilde v)\iota(\tilde g_4)\iota(\tilde g_1)\iota(\tilde g_2)\\
&=\iota(g)\iota(\tilde g).
\end{align*}
This completes the proof.
\endproof

Let $P(\bm x)\in\mc R$. Define the map $\widehat P\colon \left(\Zs\right)^3\to \Zs$
by the formula
\[
\widehat P\colon (f_1,f_2,f_3)\mapsto P(f_1,f_1^*,f_2,f_2^*,f_3,f_3^*).
\]

\begin{prop}\label{iota}
For any $P\in\mc R$,  $g\in\smg$, $ (f_1,f_2,f_3)\in\left(\Zs\right)^3$, we have
\beq
\notag
\widehat{\iota(g)P}(f_1,f_2,f_3)=\widehat P\left(g^{-1}(f_1,f_2,f_3)\right).
\eeq
\end{prop}
\proof
This is easily checked on the generators of $\smg$.
\endproof

\begin{prop}
The morphism $\iota$ is an isomorphism.
\end{prop}

\proof
Consider the polynomials $x_1,x_3,x_5\in\mc R$. They define the natural projections
\begin{align*}
\widehat x_1\colon \left(\Zs\right)^3\to \Zs,\quad (f_1,f_2,f_3)\mapsto f_1,\\
\widehat x_3\colon \left(\Zs\right)^3\to \Zs,\quad (f_1,f_2,f_3)\mapsto f_2,\\
\widehat x_5\colon \left(\Zs\right)^3\to \Zs,\quad (f_1,f_2,f_3)\mapsto f_3.
\end{align*}
Let $g\in\ker \iota$. By Proposition \ref{iota}, we have 
\[
\widehat x_1(g^{-1}(f_1,f_2,f_3))=f_1,\quad 	\widehat x_3(g^{-1}(f_1,f_2,f_3))=f_2,\quad \widehat x_5(g^{-1}(f_1,f_2,f_3))=f_3.
\]
Hence $g^{-1}=\rm id$. 
\endproof

\subsection{$\nu$-Endomorphisms of maximal rank} Let $\phi\colon\mc R\to\mc R$ 
be a $\nu$-endomorphism, defined by a triple $ P,Q,R\in\mc R$,
\[
x_1\mapsto P,\quad x_3\mapsto Q,\quad x_5\mapsto R,
\quad x_2\mapsto \nu P,\quad x_4\mapsto \nu Q,\quad x_6\mapsto \nu R.
\]
For any fixed $\bm p\in\C^3$, denote by $P_{\bm p}, Q_{\bm p}, R_{\bm p},\nu P_{\bm p}, \nu Q_{\bm p}, \nu R_{\bm p}\in\C[\bm x]$ the specialization of $P,Q,R,\nu P,\nu Q,\nu R$ at $\bm z=\bm p$. 
\vskip2mm
For any fixed $\bm p\in\C^3$, we have a polynomial map $\phi_{\bm p}\colon \C^6\to\C^6$ defined by
\[\bm q\mapsto 
\left(
P_{\bm p}(\bm q),\quad
\nu P_{\bm p}(\bm q),\quad
Q_{\bm p}(\bm q),\quad
\nu Q_{\bm p}(\bm q),\quad
R_{\bm p}(\bm q),\quad
\nu R_{\bm p}(\bm q)
\right).
\]

\vsk.2>
A $\nu$-endomorphism $\phi$ of $\mc R$ is said to be \emph{of maximal rank}
 if there exist $\bm p\in\C^3$ and $\bm q\in\C^6$ such that the Jacobian matrix 
 of $\phi_{\bs p}$ at the point $\bs q$  is invertible.

\subsection{Horowitz type theorem  for  $*$-Markov group}

Define the $\nu$-Horowitz group $\GH$ as the group of $\nu$-automorphisms of $\mathcal R$ which preserve the polynomial
\[
H=x_1x_2+x_3x_4+x_5x_6-x_1x_3x_5.
\]
Define $\Gamma^{\rm max}$ to be the set of $\nu$-endomorphisms of
$\mc R$ of maximal rank,  which preserve the polynomial $H$.

\vsk.2>
We have $\amg\subseteq \GH\subseteq \Gamma^{\rm max}$.

\begin{thm}\label{thmhor}
We have $\amg= \GH= \Gamma^{\rm max}$. In particular, any element of $\Gamma^{\rm max}$ is a $\nu$-automorphism.
\end{thm}

\proof
It is sufficient to prove that $\Gamma^{\rm max}\subseteq \amg$. The proof is an adaptation of the original argument of \cite[Theorem 2]{H}.
Let 
\beq\label{auto}
x_1\mapsto P,\quad x_3\mapsto Q,\quad x_5\mapsto R,\quad x_2\mapsto \nu P,\quad x_4\mapsto \nu Q,\quad x_6\mapsto \nu R,
\eeq
be an element of $\gmax$, where $P,Q,R\in \mathcal R$. Up to an action of $\autG{2}$, we can 
assume that the total degrees of $P,Q,R$ with respect to $x_1,x_2,x_3,x_4,x_5,x_6$ are in ascending order, i.e.
\beq
\notag
{\rm Deg}\, P\leq {\rm Deg}\, Q\leq {\rm Deg}\, R.
\eeq
Set
\begin{align}
\label{formauto}
P&=P_p+P_{p-1}+\dots + P_0,\\
\notag
Q&=Q_q+Q_{q-1}+\dots +Q_0,\\
\nonumber R&=R_r+R_{r-1}+\dots+ R_0,
\end{align}
where $P_k,Q_k,R_k$ are homogeneous polynomials in $x_1,x_2,x_3,x_4,x_5,x_6$ of degree $k$. Necessarily, we must have $p,q,r\geq 1$, otherwise \eqref{auto} does not define an endomorphism of maximal rank. 
Since \eqref{auto} is an element of $\gmax$, we have
\beq\label{inv}
P\cdot \nu P+Q\cdot\nu Q+R\cdot\nu R-PQR=x_1x_2+x_3x_4+x_5x_6-x_1x_3x_5.
\eeq

Suppose that $p=q=r=1$. By comparison of the highest degree terms of the l.h.s. and r.h.s. 
of \eqref{inv}, we deduce that $P_1Q_1R_1=x_1x_3x_5$. Since $x_1,x_3,x_5$ are irreducible,
 unique factorization implies that up to reordering of $P_1,Q_1,R_1$ we have
\beq\label{base}
P_1=\gamma_1 x_1,\quad Q_1=\gamma_3 x_3,\quad R_1=\gamma_5 x_5,
\eeq
 where $\ga_1,\ga_3,\ga_5\in\Zs$ and 
 $\gamma_1\gamma_3\gamma_5=1$. Hence each of $\ga_j$  is of the form $\pm s_3^{a_j}$.
 
If we substitute \eqref{base} in \eqref{formauto}, and expand \eqref{inv}, we deduce that $P_0=Q_0=R_0=0$ 
(since the r.h.s. of \eqref{inv} has no terms $x_1x_3, \,x_1x_5,\,x_3x_5$). 
Thus, the only possible form of \eqref{auto} is
\bea
\notag
&&
x_1\mapsto (-1)^i s_3^{a_3}x_1, \ \
\quad x_3\mapsto (-1)^{i+j}s_3^{a_3}x_3,\quad 
\ 
x_5\mapsto (-1)^j s_3^{a_5}x_5,
\\
&&
x_2 \mapsto (-1)^i s_3^{-a_1}x_2, 
\quad x_4\mapsto (-1)^{i+j}s_3^{-a_3}x_4,\quad x_6\mapsto (-1)^j s_3^{-a_5}x_6,
\eea
where $i,j\in\mathbb Z_2$ and $a_1,a_2,a_3\in\Z$, $a_1+a_2+a_3=0$. 
All these transformations are in $\langle\autG{1},\autG{4}\rangle\subseteq\amg$.

Now we proceed by induction on the maximum $r$ of the degrees of $P,Q,R$. If we expand \eqref{inv} using \eqref{formauto}, we obtain
\beq\label{inv2}
P_p\cdot\nu P_p+Q_q\cdot\nu Q_q+R_r\cdot\nu R_r-P_pQ_qR_r+\dots=x_1x_2+x_3x_4+x_5x_6-x_1x_3x_5,
\eeq
where the dots denote lower degree terms. The term $P_pQ_qR_r$ is 
of degree at least 4. The degree of every term of the r.h.s. of \eqref{inv2}  is less than 4.
Hence $P_pQ_qR_r$  must cancel with another term of the l.h.s.
This is possible  if and only if $r=p+q$. 
If $r=p+q$, then the terms of highest degree are 
$R_r\cdot\nu R_r$ and $P_pQ_qR_r$, and we must have
$R_r\cdot\nu R_r-P_pQ_qR_r=0$. Thus, 
\beq\label{id1}
\nu R_r=P_pQ_q.
\eeq
The transformation
\beq\label{auto2}
x_1\mapsto \nu P,\quad x_3\mapsto \nu Q,\quad x_5\mapsto PQ-\nu R,
\eeq
extends to a $\nu$-endomorphism of $\mathcal R$. Such an endomorphism 
is the composition of \eqref{auto} with $V_3\in\amg$. 
The endomorphism \eqref{auto2} has the highest degree less than $r$, because of \eqref{id1}. 
 Hence, by induction hypothesis, the endomorphism \eqref{auto2} is a $\nu$-automorphism in $\amg$. 
 This completes the proof.
\endproof

\subsection{Horowitz type theorem for $\C^5$}

Recall the action of the $*$-Markov group on $\C^5$ with coordinates $(y_1,\dots,y_5)$.
In particular, the $*$-Vi\`ete involutions act on $\C^5$ by the formulas
\bea
&&
v_1: (y_1,\dots,y_5) \mapsto (y_1+y_2y_3-y_4-y_5,y_2,y_3, -y_5+y_2y_3,-y_4+y_2y_3),
\\
&&
v_2: (y_1,\dots,y_5) \mapsto (y_1,y_2+y_1y_3-y_4-y_5,y_3, -y_5+y_1y_3,-y_4+y_1y_3),
\\
&&
v_2: (y_1,\dots,y_5) \mapsto (y_1,y_2, y_3+y_1y_2-y_4-y_5, -y_5+y_1y_2,-y_4+y_1y_2).
\eea

\begin{thm}
\label{thm H5}

Let $\psi :\C^5\to\C^5$ be a maximal rank polynomial map,  which
preserves the polynomials
\bea
J = y_1y_2y_3-y_4y_5,
\qquad
J_1 =y_1+y_2+y_3-y_4,
\qquad
J_2 =y_1+y_2+y_3-y_5. 
 \eea
 Then $\psi$ is invertible and lies in the image of the $*$-Markov group.

\end{thm}

\begin{rem}
One can easily add the parameters $\Zs$ and reformulate Theorem \ref{thm H5} similarly to
Theorem \ref{thmhor}.

\end{rem}

\begin{cor}
If the map  $\psi$ satisfies the assumptions of Theorem \ref{thm H5}, then it
commutes with the involution 
\bea
\nu :(y_1,\dots,y_5)
\mapsto (y_1,y_2,y_3,y_5,y_4).
\eea
\end{cor}

\vsk.2>
\noindent
{\it Proof of Theorem \ref{thm H5}.}
Let $\psi$ send $(y_1,\dots,y_5)$ to  $(P_1,\dots, P_5)$. Then
\bea
P_1P_2P_3-P_4P_5 = y_1y_2y_3-y_4y_5.
\eea
\bea
P_4-P_1-P_2-P_3 = y_4- y_1-y_2-y_3,
\qquad P_5-P_1-P_2-P_3 = y_5- y_1-y_2-y_3.
\eea
Hence
\bea
&&
P_4  = y_4-y_1-y_2-y_3 +P_1+P_2+P_3,
\\
&&
P_5 =
y_5-y_1-y_2-y_3 +P_1+P_2+P_3,
\eea
and the map $\psi$ is completely determined by the three polynomials $P_1,P_2,P_3$.

\vsk.2>

First assume that $\psi$ is a linear map, $P_i= P_{i,0}+a_i$, 
where $a_i\in \C$ and
$P_{i,0}$ are  homogeneous  polynomials in $\bs y$ of degree 1.
 Then
\bea
y_1y_2y_3 =   P_{1,0}P_{2,0}P_{3,0}.
\eea
Hence after a permutation of $P_1,P_2,P_3$  we will have
\bea
P_i = b_iy_i + a_i, \qquad i=1,2,3, \quad b_i\in\C, \ \ b_1b_2b_3=1.
 \eea
We have
\bea
P_4  
&=& 
y_4-y_1-y_2-y_3 +P_1+P_2+P_3 
\\
&&
= y_4
+(b_1-1)y_1+(b_2-1)y_2 + (b_3-1)y_3
+a_1+a_2+a_3,
\\
P_5 
& = &
y_5-y_1-y_2-y_3 +P_1+P_2+P_3 
\\
&&
= y_4
+(b_1-1)y_1+(b_2-1)y_2 + (b_3-1)y_3
+a_1+a_2+a_3,
\eea
Hence
\bea
&&
\big(y_4
+(b_1-1)y_1+(b_2-1)y_2 + (b_3-1)y_3
+a_1+a_2+a_3\big)
\\
&&
\times
\big(y_5
+(b_1-1)y_1+(b_2-1)y_2 + (b_3-1)y_3
+a_1+a_2+a_3\big)
\\
&&
-(b_1y_1+a_1)
(b_2y_2+a_2)(b_3y_3+a_3)= y_4y_5-y_1y_2y_3.
\eea
This means that $b_i=1$ for all $i$ and hence
\bea
&&
\big(y_4
+a_1+a_2+a_3\big)
\big(y_5
+a_1+a_2+a_3\big)
\\
&&
\phantom{aaa}
-(y_1+a_1)
(y_2+a_2)(y_3+a_3)= y_4y_5-y_1y_2y_3.
\eea
This implies that $a_i=0$ for all $i$.

\vsk.2>
Equation
$P_4P_5-P_1P_2P_3 = y_4y_5-y_1y_2y_3$
can be rewritten as
\bean
\label{eQ}
&&
(y_4-y_1-y_2-y_3 +P_1+P_2+P_3)
(y_5-y_1-y_2-y_3 +P_1+P_2+P_3)
- P_1P_2P_3
\phantom{aaaaaaaaaaaaa}
\\
\notag
&&
\phantom{aaaaaaaaaaaaaaaaaaaaaaaa}
=y_4y_5-y_1y_2y_3.
\eean
Let  us write
$P_i = P_{i,1}+\dots$,\ $i=1,2,3$,
where $P_{i,1}$ is the top degree homogeneous component of $P_i$.
Denote $d_i$ the degree of $P_i$.
Since $\psi$ is of maximal rank, we have $d_i>0$.

\vsk.2>

Assume that the maximum of $d_1,d_2,d_3$ is greater than 1. 
After a permutation of the first three coordinates we may assume that $d_1\geq d_2\ge d_3$.
Then equation \eqref{eQ} implies that $d_1=d_2+d_3$ and there are exactly two terms of degree
$2d_1$ which have to cancel,  
\bea
P_{1,1}P_{2,1}P_{3,1} - P_{1,1}^2 = P_{1,1}(P_{2,1}P_{3,1} - P_{1,1})=0.
\eea
Let us compose $\psi$ with involution $v_1$. Then $v_1\circ \psi$ sends
$(y_1,\dots,y_5)$ to $(\tilde P_1,\dots,\tilde P_5)$, where $\tilde P_2=P_2$, $\tilde P_3=P_3$,
\bea
&&
\tilde P_1 =
P_1+P_2P_3-P_4-P_5
\\
&&
\phantom{aa}
= P_1 + P_2 P_3 - (y_4-y_1-y_2-y_3 +P_1+P_2+P_3)
\\
&&
\phantom{aaaa}
-(y_5-y_1-y_2-y_3 +P_1+P_2+P_3)
\\
&&
\phantom{aaaaaa}
= P_2P_3 - P_1 -2P_2-2P_3 -y_4-y_5 + 2y_1+2y_3+2y_3,
\eea
\bea
&&
\tilde P_4 =
-(y_5-y_1-y_2-y_3 +P_1+P_2+P_3) + P_2P_3,
\\
&&
\tilde P_5 =
-(y_4-y_1-y_2-y_3 +P_1+P_2+P_3) + P_2P_3,
\eea
These formulas show that $\deg \tilde P_1 <\deg P_1$, 
while $\deg \tilde P_i=\deg P_i$ for $i=2,3$, 
and the theorem 
follows from the iteration of this procedure.
\qed

\appendix

\section{Horowitz type theorems}
\label{appendix}

\subsection{Classical setting}
Let $a_0,\dots,a_n\in\Z$ be non-zero integers such that 
\bea
a_j\ \ \on{divides}\ \  a_0,\quad \on{for}\ \ j=1,\dots,n.
\eea
Consider the polynomial in $n$ variables $x_1,\dots,x_n$,
\bea
H:=\sum_{j=1}^n a_j x_j^2 - a_0\prod_{j=1}^n x_j \,.
\eea
The polynomial $H$ is quadratic with respect to each variable $x_j$. This
 ensures that the polynomial has a nontrivial group of symmetries.

\begin{thm}
\label{thmperm}
Let $\si\in\frak S_n$ be such that $(a_{\si(1)},\dots, a_{\si(n)})=(a_1,\dots,a_n)$. Then the permutation $(x_1,\dots, x_n)\mapsto (x_{\si(1)},\dots, x_{\si(n)})$  preserves the polynomial $H$.
\qed
\end{thm}

\begin{thm}
\label{rgm A.2}
For any $i=1,\dots,n$, the transformation
\bea
v_i : (x_1,\dots,x_n) \mapsto \left(x_1, \dots,x_{i-1}, -x_i + \frac{a_0}{a_i}\prod_{j\ne i} x_j, x_{i+1},\dots,x_n\right)
\eea
is an involution preserving  the polynomial $H$.
\end{thm} 

\proof We check this for $v_1$. Denote $y=\frac{a_0}{a_1}\prod_{j>1} x_j$.
Then
\begin{align*}
&
H(-x_1 + y, x_2,\dots,x_n) =
a_1(-x_1 + y)^2 -a_0 (-x_1+y) \prod_{j>1} x_j + \sum_{l>1}x_l^2
\\
&
= a_1x_1^2 - 2x_1 a_0\prod_{j> 1} x_j + \frac{a_0^2}{a_1}\prod_{j> 1} x_j^2
+
a_0 x_1\prod_{j>1} x_j - \frac{a_0^2}{a_1} \prod_{j>1} x_j ^2 + \sum_{l>1}x_l^2
\\
&
\puqed
= H(x_1, x_2,\dots,x_n).
\phantom{aaaaaaaaaaaaaaaaaaaaaaaaaaaaaaaaaaaaaaa}
\qedhere
\poqed
\end{align*}

The permutations of Theorem \ref{thmperm} and the Vi\`ete maps of Theorem \ref{rgm A.2}
are automorphisms of the algebra $\Z[\bm x]$.

We say that an endomorphism of algebras $\phi\colon \Z[\bm x]\to \Z[\bm x]$ defined by 
\[x_j\mapsto P_j(\bm x),\quad P_j\in\Z[\bm x],\quad j=1,\dots,n,
\]is \emph{of maximal rank} if there exists a point $\bm q\in\C^n$ such that the Jacobian matrix 
of $\phi$ at $\bs q$  is invertible.

The following is a stronger version of the original Horowitz Theorem.
\begin{thm} Any endomorphism of maximal rank preserving the polynomial $H$ is an automorphism.
The group of all automorphisms of $\Z[\bm x]$ preserving $H$ is generated by the Vi\`ete transformations,
by the permutation of variables preserving the $n$-tuple $(a_1,\dots,a_n)$, and by multiplication by $-1$
of an even number of variables.
\end{thm}

\proof
The argument is the same of the proof of Theorem \ref{thmhor}.
\endproof

\subsection{$*$-Setting}Let $m,n$ be two positive integers.
Let $a_0, \dots, a_n$ be symmetric Laurent polynomials in $z_1,\dots,z_m$ with integer coefficients and such that
$$
a_j^* = a_j,\quad\text{and}\quad a_j\ \on{divides}\  a_0\quad \on{for}\  j=1,\dots,n,
$$ 
in the algebra $\Z[\bm z^{\pm 1}]^{\frak S_m}$ of symmetric Laurent polynomials.
\vsk.2>

Consider the polynomial in $2n$ variables $x_1,x_2\dots,x_{2n-1},x_{2n}$,
\bea
H :=\sum_{j=1}^n a_jx_{2j-1}x_{2j} - a_0\prod_{j=1}^n x_{2j-1} \,.
\eea
The algebra $\Z[\bm z^{\pm 1}]^{\frak S_m}[\bm x]$ admits an involution 
\[\nu\colon x_{2j-1}\mapsto x_{2j},\quad x_{2j}\mapsto x_{2j-1},\qquad j=1,\dots,n.
\]The notions of a $\nu$-endomorphism and a $\nu$-automorphism given in Section \ref{*endo} 
obviously extend to the algebra $\Z[\bm z^{\pm 1}]^{\frak S_m}[\bm x]$.

\vskip2mm
The polynomial $H$ has a nontrivial group of symmetries. 
\begin{thm}
\label{thmperm*}
Let $\si\in\frak S_n$ be such that $(a_{\si(1)},\dots, a_{\si(n)})=(a_1,\dots,a_n)$. 
Then the permutation 
$\bm x\mapsto (x_{2\si(1)-1},x_{2\si(1)}\dots, x_{2\si(n)-1},x_{2\si(n)})$  preserves the polynomial $H$.
\qed
\end{thm}

\begin{thm}
\label{thm A.6}
For any $i=1,\dots,n$, the transformation $v_i$ defined by
\bea
&&
x_1\mapsto x_2,\qquad x_2\mapsto x_1, \quad \dots\ ,
\quad
x_{2i-3}\mapsto x_{2i-2},\qquad x_{2i-2}\mapsto x_{2i-3},
\\
&&
x_{2i-1}\mapsto -x_{2i} + \frac{a_0}{a_i}\prod_{j\ne i} x_{2j-1},\qquad \ \
x_{2i}\mapsto-x_{2i-1} + \frac{a_0}{a_i} \prod_{j\ne i} x_{2j},
\\
&&
x_{2i+1}\mapsto x_{2i+2},\quad x_{2i+2}\mapsto x_{2i+1},
\quad\dots\ , \quad
x_{2n-1}\mapsto x_{2n},\quad
x_{2n}\mapsto x_{2n-1},
\eea
is an involution preserving the polynomial $H$.

\end{thm}

\begin{proof} 
The proof is by straightforward calculation, the same as for $n=3$.
\end{proof}

The permutations of Theorem \ref{thmperm*} and the Vi\`ete maps of Theorem \ref{thm A.6}
are $\nu$-automorphisms of the algebra $\Z[\bm z^{\pm 1}]^{\frak S_m}[\bm x]$.
\vskip2mm
Given $P\in \Z[\bm z^{\pm 1}]^{\frak S_m}[\bm x]$ and $\bm p\in\C^m$
 denotes by $P_{\bm p}\in\C[\bm x]$ the specialization of $P$ at $\bm z=\bm p$. 
\vskip2mm
Let  $\phi\colon \Z[\bm z^{\pm 1}]^{\frak S_m}[\bm x]\to \Z[\bm z^{\pm 1}]^{\frak S_m}[\bm x]$
 be a $\nu$-endomorphism defined by 
 $$
 x_{2j-1}\mapsto P_j(\bm x),\quad x_{2j}\mapsto \nu P_j(\bm x),\quad j=1,\dots, n.
 $$
For any $\bm p\in\C^m$ there is a map $\phi_{\bm p}\colon\C^{2n}\to\C^{2n}$ defined by
\[\bm x\mapsto (P_{1,\bm p}(\bm x),\nu P_{1,\bm p}(\bm x), \dots,P_{n,\bm p}(\bm x),\nu P_{n,\bm p}(\bm x)).
\]

The $\nu$-endomorphism $\phi\colon \Z[\bm z^{\pm 1}]^{\frak S_m}[\bm x]\to \Z[\bm z^{\pm 1}]^{\frak S_m}[\bm x]$
is said to be \emph{of maximal rank} if there exist a point $\bm p\in\C^m$ and a point $\bm q\in\C^{2n}$
 such that the Jacobian matrix of  $ \phi_{\bm p}$ at $q$ is invertible.

\begin{thm} 
Any $\nu$-endomorphism of maximal rank preserving $H$ 
is a $\nu$-automorphism. The group of all $\nu$-automorphisms of
 $\Z[\bm z^{\pm1}]^{\frak S_m}[\bm x]$ preserving 
$H$ is generated by the Vi\`ete transformations of Theorem \ref{thm A.6},
by the permutation of variables preserving $(a_1,\dots, a_n)$, by multiplication by 
$-1$ of an even number of variables,
 and by multiplication of variables by powers of
$s_m:=\prod_{j=1}^mz_j$.
\end{thm}

\begin{proof}
The proof is the same as for $n=3$.
\end{proof}

\section{$*$-Equations for $\Bbb P^3$ and Poisson structures}
\label{app B}

\subsection{$*$-Equations for $\Bbb P^3$}

As we wrote in the Introduction, a $T$-full exceptional collection $(E_1,E_2,E_3)$ in
$\D^b_T(\Bbb P^2)$
has the matrix 
  $(\chi_T(E_i^*\otimes E_j))$ of equivariant Euler characteristics of the form
$\left(\begin{array}{cccccc}
 1 & a & b
 \\
 0 & 1 & c
 \\
 0&0&1
 \end{array}
\right)$, 
  where $(a,b,c)$ are symmetric  Laurent polynomials 
  in the equivariant parameters $z_1,z_2,z_3$ satisfying the
  $*$-Markov equation 
\bean
\label{mma}
aa^*+bb^*+cc^*-ab^*c = 3-\frac{z_1^3+z_2^3+z_3^3}{z_1z_2z_3}.
\eean
Similar objects and equations are available for any projective space $\Bbb P^n$.
For example, for $\Bbb P^3$ the matrix  of equivariant Euler characteristics has the form
$\begin{pmatrix}
1&a&b&c\\
0&1&d&e\\
0&0&1&f\\
0&0&0&1
\end{pmatrix}$, where $(a,b,c,d,e,f)$
are symmetric  Laurent polynomials 
  in the equivariant parameters $z_1,z_2,z_3,z_4$ satisfying the
 system of equations 
\begin{align}\nonumber
aa^*&+bb^*+cc^*+dd^*+ee^*+ff^*\\
\label{mark4n1}
&-a^*bd^*-a^*ce^*-b^*cf^*-d^*ef^*+a^*cd^*f^*\\
\nonumber
&=4+\frac{z_2z_3z_4}{z_1^3}+\frac{z_1z_3z_4}{z_2^3}+\frac{z_1z_2z_4}{z_3^3}+\frac{z_1z_2z_3}{z_4^3},
\end{align}
\begin{align}
\nonumber -2 a a^*&-2 b b^*-2 c c^*-2 d d^*-2 e e^*-2 f f^*\\
\nonumber &+a b^* d+a^* b d^*+a c^* e+a^* c e^*+b^* c f^*+b c^* f+d e^* f+d^* e f^*\\
\nonumber&-a b^* e f^*-a^* b e^* f-b c^* d^* e-b^* c d e^*\\
\label{mark4n2}
&+a a^* f f^*+b b^* e e^*+c c^* d d^*\\
\nonumber&=-6+\frac{z_2^2 z_4^2}{z_1^2 z_3^2}+\frac{z_2^2 z_3^2}{z_1^2 z_4^2}+
\frac{z_1^2 z_2^2}{z_3^2 z_4^2}+\frac{z_3^2 z_4^2}{z_1^2 z_2^2}+\frac{z_1^2 z_4^2}{z_2^2 z_3^2}+
\frac{z_1^2 z_3^2}{z_2^2 z_4^2},
\end{align}
\begin{align}
\nonumber aa^*&+bb^*+cc^*+dd^*+ee^*+ff^*\\
\label{mark4n3}
&-ab^*d-ac^*e-bc^*f-de^*f+ac^*df\\
\nonumber&=4+\frac{z_1^3}{z_2 z_3 z_4}+\frac{z_2^3}{z_1 z_3 z_4}+\frac{z_3^3}{z_1 z_2 z_4}+\frac{z_4^3}{z_1 z_2 z_3},
\end{align}
see \cite[Formulas (3.24)-(3.26)]{CV}.  One may study this system of equations similarly to our study of the $*$-Markov equation.

In this appendix we briefly discuss the analogs for the system of equations
\eqref{mark4n1}-\eqref{mark4n3}  of the Poisson structure on $\C^6$ constructed  in  Section \ref{sec NPo}.  It
 will be a family of Poisson structures on $\C^{12}$.

\subsection{Poisson structures on $\C^{12}$}

Consider $\C^{12}$ with coordinates $\bm x=(x_1,\dots, x_{12})$, involution
\[\nu\colon\C^{12}\to\C^{12},\qquad x_{2j-1}\mapsto x_{2j},\qquad 
x_{2j}\mapsto x_{2j-1},
\qquad
j=1,\dots,6,
\]
and polynomials
\bea
H_1(\bm x)=&&x_1 x_2+x_3 x_4+x_5 x_6+x_7 x_8+x_9 x_{10}+x_{11} x_{12}
\\
&&-x_2x_3 x_8 -x_2x_5 x_{10} -x_4 x_5 x_{12}-x_8 x_9 x_{12}+x_2x_5 x_8 x_{12},\\
H_2(\bm x)=&&-2 x_1 x_2-2 x_3 x_4-2 x_5 x_6-2 x_7 x_8-2 x_9 x_{10}-2 x_{11} x_{12}\\
&&+x_3 x_8 x_2+x_5 x_{10} x_2+x_1 x_4 x_7+x_1 x_6 x_9+x_3 x_6 x_{11}\\
&&+x_7 x_{10} x_{11}+x_4 x_5 x_{12}+x_8 x_9 x_{12}-x_3 x_{10} x_{11} x_2+x_1 x_{11} x_{12} x_2\\
&&+x_5 x_6 x_7 x_8-x_3 x_6 x_8 x_9-x_4 x_5 x_7 x_{10}\\
&&+x_3 x_4 x_9 x_{10}-x_1 x_4 x_9 x_{12},\\
H_3(\bm x)=&&x_1 x_2+x_3 x_4+x_5 x_6+x_7 x_8+x_9 x_{10}+x_{11} x_{12}\\
&&-x_1 x_4 x_7-x_1 x_6 x_9-x_3 x_6 x_{11}-x_7 x_{10} x_{11}+x_1 x_6 x_7 x_{11}.
\eea
We have  $\nu^\star H_1=H_3$, $\nu^\star H_2=H_2$,  $\nu^\star H_3=H_1$.
\vsk.2>

Consider the braid group $\mc B_4$ with standard generators $\tau_1,\tau_2,\tau_3$.
The group $\mc B_4$ acts on $\C^{12}$,
\begin{align*}
\tau_1^\star x_1&=-x_2,& \tau_2^\star x_1&=x_3-x_1 x_7,& \tau_3^\star x_1&=x_1, \\
\tau_1^\star x_2&=-x_1,& \tau_2^\star x_2&=x_4-x_2 x_8,& \tau_3^\star x_2&=x_2, \\
\tau_1^\star x_3&=-x_2 x_3+x_7,& \tau_2^\star x_3&=x_1,& \tau_3^\star x_3&=x_5-x_3 x_{11}, \\
\tau_1^\star x_4&=-x_1 x_4+x_8,& \tau_2^\star x_4&=x_2,& \tau_3^\star x_4&=x_6-x_4 x_{12}, \\
\tau_1^\star x_5&=-x_2 x_5+x_9,& \tau_2^\star x_5&=x_5,& \tau_3^\star x_5&=x_3, \\
\tau_1^\star x_6&=-x_1 x_6+x_{10},& \tau_2^\star x_6&=x_6,& \tau_3^\star x_6&=x_4, \\
\tau_1^\star x_7&=x_3,& \tau_2^\star x_7&=-x_8,& \tau_3^\star x_7&=x_9-x_7 x_{11}, \\
\tau_1^\star x_8&=x_4,& \tau_2^\star x_8&=-x_7,& \tau_3^\star x_8&=x_{10}-x_8 x_{12}, \\
\tau_1^\star x_9&=x_5,& \tau_2^\star x_9&=-x_8 x_9+x_{11},& \tau_3^\star x_9&=x_7, \\
\end{align*}
\begin{align*}
\tau_1^\star x_{10}&=x_6,& \tau_2^\star x_{10}&=-x_7 x_{10}+x_{12},& \tau_3^\star x_{10}&=x_8, \\
\tau_1^\star x_{11}&=x_{11},& \tau_2^\star x_{11}&=x_9,& \tau_3^\star x_{11}&=-x_{12}, \\
\tau_1^\star x_{12}&=x_{12},& \tau_2^\star x_{12}&=x_{10},& \tau_3^\star x_{12}&=-x_{11}. \\
\end{align*}

\begin{thm}
\label{thm B1}

The space $V$ of all quadratic Poisson structures on $\C^{12}$, which have $H_1,H_2,H_3$ as Casimir elements,
is a 3-dimensional vector space consisting  of log-canonical structures. 
For suitable coordinates $\bs b=(b_1,b_2,b_3)$ on $V$, the Poisson structures have the form:
\begin{align*}
\{x_1,x_2\}&=0, &
\{x_3,x_4\}&=0, &
\{x_5,x_6\}&=0, &
\{x_7,x_8\}&=0, &
\{x_9,x_{10}\}&=0, &
\{x_{11},x_{12}\}&=0, &
\end{align*}
\begin{align*}
\{x_1,x_{11}\}&=b_2x_1x_{11}, &
\{x_1,x_{12}\}&=-b_2x_1x_{12}, &
\\
\{x_2,x_{11}\}&=-b_2x_2x_{11}, &
\{x_2,x_{12}\}&=b_2x_2x_{12}, &
\\
\{x_3,x_{9}\}&=-(b_1-b_2+b_3)x_3x_{9}, &
\{x_3,x_{10}\}&=(b_1-b_2+b_3)x_3x_{10}, &
\\
\{x_4,x_{9}\}&=(b_1-b_2+b_3)x_4x_{9}, &
\{x_4,x_{10}\}&=-(b_1-b_2+b_3)x_4x_{10}, &
\\
\{x_5,x_{7}\}&=(b_1-b_3) x_5 x_7, &
\{x_5,x_{8}\}&=-(b_1-b_3) x_5 x_8, &
\\
\{x_6,x_{7}\}&=-(b_1-b_3) x_6 x_7, &
\{x_6,x_{8}\}&=(b_1-b_3) x_6 x_8, &
\\
\{x_1,x_{3}\}&=-b_3 x_1 x_3, &
\{x_1,x_{4}\}&=b_3 x_1 x_4, &
\\
\{x_2,x_{3}\}&=b_3 x_2 x_3, &
\{x_2,x_{4}\}&=-b_3 x_2 x_4, &
\\
\{x_1,x_{5}\}&=-(b_3-b_2) x_1 x_5, &
\{x_1,x_{5}\}&=(b_3-b_2) x_1 x_6, &
\\
\{x_2,x_{5}\}&=(b_3-b_2) x_2 x_5, &
\{x_2,x_{6}\}&=-(b_3-b_2) x_2 x_6, &
\\
\{x_1,x_{7}\}&=-b_3 x_1 x_7, &
\{x_1,x_{8}\}&=b_3 x_1 x_8, &
\\
\{x_2,x_{7}\}&=b_3 x_2 x_7, &
\{x_2,x_{8}\}&=-b_3 x_2 x_8, &
%\end{align*}
\\
%\begin{align*}
\{x_1,x_{9}\}&=-(b_3-b_2) x_1 x_9, &
\{x_1,x_{10}\}&=(b_3-b_2) x_1 x_{10}, &
\\
\{x_2,x_{9}\}&=(b_3-b_2) x_2 x_9, &
\{x_2,x_{10}\}&=-(b_3-b_2) x_2 x_{10}, &
\\
\{x_3,x_{5}\}&=-(b_1-b_2) x_3 x_5, &
\{x_3,x_{6}\}&=(b_1-b_2) x_3 x_6, &
\\
\{x_4,x_{5}\}&=(b_1-b_2) x_4 x_5, &
\{x_4,x_{6}\}&=-(b_1-b_2) x_4 x_6, &
\\
\{x_3,x_{7}\}&=-b_3 x_3 x_7, &
\{x_3,x_{8}\}&=b_3 x_3 x_8, &
\\
\{x_4,x_{7}\}&=b_3 x_4 x_7, &
\{x_4,x_{8}\}&=-b_3 x_4 x_8, &
\\
\{x_3,x_{11}\}&=-(b_1-b_2) x_3 x_{11}, &
\{x_3,x_{12}\}&=(b_1-b_2) x_3 x_{12}, &
\\
\{x_4,x_{11}\}&=(b_1-b_2) x_4 x_{11}, &
\{x_4,x_{12}\}&=-(b_1-b_2) x_4 x_{12}, &
\\
\{x_5,x_{7}\}&=(b_1-b_3) x_5 x_{7}, &
\{x_5,x_{8}\}&=-(b_1-b_3) x_5 x_{8}, &
\\
\{x_6,x_{7}\}&=-(b_1-b_3) x_6 x_{7}, &
\{x_6,x_{8}\}&=(b_1-b_3) x_6 x_{8}, &
\\
\{x_5,x_{9}\}&=-(b_3-b_2) x_5 x_{9}, &
\{x_5,x_{10}\}&=(b_3-b_2) x_5 x_{10}, &
\\
\{x_6,x_{9}\}&=(b_3-b_2) x_6 x_{9}, &
\{x_6,x_{10}\}&=-(b_3-b_2) x_6 x_{10}, &
\\
\{x_5,x_{11}\}&=-(b_1-b_2)x_5 x_{11}, &
\{x_5,x_{12}\}&=(b_1-b_2)x_5 x_{12}, &
\\
\{x_6,x_{11}\}&=(b_1-b_2)x_6 x_{11}, &
\{x_6,x_{12}\}&=-(b_1-b_2)x_6 x_{12}, &
\end{align*}
%\\
\begin{align*}
\{x_7,x_{9}\}&=-b_1 x_7 x_{9}, &
\{x_7,x_{10}\}&=b_1 x_7 x_{10}, &
\\
\{x_8,x_{9}\}&=b_1 x_8 x_{9}, &
\{x_8,x_{10}\}&=-b_1 x_8 x_{10}, &
\\
\{x_7,x_{11}\}&=-b_1 x_7 x_{11}, &
\{x_7,x_{12}\}&=b_1 x_7 x_{12}, &
\\
\{x_8,x_{11}\}&=b_1 x_8 x_{11}, &
\{x_8,x_{12}\}&=-b_1 x_8 x_{12}, &
\\
\{x_9,x_{11}\}&=-b_1 x_9 x_{11}, &
\{x_9,x_{12}\}&=b_1 x_9 x_{12}, &
\\
\{x_{10},x_{11}\}&=b_1 x_{10} x_{11}, &
\{x_{10},x_{12}\}&=-b_1 x_{10} x_{12}. &
\end{align*}
Each of these Poisson structures is $\nu$-invariant.
If $(b_1,b_2,b_3)\ne (0,0,0)$, then the Poisson structure is of rank 2.

\end{thm}

The Poisson structure with parameters $\bs b$ is denoted by $\{\,,\,\}_{\bs b}$.

\proof
A computer assistant calculation shows that the only requirements on a quadratic bracket $\{\,,\,\}$ to be skew-symmetric and have $H_1,H_2,H_3$ as Casimir elements  uniquely determines the Poisson structures above. 
 \endproof

Another computer assistant calculation shows that if 
    a polynomial Poisson structure $\{\,,\,\}$ on $\C^{12}$ has  $H_1, H_2, H_3$ as Casimir elements,
   then its Taylor expansion  at the origin, has to start with at least  quadratic terms.

\vsk.2>

\subsection{Braid group $\mc B_4$ action}

Given a Poisson bracket $\{\,,\,\}$ on $\C^{12}$ define the  Poisson bracket $\{\,,\,\}^{\tau_i}$ by
\[
\{f,g\}^{\tau_i}:=\tau_i^\star\{f\circ \tau_i^{-1}, g\circ \tau_i^{-1}\},\quad i=1,2,3.
\]
These formulas define a braid group $\mc B_4$ action on the space of Poisson structures on $\C^{12}$.

\begin{thm}
\label{thm b act}
The three-parameter family of Poisson structures $\{\,,\,\}_{\bs b}$ is invariant with 
respect to the braid group $\mc B_4$-action on the space of all Poisson structures. The induced braid group $\mc B_4$ action
$\rho$  on the space of parameters
$V$ is a vector representation defined by the formulas,
\bea
\tau_1\colon &&(b_1,b_2,b_3)\mapsto (b_1-b_2,\quad -b_2,\quad -b_3),
\\
\tau_2\colon &&(b_1,b_2,b_3)\mapsto (-b_1,\quad -b_1+b_2-b_3,\quad -b_3),
\\
\tau_3\colon &&(b_1,b_2,b_3)\mapsto (-b_1,\quad -b_2,\quad -b_2+b_3).
\eea
The representation $\rho$ factors through a representation of the symmetric group
$\frak S_4$, $\rho(\tau_i^2)=\id$, $i=1,2,3$.
The representation $\rho: \frak S_4 \to \on{GL}(V)$ is irreducible and is isomorphic to
the standard three-dimensional representation tensored with the $\on{sgn}$ representation.
\qed
\end{thm}

By Theorem \ref{thm b act},
there is no a $\mc B_4$-invariant or $\mc B_4$-anti-invariant quadratic Poisson structure on $\C^{12}$, which has
$H_1,H_2,H_3$ as Casimir elements.

\begin{rem}
A computer assistant calculation shows that if  a log-canonical Poisson structure $\{\,,\,\}$ on $\C^{12}$ remains
to be log-canonical
 after the action on it  by any element of the
 braid group $\mc B_4$, then $\{\,,\,\}$ is one of the Poisson structures $\{\,,\,\}_{\bs b}$
in Theorem \ref{thm B1}.

\end{rem}

\subsection{Coefficients of $\{\,,\,\}_{\bs b}$ and elements of weight lattice} 
Consider $\C^4$ with  standard Euclidean quadratic form $(\,,\,)$. Denote 
$(1,-1,0,0)$, $(0,1,-1,0)$, $(0,0,1,-1) \in \C^4$ by $v_1,v_2,v_3$.
We identify the space of parameters $V$ with the subspace
$\{\bs t\in \C^4\ |\ \sum_{i=1}^4t_i=0\}$, by sending a point of $V$ with coordinates $(b_1,b_2,b_3)$
to the point $b_1v_1+b_2v_2+b_3v_3$. The vectors $v_1,v_2,v_3$ generate the {\it root lattice} in $V$.

\vsk.2>
For $i=1,2,3$, the linear map $\rho(\tau_i):V\to V$ permutes the $i$-th and $i+1$-st coordinates of vectors of $V$ and multiplies 
the vectors by $-1$.

\vsk.2>
The {\it weight lattice} in $V$ is the lattice of the elements $t=(t_1,t_2,t_3,t_4)\in \C^4$ such that
$\sum_{i=1}^4t_i=0$ and $(t, v_i)\in \Z$, $i=1,2,3$.
The weight lattice has a basis
$w_1=(3,-1,-1,-1)/4$, $w_2=(2,2,-2,-2)/4$, $w_3=(1,1,1,-3)/4$
with the property $(w_i,v_j) = \delta_{ij}$  for all $i,j$.

There are exactly 8 vectors of the weight lattice of square length 12/16,
\bean
\label{l12}
\pm w_1, \quad
\pm w_1\mp w_2, 
\quad
\pm w_2\mp w_3, \quad \pm w_3,
\eean
and there are exactly 6 vectors of the weight lattice of square length 1,
\bean
\label{l16}
\pm w_2,\quad, \pm w_1\mp w_2\pm w_3,\quad \pm w_1\mp w_3.
\eean
All other vectors of the root lattice are longer. These two groups of vectors form
two $\frak S_4$-orbits.

\vsk.2>
The scalar products of these  $14$ vectors with the vector $b_1v_1+b_2v_2+b_3v_3$ give us
the linear functions in $b_1,b_2,b_3$,
\bea
&&
\pm b_1, \quad
\pm b_1\mp b_2, 
\quad
\pm b_2\mp b_3, \quad \pm b_3,
\quad\pm b_2,\quad, \pm b_1\mp b_2\pm b_3,\quad \pm b_1\mp b_3.
\eea
These are exactly the linear functions appearing as coefficients of the Poisson structure
$\{\,,\,\}_{\bs b}$ of Theorem \ref{thm B1}.

\subsection{Casimir subalgebra} 
Denote by $\mc C$ the subalgebra of $\C[\bs x]$ generated by the following $20$ monomials:
\begin{align*}
m_{1}&=x_{1} x_{2}, & 
m_{2}&=x_{3} x_{4}, &
m_{3}&=x_{5} x_{6}, &
m_{4}&=x_{7} x_{8}, \\
m_{5}&=x_{9} x_{10}, &
m_{6}&=x_{11} x_{12}, &
m_{7}&=x_{2} x_{3} x_{8}, &
m_{8}&= x_{2} x_{5} x_{10}, \\
m_{9}&= x_{4} x_{5} x_{12}, &
m_{10}&= x_{8} x_{9} x_{12}, &
m_{11}&=x_{1} x_{4} x_{7}, &
m_{12}&=x_{1} x_{6} x_{9}, \\
m_{13}&= x_{3} x_{6} x_{11}, &
m_{14}&=x_{7} x_{10} x_{11} , &
m_{15}&= x_{2} x_{5} x_{8}x_{12}, &
m_{16}&=x_{2} x_{3}x_{10} x_{11}, \\
m_{17}&=x_{3} x_{6} x_{8} x_{9}, &
m_{18}&= x_{4} x_{5} x_{7}x_{10}, &
m_{19}&=x_{1} x_{4} x_{9}x_{12} , &
m_{20}&=x_{1}  x_{6} x_{7}x_{11}.&
\end{align*}
It is easy to see that the polynomials $H_1,H_2,H_3$ are elements of the subalgebra $\mc C$.

\begin{thm}
For every $\bs b$ each element of the subalgebra $\mc C$ is a Casimir element of the Poisson structure
$\{\,,\,\}_{\bs b}$.
The subalgebra $\mc C$ is $\nu$-invariant and the braid group $\mc B_4$ action invariant.

\end{thm}

\begin{proof} 
The theorem is proved by direct verification.
For example,  easy calculations lead to formulas like
\bea
\tau_1^\star m_2
&=&
-m_{11}+m_{1} m_{2}+m_{4}-m_{7},
\\
\tau_3^\star m_{18}&=&
m_{17}-m_{10} m_2-m_{13} m_4+m_2 m_4 m_6 ,
\eea
which prove the braid group invariance of  $\mc C$.
\end{proof}

\subsection{Symplectic leaves}
Since $\{\,,\,\}_{\bs b}$ is of rank 2, the symplectic leaves of  $\{\,,\,\}_{\bs b}$ are two-dimensional.
In logarithmic coordinates $\log x_i$, $i=1,\dots,12$, they  are two-dimensional
affine subspaces.
More precisely, we have the following statement.

\begin{thm}
Given $(b_1,b_2,b_3)\ne (0,0,0)$, then  the function 
\bea
C_{\bs b}(\bs x) = (b_1-b_2)\log x_1+(b_2-b_3)\log x_3+b_3\log x_5
\eea
is a Casimir element of $\{\,,\,\}_{\bs b}$; the $\C$-span of  $C_{\bs b}$ 
and the functions $\log m_i$, $i=1,\dots, 20$, 
is $10$-dimensional, while the  $\C$-span of  the functions $\log m_i$, $i=1,\dots, 20$, 
is $9$-dimensional.
\qed

\end{thm}

Hence the symplectic leaves of $\{\,,\,\}_{\bs b}$ are the surfaces, on which
the functions of this $10$-dimensional $\C$-span are constant. In particular,
the leaves do depend on $\bs b$.

We may also conclude that  $x_1^{b_1-b_2}x_3^{b_2-b_3}x_5^{b_3}$ is
a  Casimir element of $\{\,,\,\}_{\bs b}$
functionally independent of the Casimir elements $m_i$, $i=1,\dots, 20$.

\end{document}

%% file: introdis.pdf_tex
%% Creator: Inkscape inkscape 0.91, www.inkscape.org
%% Accompanies image file 'introdis.pdf' (pdf, eps, ps)
\begingroup%
  \makeatletter%
  \providecommand\color[2][]{%
    \errmessage{(Inkscape) Color is used for the text in Inkscape, but the package 'color.sty' is not loaded}%
    \renewcommand\color[2][]{}%
  }%
  \providecommand\transparent[1]{%
    \errmessage{(Inkscape) Transparency is used (non-zero) for the text in Inkscape, but the package 'transparent.sty' is not loaded}%
    \renewcommand\transparent[1]{}%
  }%
  \providecommand\rotatebox[2]{#2}%
  \ifx\svgwidth\undefined%
    \setlength{\unitlength}{507.09784219bp}%
    \ifx\svgscale\undefined%
      \relax%
    \else%
      \setlength{\unitlength}{\unitlength * \real{\svgscale}}%
    \fi%
  \else%
    \setlength{\unitlength}{\svgwidth}%
  \fi%
  \global\let\svgwidth\undefined%
  \global\let\svgscale\undefined%
  \makeatother%
  \begin{picture}(1,0.21774261)%
    \put(0,0){\includegraphics[width=\unitlength,page=1]{introdis.pdf}}%
    \put(0.16853396,0.05860668){\color[rgb]{0,0,0}\makebox(0,0)[lb]{\smash{\tiny$(3,15,6)$}}}%
    \put(0.22375011,0.14312123){\color[rgb]{0,0,0}\makebox(0,0)[lb]{\smash{\tiny$(15,87,6)$}}}%
    \put(0.00558994,0.14447343){\color[rgb]{0,0,0}\makebox(0,0)[lb]{\smash{\tiny$(3,39,15)$}}}%
    \put(0.89761269,0.05522608){\color[rgb]{0,0,0}\makebox(0,0)[lb]{\smash{\tiny$(1,3,2)$}}}%
    \put(0.94967367,0.14131822){\color[rgb]{0,0,0}\makebox(0,0)[lb]{\smash{\tiny$(3,5,2)$}}}%
    \put(0.75765668,0.14424805){\color[rgb]{0,0,0}\makebox(0,0)[lb]{\smash{\tiny$(1,4,3)$}}}%
    \put(0.52710098,0.04846494){\color[rgb]{0,0,0}\makebox(0,0)[lb]{\smash{\tiny$\binom{1}{2},\binom{3}{4},\binom{2}{2}$}}}%
    \put(0.58479619,0.14875549){\color[rgb]{0,0,0}\makebox(0,0)[lb]{\smash{\tiny$\binom{3}{4},\binom{5}{8},\binom{2}{4}$}}}%
    \put(0.3578465,0.14672716){\color[rgb]{0,0,0}\makebox(0,0)[lb]{\smash{\tiny$\binom{1}{1},\binom{4}{5},\binom{3}{4}$}}}%
  \end{picture}%
\endgroup%

%% file: Figures1_2.pdf_tex
%% Creator: Inkscape inkscape 0.91, www.inkscape.org
%% Accompanies image file 'Figures1&2.pdf' (pdf, eps, ps)
\begingroup%
  \makeatletter%
  \providecommand\color[2][]{%
    \errmessage{(Inkscape) Color is used for the text in Inkscape, but the package 'color.sty' is not loaded}%
    \renewcommand\color[2][]{}%
  }%
  \providecommand\transparent[1]{%
    \errmessage{(Inkscape) Transparency is used (non-zero) for the text in Inkscape, but the package 'transparent.sty' is not loaded}%
    \renewcommand\transparent[1]{}%
  }%
  \providecommand\rotatebox[2]{#2}%
  \ifx\svgwidth\undefined%
    \setlength{\unitlength}{460.18595298bp}%
    \ifx\svgscale\undefined%
      \relax%
    \else%
      \setlength{\unitlength}{\unitlength * \real{\svgscale}}%
    \fi%
  \else%
    \setlength{\unitlength}{\svgwidth}%
  \fi%
  \global\let\svgwidth\undefined%
  \global\let\svgscale\undefined%
  \makeatother%
  \begin{picture}(1,0.27467434)%
    \put(0,0){\includegraphics[width=\unitlength,page=1]{Figures1_2.pdf}}%
    \put(0.58415672,0.05794){\color[rgb]{0,0,0}\makebox(0,0)[lb]{\smash{$D_1$}}}%
    \put(0.84554175,0.0616652){\color[rgb]{0,0,0}\makebox(0,0)[lb]{\smash{$D_3$}}}%
    \put(0.71329705,0.04179742){\color[rgb]{0,0,0}\makebox(0,0)[lb]{\smash{$l_1$}}}%
    \put(0.68287456,0.1156806){\color[rgb]{0,0,0}\makebox(0,0)[lb]{\smash{$l_2$}}}%
    \put(0.61520003,0.15479522){\color[rgb]{0,0,0}\makebox(0,0)[lb]{\smash{$l_3$}}}%
    \put(0.54752554,0.1516909){\color[rgb]{0,0,0}\makebox(0,0)[lb]{\smash{$l_4$}}}%
    \put(0.75799954,0.04303917){\color[rgb]{0,0,0}\makebox(0,0)[lb]{\smash{$r_1$}}}%
    \put(0.78035071,0.11257628){\color[rgb]{0,0,0}\makebox(0,0)[lb]{\smash{$r_2$}}}%
    \put(0.85237131,0.15976214){\color[rgb]{0,0,0}\makebox(0,0)[lb]{\smash{$r_3$}}}%
    \put(0.93432581,0.16286654){\color[rgb]{0,0,0}\makebox(0,0)[lb]{\smash{$r_4$}}}%
    \put(0,0){\includegraphics[width=\unitlength,page=2]{Figures1_2.pdf}}%
    \put(0.06744159,0.07666874){\color[rgb]{0,0,0}\makebox(0,0)[lb]{\smash{$D_1$}}}%
    \put(0.27605293,0.07666874){\color[rgb]{0,0,0}\makebox(0,0)[lb]{\smash{$D_3$}}}%
    \put(0.17174725,0.1809744){\color[rgb]{0,0,0}\makebox(0,0)[lb]{\smash{$D_2$}}}%
    \put(0,0){\includegraphics[width=\unitlength,page=3]{Figures1_2.pdf}}%
    \put(0.19681439,0.09155511){\color[rgb]{0,0,0}\makebox(0,0)[lb]{\smash{$v_1$}}}%
  \end{picture}%
\endgroup%

%% file: Figures3_4.pdf_tex
%% Creator: Inkscape inkscape 0.91, www.inkscape.org
%% Accompanies image file 'Figures3&4.pdf' (pdf, eps, ps)
\begingroup%
  \makeatletter%
  \providecommand\color[2][]{%
    \errmessage{(Inkscape) Color is used for the text in Inkscape, but the package 'color.sty' is not loaded}%
    \renewcommand\color[2][]{}%
  }%
  \providecommand\transparent[1]{%
    \errmessage{(Inkscape) Transparency is used (non-zero) for the text in Inkscape, but the package 'transparent.sty' is not loaded}%
    \renewcommand\transparent[1]{}%
  }%
  \providecommand\rotatebox[2]{#2}%
  \ifx\svgwidth\undefined%
    \setlength{\unitlength}{361.56728635bp}%
    \ifx\svgscale\undefined%
      \relax%
    \else%
      \setlength{\unitlength}{\unitlength * \real{\svgscale}}%
    \fi%
  \else%
    \setlength{\unitlength}{\svgwidth}%
  \fi%
  \global\let\svgwidth\undefined%
  \global\let\svgscale\undefined%
  \makeatother%
  \begin{picture}(1,0.34803022)%
    \put(0,0){\includegraphics[width=\unitlength,page=1]{Figures3_4.pdf}}%
    \put(0.11715988,0.11848722){\color[rgb]{0,0,0}\makebox(0,0)[lb]{\smash{$r_1$}}}%
    \put(0.13098857,0.18763064){\color[rgb]{0,0,0}\makebox(0,0)[lb]{\smash{$r_2$}}}%
    \put(0.19025434,0.23101321){\color[rgb]{0,0,0}\makebox(0,0)[lb]{\smash{$r_3$}}}%
    \put(0.15256135,0.09454385){\color[rgb]{0,0,0}\makebox(0,0)[lb]{\smash{$l_1$}}}%
    \put(0.22257389,0.12875998){\color[rgb]{0,0,0}\makebox(0,0)[lb]{\smash{$l_2$}}}%
    \put(0.25544666,0.18407472){\color[rgb]{0,0,0}\makebox(0,0)[lb]{\smash{$l_3$}}}%
    \put(0,0){\includegraphics[width=\unitlength,page=2]{Figures3_4.pdf}}%
    \put(0.70550654,0.1176){\color[rgb]{0,0,0}\makebox(0,0)[lb]{\smash{$r_1$}}}%
    \put(0.7237604,0.18674343){\color[rgb]{0,0,0}\makebox(0,0)[lb]{\smash{$r_2$}}}%
    \put(0.78302618,0.22570082){\color[rgb]{0,0,0}\makebox(0,0)[lb]{\smash{$r_3$}}}%
    \put(0.7497584,0.09808177){\color[rgb]{0,0,0}\makebox(0,0)[lb]{\smash{$l_1$}}}%
    \put(0.81092062,0.13672309){\color[rgb]{0,0,0}\makebox(0,0)[lb]{\smash{$l_2$}}}%
    \put(0.84379332,0.18318747){\color[rgb]{0,0,0}\makebox(0,0)[lb]{\smash{$l_3$}}}%
    \put(0,0){\includegraphics[width=\unitlength,page=3]{Figures3_4.pdf}}%
    \put(0.67327443,0.02041008){\color[rgb]{0,0,0}\makebox(0,0)[lb]{\smash{$r_1$}}}%
    \put(0.68746839,0.06042203){\color[rgb]{0,0,0}\makebox(0,0)[lb]{\smash{$r_2$}}}%
    \put(0.7481206,0.06571764){\color[rgb]{0,0,0}\makebox(0,0)[lb]{\smash{$r_3$}}}%
    \put(0.80213819,0.05431777){\color[rgb]{0,0,0}\makebox(0,0)[lb]{\smash{$r_4$}}}%
    \put(0.60325022,0.09709262){\color[rgb]{0,0,0}\makebox(0,0)[lb]{\smash{$r_1$}}}%
    \put(0.63438199,0.2302692){\color[rgb]{0,0,0}\makebox(0,0)[lb]{\smash{$r_1$}}}%
    \put(0.74631047,0.31917722){\color[rgb]{0,0,0}\makebox(0,0)[lb]{\smash{$r_1$}}}%
    \put(0.87470799,0.207475){\color[rgb]{0,0,0}\makebox(0,0)[lb]{\smash{$r_1$}}}%
    \put(0.83827897,0.11253578){\color[rgb]{0,0,0}\makebox(0,0)[lb]{\smash{$r_1$}}}%
    \put(0.66182826,0.07835864){\color[rgb]{0,0,0}\makebox(0,0)[lb]{\smash{$l_1$}}}%
    \put(0.70072311,0.20881153){\color[rgb]{0,0,0}\makebox(0,0)[lb]{\smash{$l_1$}}}%
    \put(0.77764798,0.26084107){\color[rgb]{0,0,0}\makebox(0,0)[lb]{\smash{$l_1$}}}%
    \put(0.90101382,0.16802463){\color[rgb]{0,0,0}\makebox(0,0)[lb]{\smash{$l_1$}}}%
    \put(0.71275645,0.27965933){\color[rgb]{0,0,0}\makebox(0,0)[lb]{\smash{$l_2$}}}%
    \put(0.67696395,0.12366052){\color[rgb]{0,0,0}\makebox(0,0)[lb]{\smash{$l_2$}}}%
    \put(0.58276042,0.06946328){\color[rgb]{0,0,0}\makebox(0,0)[lb]{\smash{$l_2$}}}%
    \put(0.85582427,0.05835543){\color[rgb]{0,0,0}\makebox(0,0)[lb]{\smash{$l_1$}}}%
    \put(0.6271245,0.19339078){\color[rgb]{0,0,0}\makebox(0,0)[lb]{\smash{$l_3$}}}%
    \put(0.63456532,0.02074141){\color[rgb]{0,0,0}\makebox(0,0)[lb]{\smash{$l_1$}}}%
    \put(0.91641296,0.13925138){\color[rgb]{0,0,0}\makebox(0,0)[lb]{\smash{$r_2$}}}%
  \end{picture}%
\endgroup%

%% file: Figures5_6.pdf_tex
%% Creator: Inkscape inkscape 0.91, www.inkscape.org
%% Accompanies image file 'Figures5&6.pdf' (pdf, eps, ps)
\begingroup%
  \makeatletter%
  \providecommand\color[2][]{%
    \errmessage{(Inkscape) Color is used for the text in Inkscape, but the package 'color.sty' is not loaded}%
    \renewcommand\color[2][]{}%
  }%
  \providecommand\transparent[1]{%
    \errmessage{(Inkscape) Transparency is used (non-zero) for the text in Inkscape, but the package 'transparent.sty' is not loaded}%
    \renewcommand\transparent[1]{}%
  }%
  \providecommand\rotatebox[2]{#2}%
  \ifx\svgwidth\undefined%
    \setlength{\unitlength}{299.58522437bp}%
    \ifx\svgscale\undefined%
      \relax%
    \else%
      \setlength{\unitlength}{\unitlength * \real{\svgscale}}%
    \fi%
  \else%
    \setlength{\unitlength}{\svgwidth}%
  \fi%
  \global\let\svgwidth\undefined%
  \global\let\svgscale\undefined%
  \makeatother%
  \begin{picture}(1,0.29221297)%
    \put(0,0){\includegraphics[width=\unitlength,page=1]{Figures5_6.pdf}}%
    \put(0.17829849,0.06082703){\color[rgb]{0,0,0}\makebox(0,0)[lb]{\smash{$v_1$}}}%
    \put(0.22948037,0.19047485){\color[rgb]{0,0,0}\makebox(0,0)[lb]{\smash{$v_2$}}}%
    \put(0,0){\includegraphics[width=\unitlength,page=2]{Figures5_6.pdf}}%
    \put(0.81872381,0.15844497){\color[rgb]{0,0,0}\makebox(0,0)[lb]{\smash{$v$}}}%
    \put(0.79607327,0.25124825){\color[rgb]{0,0,0}\makebox(0,0)[lb]{\smash{$C_2$}}}%
    \put(0.67590713,0.10437852){\color[rgb]{0,0,0}\makebox(0,0)[lb]{\smash{$C_1$}}}%
    \put(0.88953582,0.10437852){\color[rgb]{0,0,0}\makebox(0,0)[lb]{\smash{$C_3$}}}%
    \put(0.82277686,0.07767493){\color[rgb]{0,0,0}\makebox(0,0)[lb]{\smash{$r_b$}}}%
    \put(0,0){\includegraphics[width=\unitlength,page=3]{Figures5_6.pdf}}%
    \put(0.76250301,0.07691192){\color[rgb]{0,0,0}\makebox(0,0)[lb]{\smash{$l_a$}}}%
  \end{picture}%
\endgroup%

%% file: Figure8.pdf_tex
%% Creator: Inkscape inkscape 0.91, www.inkscape.org
%% Accompanies image file 'Figure8.pdf' (pdf, eps, ps)
\begingroup%
  \makeatletter%
  \providecommand\color[2][]{%
    \errmessage{(Inkscape) Color is used for the text in Inkscape, but the package 'color.sty' is not loaded}%
    \renewcommand\color[2][]{}%
  }%
  \providecommand\transparent[1]{%
    \errmessage{(Inkscape) Transparency is used (non-zero) for the text in Inkscape, but the package 'transparent.sty' is not loaded}%
    \renewcommand\transparent[1]{}%
  }%
  \providecommand\rotatebox[2]{#2}%
  \ifx\svgwidth\undefined%
    \setlength{\unitlength}{213.28392473bp}%
    \ifx\svgscale\undefined%
      \relax%
    \else%
      \setlength{\unitlength}{\unitlength * \real{\svgscale}}%
    \fi%
  \else%
    \setlength{\unitlength}{\svgwidth}%
  \fi%
  \global\let\svgwidth\undefined%
  \global\let\svgscale\undefined%
  \makeatother%
  \begin{picture}(1,0.58303827)%
    \put(0,0){\includegraphics[width=\unitlength,page=1]{Figure8.pdf}}%
    \put(0.24825537,0.16520066){\color[rgb]{0,0,0}\makebox(0,0)[lb]{\smash{\small$A_1(\bm s)$}}}%
    \put(0.64477571,0.1585028){\color[rgb]{0,0,0}\makebox(0,0)[lb]{\smash{\small$A_3(\bm s)$}}}%
    \put(0.01329048,0.32461258){\color[rgb]{0,0,0}\makebox(0,0)[lb]{\smash{\small$A_6(\bm s)$}}}%
    \put(0.20217336,0.38221522){\color[rgb]{0,0,0}\makebox(0,0)[lb]{\smash{\small$A_4(\bm s)$}}}%
    \put(0.43687042,0.31389581){\color[rgb]{0,0,0}\makebox(0,0)[lb]{\smash{\small$A_2(\bm s)$}}}%
    \put(0.68630321,0.38355486){\color[rgb]{0,0,0}\makebox(0,0)[lb]{\smash{\small$A_5(\bm s)$}}}%
    \put(0.85911104,0.32193341){\color[rgb]{0,0,0}\makebox(0,0)[lb]{\smash{\small$A_7(\bm s)$}}}%
    \put(0.34175901,0.52876725){\color[rgb]{0,0,0}\makebox(0,0)[lb]{\smash{\small$A_8(\bm s)$}}}%
    \put(0.5614527,0.53412553){\color[rgb]{0,0,0}\makebox(0,0)[lb]{\smash{\small$A_9(\bm s)$}}}%
  \end{picture}%
\endgroup%

%% file: Trees.pdf_tex
%% Creator: Inkscape inkscape 0.91, www.inkscape.org
%% Accompanies image file 'Trees.pdf' (pdf, eps, ps)
\begingroup%
  \makeatletter%
  \providecommand\color[2][]{%
    \errmessage{(Inkscape) Color is used for the text in Inkscape, but the package 'color.sty' is not loaded}%
    \renewcommand\color[2][]{}%
  }%
  \providecommand\transparent[1]{%
    \errmessage{(Inkscape) Transparency is used (non-zero) for the text in Inkscape, but the package 'transparent.sty' is not loaded}%
    \renewcommand\transparent[1]{}%
  }%
  \providecommand\rotatebox[2]{#2}%
  \ifx\svgwidth\undefined%
    \setlength{\unitlength}{473.50682481bp}%
    \ifx\svgscale\undefined%
      \relax%
    \else%
      \setlength{\unitlength}{\unitlength * \real{\svgscale}}%
    \fi%
  \else%
    \setlength{\unitlength}{\svgwidth}%
  \fi%
  \global\let\svgwidth\undefined%
  \global\let\svgscale\undefined%
  \makeatother%
  \begin{picture}(1,0.27527632)%
    \put(0,0){\includegraphics[width=\unitlength,page=1]{Trees.pdf}}%
    \put(0.13073664,0.03932693){\color[rgb]{0,0,0}\makebox(0,0)[lb]{\smash{\small$ \binom{1}{2}$}}}%
    \put(0.24848119,0.0388951){\color[rgb]{0,0,0}\makebox(0,0)[lb]{\smash{\small$\binom{2}{2}$}}}%
    \put(0.19807556,0.10389077){\color[rgb]{0,0,0}\makebox(0,0)[lb]{\smash{\small$\binom{3}{4}$}}}%
    \put(0.12240055,0.13150079){\color[rgb]{0,0,0}\makebox(0,0)[lb]{\smash{\small$\binom{4}{5}$}}}%
    \put(0.05871575,0.10545105){\color[rgb]{0,0,0}\makebox(0,0)[lb]{\smash{\small$\binom{5}{7}$}}}%
    \put(0.1531654,0.19013418){\color[rgb]{0,0,0}\makebox(0,0)[lb]{\smash{\small$\binom{7}{10}$}}}%
    \put(0.27462714,0.11601941){\color[rgb]{0,0,0}\makebox(0,0)[lb]{\smash{\small$\binom{5}{8}$}}}%
    \put(0.34381053,0.10303743){\color[rgb]{0,0,0}\makebox(0,0)[lb]{\smash{\small$\binom{7}{10}$}}}%
    \put(0.23262925,0.19440086){\color[rgb]{0,0,0}\makebox(0,0)[lb]{\smash{\small$\binom{8}{13}$}}}%
    \put(0.1772757,0.24326156){\color[rgb]{0,0,0}\makebox(0,0)[lb]{\smash{\tiny$\binom{10}{14}$}}}%
    \put(0.12974609,0.25324432){\color[rgb]{0,0,0}\makebox(0,0)[lb]{\smash{\tiny$\binom{11}{17}$}}}%
    \put(0.0287693,0.05272297){\color[rgb]{0,0,0}\makebox(0,0)[lb]{\smash{\small$\binom{6}{8}$}}}%
    \put(0.04987974,0.18037622){\color[rgb]{0,0,0}\makebox(0,0)[lb]{\smash{\small$\binom{9}{14}$}}}%
    \put(0.00598649,0.15662001){\color[rgb]{0,0,0}\makebox(0,0)[lb]{\smash{\tiny$\binom{14}{22}$}}}%
    \put(0.07189499,0.2457358){\color[rgb]{0,0,0}\makebox(0,0)[lb]{\smash{\tiny$\binom{13}{19}$}}}%
    \put(0.22679917,0.25008826){\color[rgb]{0,0,0}\makebox(0,0)[lb]{\smash{\tiny$\binom{11}{17}$}}}%
    \put(0.27342401,0.24228817){\color[rgb]{0,0,0}\makebox(0,0)[lb]{\smash{\tiny$\binom{13}{20}$}}}%
    \put(0.31356091,0.24155489){\color[rgb]{0,0,0}\makebox(0,0)[lb]{\smash{\tiny$\binom{17}{25}$}}}%
    \put(0.3867249,0.18847644){\color[rgb]{0,0,0}\makebox(0,0)[lb]{\smash{\tiny$\binom{19}{28}$}}}%
    \put(0.33136092,0.17064212){\color[rgb]{0,0,0}\makebox(0,0)[lb]{\smash{\small$\binom{12}{17}$}}}%
    \put(0.38169061,0.0637306){\color[rgb]{0,0,0}\makebox(0,0)[lb]{\smash{\small$\binom{9}{14}$}}}%
    \put(0.73272944,0.04027029){\color[rgb]{0,0,0}\makebox(0,0)[lb]{\smash{1}}}%
    \put(0.78691422,0.10093441){\color[rgb]{0,0,0}\makebox(0,0)[lb]{\smash{-1}}}%
    \put(0.82915226,0.03335355){\color[rgb]{0,0,0}\makebox(0,0)[lb]{\smash{-2}}}%
    \put(0.71088575,0.11782963){\color[rgb]{0,0,0}\makebox(0,0)[lb]{\smash{-2}}}%
    \put(0.64330488,0.10093441){\color[rgb]{0,0,0}\makebox(0,0)[lb]{\smash{-1}}}%
    \put(0.61796206,0.05024877){\color[rgb]{0,0,0}\makebox(0,0)[lb]{\smash{-2}}}%
    \put(0.86294269,0.11782963){\color[rgb]{0,0,0}\makebox(0,0)[lb]{\smash{1}}}%
    \put(0.92284347,0.09590016){\color[rgb]{0,0,0}\makebox(0,0)[lb]{\smash{-1}}}%
    \put(0.96431398,0.05869637){\color[rgb]{0,0,0}\makebox(0,0)[lb]{\smash{1}}}%
    \put(0.91362834,0.16851528){\color[rgb]{0,0,0}\makebox(0,0)[lb]{\smash{-2}}}%
    \put(0.82488555,0.1879705){\color[rgb]{0,0,0}\makebox(0,0)[lb]{\smash{2}}}%
    \put(0.73622857,0.18541049){\color[rgb]{0,0,0}\makebox(0,0)[lb]{\smash{-1}}}%
    \put(0.64330488,0.16851528){\color[rgb]{0,0,0}\makebox(0,0)[lb]{\smash{1}}}%
    \put(0.59261923,0.15162006){\color[rgb]{0,0,0}\makebox(0,0)[lb]{\smash{2}}}%
    \put(0.6602001,0.23609614){\color[rgb]{0,0,0}\makebox(0,0)[lb]{\smash{-1}}}%
    \put(0.71933335,0.24454375){\color[rgb]{0,0,0}\makebox(0,0)[lb]{\smash{1}}}%
    \put(0.75988187,0.23609614){\color[rgb]{0,0,0}\makebox(0,0)[lb]{\smash{-2}}}%
    \put(0.81225704,0.24285423){\color[rgb]{0,0,0}\makebox(0,0)[lb]{\smash{1}}}%
    \put(0.86285689,0.2326828){\color[rgb]{0,0,0}\makebox(0,0)[lb]{\smash{1}}}%
    \put(0.90518073,0.23609614){\color[rgb]{0,0,0}\makebox(0,0)[lb]{\smash{-1}}}%
    \put(0.97276159,0.18541049){\color[rgb]{0,0,0}\makebox(0,0)[lb]{\smash{-1}}}%
  \end{picture}%
\endgroup%

%% file: matree.pdf_tex
%% Creator: Inkscape inkscape 0.91, www.inkscape.org
%% Accompanies image file 'matree.pdf' (pdf, eps, ps)
\begingroup%
  \makeatletter%
  \providecommand\color[2][]{%
    \errmessage{(Inkscape) Color is used for the text in Inkscape, but the package 'color.sty' is not loaded}%
    \renewcommand\color[2][]{}%
  }%
  \providecommand\transparent[1]{%
    \errmessage{(Inkscape) Transparency is used (non-zero) for the text in Inkscape, but the package 'transparent.sty' is not loaded}%
    \renewcommand\transparent[1]{}%
  }%
  \providecommand\rotatebox[2]{#2}%
  \ifx\svgwidth\undefined%
    \setlength{\unitlength}{183.84998769bp}%
    \ifx\svgscale\undefined%
      \relax%
    \else%
      \setlength{\unitlength}{\unitlength * \real{\svgscale}}%
    \fi%
  \else%
    \setlength{\unitlength}{\svgwidth}%
  \fi%
  \global\let\svgwidth\undefined%
  \global\let\svgscale\undefined%
  \makeatother%
  \begin{picture}(1,0.67607006)%
    \put(0,0){\includegraphics[width=\unitlength,page=1]{matree.pdf}}%
    \put(0.19613483,0.13881044){\color[rgb]{0,0,0}\makebox(0,0)[lb]{\smash{$\binom{1\,\, 1}{2\,\, 1}$}}}%
    \put(0.68349653,0.14502669){\color[rgb]{0,0,0}\makebox(0,0)[lb]{\smash{$\binom{2\,\,2}{2\,\,4}$}}}%
    \put(0.45772068,0.36259546){\color[rgb]{0,0,0}\makebox(0,0)[lb]{\smash{$\binom{3\,\,3}{4\,\,5}$}}}%
    \put(0.21089379,0.44962294){\color[rgb]{0,0,0}\makebox(0,0)[lb]{\smash{$\binom{4\,\,4}{5\,\,7}$}}}%
    \put(0.7125113,0.44029846){\color[rgb]{0,0,0}\makebox(0,0)[lb]{\smash{$\binom{5\,\,5}{8\,\,7}$}}}%
    \put(0.9314519,0.3843524){\color[rgb]{0,0,0}\makebox(0,0)[lb]{\smash{\tiny$\binom{7\,\,7}{10\,\,11}$}}}%
    \put(0.58669162,0.61373181){\color[rgb]{0,0,0}\makebox(0,0)[lb]{\smash{\tiny$\binom{8\,\,8}{13\,\,11}$}}}%
    \put(0.32288148,0.6193265){\color[rgb]{0,0,0}\makebox(0,0)[lb]{\smash{\tiny$\binom{7\,\,7}{10\,\,11}$}}}%
    \put(0.01541825,0.38746044){\color[rgb]{0,0,0}\makebox(0,0)[lb]{\smash{\tiny$\binom{5\,\,5}{7\,\,8}$}}}%
  \end{picture}%
\endgroup%

%% file: Figures8_9_10.pdf_tex
%% Creator: Inkscape inkscape 0.91, www.inkscape.org
%% Accompanies image file 'Figures8&9&10.pdf' (pdf, eps, ps)
\begingroup%
  \makeatletter%
  \providecommand\color[2][]{%
    \errmessage{(Inkscape) Color is used for the text in Inkscape, but the package 'color.sty' is not loaded}%
    \renewcommand\color[2][]{}%
  }%
  \providecommand\transparent[1]{%
    \errmessage{(Inkscape) Transparency is used (non-zero) for the text in Inkscape, but the package 'transparent.sty' is not loaded}%
    \renewcommand\transparent[1]{}%
  }%
  \providecommand\rotatebox[2]{#2}%
  \ifx\svgwidth\undefined%
    \setlength{\unitlength}{466.70948002bp}%
    \ifx\svgscale\undefined%
      \relax%
    \else%
      \setlength{\unitlength}{\unitlength * \real{\svgscale}}%
    \fi%
  \else%
    \setlength{\unitlength}{\svgwidth}%
  \fi%
  \global\let\svgwidth\undefined%
  \global\let\svgscale\undefined%
  \makeatother%
  \begin{picture}(1,0.28134658)%
    \put(0,0){\includegraphics[width=\unitlength,page=1]{Figures8_9_10.pdf}}%
    \put(0.13488866,0.04343666){\color[rgb]{0,0,0}\makebox(0,0)[lb]{\smash{\tiny$3$}}}%
    \put(0.23703081,0.04343656){\color[rgb]{0,0,0}\makebox(0,0)[lb]{\smash{\tiny$6$}}}%
    \put(0.18415474,0.10874415){\color[rgb]{0,0,0}\makebox(0,0)[lb]{\smash{\tiny$15$}}}%
    \put(0.2639922,0.12244381){\color[rgb]{0,0,0}\makebox(0,0)[lb]{\smash{\tiny$87$}}}%
    \put(0.32834953,0.10018288){\color[rgb]{0,0,0}\makebox(0,0)[lb]{\smash{\tiny$507$}}}%
    \put(0.22070816,0.19485398){\color[rgb]{0,0,0}\makebox(0,0)[lb]{\smash{\tiny$1299$}}}%
    \put(0.11038364,0.12542331){\color[rgb]{0,0,0}\makebox(0,0)[lb]{\smash{\tiny$39$}}}%
    \put(0.14649633,0.19205685){\color[rgb]{0,0,0}\makebox(0,0)[lb]{\smash{\tiny$582$}}}%
    \put(0.25536869,-0.05358989){\color[rgb]{0,0,0}\makebox(0,0)[lb]{\smash{}}}%
    \put(0.03436856,0.104967){\color[rgb]{0,0,0}\makebox(0,0)[lb]{\smash{\tiny$102$}}}%
    \put(0.01049319,0.05721624){\color[rgb]{0,0,0}\makebox(0,0)[lb]{\smash{\tiny$267$}}}%
    \put(0.03069543,0.17842964){\color[rgb]{0,0,0}\makebox(0,0)[lb]{\smash{\tiny$3975$}}}%
    \put(0.03424611,0.24821917){\color[rgb]{0,0,0}\makebox(0,0)[lb]{\smash{\tiny$154923$}}}%
    \put(0.11150433,0.25862637){\color[rgb]{0,0,0}\makebox(0,0)[lb]{\smash{\tiny$22683$}}}%
    \put(0.15913263,0.24270944){\color[rgb]{0,0,0}\makebox(0,0)[lb]{\smash{\tiny$8691$}}}%
    \put(0.36984798,0.05905286){\color[rgb]{0,0,0}\makebox(0,0)[lb]{\smash{\tiny$2955$}}}%
    \put(0.31536319,0.17475648){\color[rgb]{0,0,0}\makebox(0,0)[lb]{\smash{\tiny$44103$}}}%
    \put(0.37474551,0.1851637){\color[rgb]{0,0,0}\makebox(0,0)[lb]{\smash{\tiny$22360134$}}}%
    \put(0.31230225,0.2433216){\color[rgb]{0,0,0}\makebox(0,0)[lb]{\smash{\tiny$3836454$}}}%
    \put(0.21863737,0.2589937){\color[rgb]{0,0,0}\makebox(0,0)[lb]{\smash{\tiny$19398$}}}%
    \put(0.72859056,0.04191155){\color[rgb]{0,0,0}\makebox(0,0)[lb]{\smash{\tiny$1$}}}%
    \put(0.81490928,0.04313591){\color[rgb]{0,0,0}\makebox(0,0)[lb]{\smash{\tiny$2$}}}%
    \put(0.77756564,0.10313045){\color[rgb]{0,0,0}\makebox(0,0)[lb]{\smash{\tiny$3$}}}%
    \put(0.69981772,0.12945457){\color[rgb]{0,0,0}\makebox(0,0)[lb]{\smash{\tiny$4$}}}%
    \put(0.85347705,0.12516921){\color[rgb]{0,0,0}\makebox(0,0)[lb]{\smash{\tiny$5$}}}%
    \put(0.72859063,0.19373436){\color[rgb]{0,0,0}\makebox(0,0)[lb]{\smash{\tiny$7$}}}%
    \put(0.81613357,0.19006122){\color[rgb]{0,0,0}\makebox(0,0)[lb]{\smash{\tiny$8$}}}%
    \put(0.9134716,0.10129383){\color[rgb]{0,0,0}\makebox(0,0)[lb]{\smash{\tiny$7$}}}%
    \put(0.9563248,0.05782847){\color[rgb]{0,0,0}\makebox(0,0)[lb]{\smash{\tiny$9$}}}%
    \put(0.90612523,0.16618585){\color[rgb]{0,0,0}\makebox(0,0)[lb]{\smash{\tiny$12$}}}%
    \put(0.89449375,0.22495599){\color[rgb]{0,0,0}\makebox(0,0)[lb]{\smash{\tiny$17$}}}%
    \put(0.9636711,0.17904185){\color[rgb]{0,0,0}\makebox(0,0)[lb]{\smash{\tiny$19$}}}%
    \put(0.856538,0.23781192){\color[rgb]{0,0,0}\makebox(0,0)[lb]{\smash{\tiny$13$}}}%
    \put(0.80695074,0.2488313){\color[rgb]{0,0,0}\makebox(0,0)[lb]{\smash{\tiny$11$}}}%
    \put(0.75981222,0.2298535){\color[rgb]{0,0,0}\makebox(0,0)[lb]{\smash{\tiny$10$}}}%
    \put(0.70716399,0.24209721){\color[rgb]{0,0,0}\makebox(0,0)[lb]{\smash{\tiny$11$}}}%
    \put(0.6349257,0.09823289){\color[rgb]{0,0,0}\makebox(0,0)[lb]{\smash{\tiny$5$}}}%
    \put(0.61227471,0.05109438){\color[rgb]{0,0,0}\makebox(0,0)[lb]{\smash{\tiny$6$}}}%
    \put(0.62696724,0.17904183){\color[rgb]{0,0,0}\makebox(0,0)[lb]{\smash{\tiny$9$}}}%
    \put(0.64288416,0.24148505){\color[rgb]{0,0,0}\makebox(0,0)[lb]{\smash{\tiny$13$}}}%
  \end{picture}%
\endgroup%

%% file: quad1.pdf_tex
%% Creator: Inkscape inkscape 0.91, www.inkscape.org
%% Accompanies image file 'quad1.pdf' (pdf, eps, ps)
\begingroup%
  \makeatletter%
  \providecommand\color[2][]{%
    \errmessage{(Inkscape) Color is used for the text in Inkscape, but the package 'color.sty' is not loaded}%
    \renewcommand\color[2][]{}%
  }%
  \providecommand\transparent[1]{%
    \errmessage{(Inkscape) Transparency is used (non-zero) for the text in Inkscape, but the package 'transparent.sty' is not loaded}%
    \renewcommand\transparent[1]{}%
  }%
  \providecommand\rotatebox[2]{#2}%
  \ifx\svgwidth\undefined%
    \setlength{\unitlength}{468.94916603bp}%
    \ifx\svgscale\undefined%
      \relax%
    \else%
      \setlength{\unitlength}{\unitlength * \real{\svgscale}}%
    \fi%
  \else%
    \setlength{\unitlength}{\svgwidth}%
  \fi%
  \global\let\svgwidth\undefined%
  \global\let\svgscale\undefined%
  \makeatother%
  \begin{picture}(1,0.3592946)%
    \put(0,0){\includegraphics[width=\unitlength,page=1]{quad1.pdf}}%
    \put(0.08428751,0.32675457){\color[rgb]{0,0,0}\makebox(0,0)[lb]{\smash{$a_3$}}}%
    \put(0.30836518,0.21615711){\color[rgb]{0,0,0}\makebox(0,0)[lb]{\smash{$a_2$}}}%
    \put(0.40486767,0.01426913){\color[rgb]{0,0,0}\makebox(0,0)[lb]{\smash{$a_1$}}}%
    \put(0.263144,0.18522651){\color[rgb]{0,0,0}\makebox(0,0)[lb]{\smash{\tiny$(0,1,0)$}}}%
    \put(0.24893577,0.04866671){\color[rgb]{0,0,0}\makebox(0,0)[lb]{\smash{\tiny$(1,0,0)$}}}%
    \put(0.07979854,0.28646156){\color[rgb]{0,0,0}\makebox(0,0)[lb]{\smash{\tiny$(0,0,1)$}}}%
    \put(0.27452709,0.14198045){\color[rgb]{0,0,0}\makebox(0,0)[lb]{\smash{\tiny$(\frac{1}{2},\frac{1}{2},0)$}}}%
    \put(0.00604468,0.20987823){\color[rgb]{0,0,0}\makebox(0,0)[lb]{\smash{\tiny$(0,0,\frac{1}{2})$}}}%
    \put(0.34337314,0.05538878){\color[rgb]{0,0,0}\makebox(0,0)[lb]{\smash{\tiny$(\frac{3}{2},0,0)$}}}%
    \put(0.07256563,0.06512645){\color[rgb]{0,0,0}\makebox(0,0)[lb]{\smash{\tiny$(\frac{3}{4},0,\frac{1}{4})$}}}%
    \put(0.21904285,0.22935156){\color[rgb]{0,0,0}\makebox(0,0)[lb]{\smash{\tiny$(0,\frac{3}{4},0)$}}}%
    \put(0,0){\includegraphics[width=\unitlength,page=2]{quad1.pdf}}%
    \put(0.65827087,0.31877053){\color[rgb]{0,0,0}\makebox(0,0)[lb]{\smash{$a_2$}}}%
    \put(0.93122156,0.05434956){\color[rgb]{0,0,0}\makebox(0,0)[lb]{\smash{$a_1$}}}%
    \put(0.77403118,0.17924887){\color[rgb]{0,0,0}\makebox(0,0)[lb]{\smash{\tiny$(\frac{1}{2},\frac{1}{2})$}}}%
    \put(0.62780761,0.03180678){\color[rgb]{0,0,0}\makebox(0,0)[lb]{\smash{\tiny$(0,0)$}}}%
    \put(0.59173906,0.2515077){\color[rgb]{0,0,0}\makebox(0,0)[lb]{\smash{\tiny$(0,\frac{3}{4})$}}}%
    \put(0.83617629,0.03302528){\color[rgb]{0,0,0}\makebox(0,0)[lb]{\smash{\tiny$(\frac{3}{4},0)$}}}%

    \put(0.07825065,-0.03094753){\color[rgb]{0,0,0}\makebox(0,0)[lb]{\smash{}}}%

  \end{picture}%
\endgroup%

%% file: quad2.pdf_tex
%% Creator: Inkscape inkscape 0.91, www.inkscape.org
%% Accompanies image file 'quad2.pdf' (pdf, eps, ps)
\begingroup%
  \makeatletter%
  \providecommand\color[2][]{%
    \errmessage{(Inkscape) Color is used for the text in Inkscape, but the package 'color.sty' is not loaded}%
    \renewcommand\color[2][]{}%
  }%
  \providecommand\transparent[1]{%
    \errmessage{(Inkscape) Transparency is used (non-zero) for the text in Inkscape, but the package 'transparent.sty' is not loaded}%
    \renewcommand\transparent[1]{}%
  }%
  \providecommand\rotatebox[2]{#2}%
  \ifx\svgwidth\undefined%
    \setlength{\unitlength}{194.28827558bp}%
    \ifx\svgscale\undefined%
      \relax%
    \else%
      \setlength{\unitlength}{\unitlength * \real{\svgscale}}%
    \fi%
  \else%
    \setlength{\unitlength}{\svgwidth}%
  \fi%
  \global\let\svgwidth\undefined%
  \global\let\svgscale\undefined%
  \makeatother%
  \begin{picture}(1,0.77162886)%
    \put(0,0){\includegraphics[width=\unitlength,page=1]{quad2.pdf}}%
    \put(0.17517621,0.7179339){\color[rgb]{0,0,0}\makebox(0,0)[lb]{\smash{$a_2$}}}%
    \put(0.83399106,0.07970702){\color[rgb]{0,0,0}\makebox(0,0)[lb]{\smash{$a_1$}}}%
    \put(0.45458413,0.38117369){\color[rgb]{0,0,0}\makebox(0,0)[lb]{\smash{\tiny$(\frac{1}{2},\frac{1}{2})$}}}%
    \put(0.10164779,0.02529604){\color[rgb]{0,0,0}\makebox(0,0)[lb]{\smash{\tiny$(0,0)$}}}%
    \put(0.0145899,0.55558317){\color[rgb]{0,0,0}\makebox(0,0)[lb]{\smash{\tiny$(0,\frac{3}{4})$}}}%
    \put(0.60458236,0.02823721){\color[rgb]{0,0,0}\makebox(0,0)[lb]{\smash{\tiny$(\frac{3}{4},0)$}}}%

    \put(0.38105585,0.18264671){\color[rgb]{0,0,0}\makebox(0,0)[lb]{\smash{width}}}%

  \end{picture}%
\endgroup%

%% file: quiver.pdf_tex
%% Creator: Inkscape inkscape 0.91, www.inkscape.org
%% Accompanies image file 'quiver.pdf' (pdf, eps, ps)
\begingroup%
  \makeatletter%
  \providecommand\color[2][]{%
    \errmessage{(Inkscape) Color is used for the text in Inkscape, but the package 'color.sty' is not loaded}%
    \renewcommand\color[2][]{}%
  }%
  \providecommand\transparent[1]{%
    \errmessage{(Inkscape) Transparency is used (non-zero) for the text in Inkscape, but the package 'transparent.sty' is not loaded}%
    \renewcommand\transparent[1]{}%
  }%
  \providecommand\rotatebox[2]{#2}%
  \ifx\svgwidth\undefined%
    \setlength{\unitlength}{202.00704999bp}%
    \ifx\svgscale\undefined%
      \relax%
    \else%
      \setlength{\unitlength}{\unitlength * \real{\svgscale}}%
    \fi%
  \else%
    \setlength{\unitlength}{\svgwidth}%
  \fi%
  \global\let\svgwidth\undefined%
  \global\let\svgscale\undefined%
  \makeatother%
  \begin{picture}(1,1.01443106)%
    \put(0,0){\includegraphics[width=\unitlength,page=1]{quiver.pdf}}%
    \put(0.42575054,0.95877031){\color[rgb]{0,0,0}\makebox(0,0)[lb]{\smash{\tiny$1$}}}%
    \put(0.41554633,0.02559966){\color[rgb]{0,0,0}\makebox(0,0)[lb]{\smash{\tiny$2$}}}%
    \put(0.0255147,0.26141349){\color[rgb]{0,0,0}\makebox(0,0)[lb]{\smash{\tiny$3$}}}%
    \put(0.81503846,0.71474825){\color[rgb]{0,0,0}\makebox(0,0)[lb]{\smash{\tiny$4$}}}%
    \put(0.81162688,0.24926836){\color[rgb]{0,0,0}\makebox(0,0)[lb]{\smash{\tiny$5$}}}%
    \put(0.02761709,0.72136987){\color[rgb]{0,0,0}\makebox(0,0)[lb]{\smash{\tiny$6$}}}%
  \end{picture}%
\endgroup%